\documentclass[10pt]{article}

\usepackage{microtype}
\usepackage{graphicx}
\usepackage{subfigure}
\usepackage{booktabs}

\pdfoutput=1

\usepackage{amsmath,amsbsy,amsgen,amscd,amssymb,amsthm,amsfonts,stmaryrd}
\usepackage{mathtools}

\usepackage{bm}
\usepackage{enumerate}

\usepackage{microtype}      

\usepackage[colorlinks,citecolor=blue]{hyperref}
\usepackage{cleveref}
\usepackage{url}            

\usepackage[usenames,dvipsnames,svgnames]{xcolor}
\usepackage{graphicx}

\graphicspath{{art/}}

\definecolor{dark-gray}{gray}{0.3}
\definecolor{dkgray}{rgb}{.4,.4,.4}
\definecolor{dkblue}{rgb}{0,0,.5}
\definecolor{medblue}{rgb}{0,0,.75}
\definecolor{rust}{rgb}{0.7,0.1,0.1}
\definecolor{drust}{rgb}{0.5,0.1,0.1}

\hypersetup{urlcolor=drust}
\hypersetup{citecolor=blue}
\hypersetup{linkcolor=rust}

\newtheorem{theorem}{Theorem}
\newtheorem{lemma}[theorem]{Lemma}
\newtheorem{corollary}[theorem]{Corollary}
\newtheorem{fact}[theorem]{Fact}
\newtheorem{definition}{Definition}
\newtheorem{proposition}[theorem]{Proposition}

\numberwithin{equation}{section}
\numberwithin{theorem}{section}

\renewcommand{\phi}{\varphi}

\newcommand{\out}{\mathrm{out}}

\DeclareFontFamily{OT1}{pzc}{}
\DeclareFontShape{OT1}{pzc}{m}{it}{<-> s * [1.200] pzcmi7t}{}
\DeclareMathAlphabet{\mathpzc}{OT1}{pzc}{m}{it}

\DeclareMathOperator{\dist}{dist}

\DeclareMathOperator{\dom}{dom}

\DeclareMathOperator{\gra}{gra}
\newcommand{\id}{\mathrm{Id}}

\newcommand{\inprod}[2]{\left\langle{#1},{#2}\right\rangle}

\newtheorem{assumption}{Assumption \!\!}

\usepackage{xcolor,colortbl}

\makeatletter
\newtheorem*{rep@theorem}{\rep@title}
\newcommand{\newreptheorem}[2]{%
	\newenvironment{rep#1}[1]{%
		\def\rep@title{#2 \ref{##1}}%
		\begin{rep@theorem}}%
		{\end{rep@theorem}}}
\makeatother

\newreptheorem{theorem}{Theorem}
\newreptheorem{lemma}{Lemma}

\theoremstyle{definition}
\newtheorem{remark}[theorem]{Remark}
\newtheorem{example}[theorem]{Example}
\newtheorem{estimator}[theorem]{Estimator}

\usepackage{lscape}
\usepackage{makecell}

\usepackage{natbib}
\usepackage{algorithm}
\usepackage{algorithmic}
\usepackage[margin=1in]{geometry}

\crefname{lemma}{lemma}{lemmas}

\usepackage[colorinlistoftodos,prependcaption]{todonotes}

\title{Revisiting Inexact Fixed-Point Iterations for Min-Max Problems: \\Stochasticity and Structured Nonconvexity}
\author{Ahmet Alacaoglu\footnote{Department of Mathematics, University of British Columbia and Wisconsin Institute for Discovery, University of Wisconsin--Madison. \url{alacaoglu@math.ubc.ca}}\and Donghwan Kim\footnote{Department of Mathematical Sciences, KAIST. \url{donghwankim@kaist.ac.kr}} \and Stephen J. Wright\footnote{Department of Computer Sciences, University of Wisconsin--Madison. \url{swright@cs.wisc.edu}}}
\date{}
\begin{document}

\maketitle
\begin{abstract}
    We focus on constrained, $L$-smooth, potentially stochastic and nonconvex-nonconcave min-max problems either satisfying $\rho$-cohypomonotonicity or admitting a solution to the $\rho$-weakly Minty Variational Inequality (MVI), where larger values of the parameter $\rho>0$ correspond to a greater degree of nonconvexity. These problem classes include examples in two player reinforcement learning, interaction dominant min-max problems, and certain synthetic test problems on which classical min-max algorithms fail. It has been conjectured that first-order methods can tolerate a value of $\rho$ no larger than $\frac{1}{L}$, but  existing results in the literature have stagnated at the tighter requirement $\rho < \frac{1}{2L}$. With a simple argument, we obtain optimal or best-known complexity guarantees with cohypomonotonicity or weak MVI conditions for $\rho < \frac{1}{L}$. First main insight for the improvements in the convergence analyses is to harness the recently proposed \emph{conic nonexpansiveness} property of operators. Second, we provide a refined analysis for inexact Halpern iteration that relaxes the required inexactness level to improve some state-of-the-art complexity results even for  constrained stochastic convex-concave min-max problems. Third, we analyze a stochastic inexact Krasnosel'ski\u{\i}-Mann iteration with a multilevel Monte Carlo estimator when the assumptions only hold with respect to a solution.
\end{abstract}
\section{Introduction}\label{sec: intro}
We consider the problem
\begin{align}\label{eq: jut5}
    \min_{u\in U} \max_{v \in V} f(u, v),
\end{align}
where $U\subseteq\mathbb{R}^m, V\subseteq\mathbb{R}^n$ are closed convex sets admitting efficient projection operators and $f\colon \mathbb{R}^m\times\mathbb{R}^n\to\mathbb{R}$ is a function such that $\nabla_u f(u, v)$ and $\nabla_v f(u, v)$ are Lipschitz continuous. The general setting where $f(u, v)$ is allowed to be nonconvex-nonconcave is \emph{extremely} relevant in machine learning (ML), with  applications in generative adversarial networks (GANs) \citep{goodfellow2014generative} and adversarial ML \citep{madry2018towards}. Yet, at the same time, such problems are \emph{extremely} challenging to solve, with documented hardness results, see e.g., \cite{daskalakis2021complexity}. 
As a result, an extensive literature has arisen about special cases of the nonconvex-nonconcave problem \eqref{eq: jut5} for which algorithms with good convergence and complexity properties can be derived \citep{diakonikolas2021efficient,bauschke2021generalized,lee2021fast,pethick2022escaping,pethick2023solving,pethick2023stable,gorbunov2023convergence,bohm2022solving,cai2022accelerateda,cai2022acceleratedb,hajizadeh2023linear,kohlenbach2022proximal,anonymous2023semianchored,anonymous2023weaker,grimmer2023landscape,tran2023randomized}.

To describe these special cases of \eqref{eq: jut5},  we state the following \emph{nonmonotone} inclusion problem, which generalizes \eqref{eq: jut5}:
\begin{equation}\label{eq: sde4}
\text{Find~} x^\star\in\mathbb{R}^d \text{~such that~} 0 \in F(x^\star) + G(x^\star),
\end{equation}
where $F\colon \mathbb{R}^d \to \mathbb{R}^d$ is $L$-Lipschitz and $G\colon\mathbb{R}^d\rightrightarrows\mathbb{R}^d$ is maximally monotone. Mapping this problem to finding stationary points of \eqref{eq: jut5} is standard by setting $x=\binom{u}{v}$, $F(x) = \binom{\nabla_u f(u, v)}{-\nabla_v f(u,v)}$ and $G(x)=\binom{\partial \iota_U}{\partial \iota_V}$, where $\iota_U$ is the indicator function for set $U$. The nonmonotonicity in problem \eqref{eq: sde4} is due to nonconvex-nonconcavity of  problem \eqref{eq: jut5}.

The main additional assumption we make is that $F+G$ is \emph{$\rho$-cohypomonotone}. Recalling the standard definition $\gra (F+G)=\{(x, u) \in \mathbb{R}^d\times \mathbb{R}^d\colon~ u \in (F+G)(x)\}$, $\rho$-cohypomonotonicity is defined as
\begin{equation}\label{eq: asp_cohypo}
\begin{aligned}
\langle u - v, x-y \rangle \geq -\rho \| u-v\|^2~~~ \forall (x, u) \in \gra (F+G) \text{~and~} \forall (y, v) \in \gra (F+G),
\end{aligned}
\end{equation}
for $\rho > 0$, see \citep[Def. 2.4]{bauschke2021generalized}. When \eqref{eq: asp_cohypo} holds only for $y=x^\star$, it is also called the \emph{weak MVI condition} or \emph{$\rho$-star-cohypomonotonicity}, due to \citep{diakonikolas2021efficient}. For $\rho > 0$, the weak MVI condition requires the existence of a solution $x^\star$ to the $\rho$-weakly MVI:
\begin{equation}\label{eq: asp_mvi}
\langle u, x-x^\star \rangle \geq -\rho \| u\|^2 ~~~ \forall (x, u) \in \gra (F+G).
\end{equation}
For standard \emph{monotone operators} (corresponding to convex-concave instances of \eqref{eq: jut5}), the inner product  in \eqref{eq: asp_cohypo} is lower bounded by $0$. The assumption \eqref{eq: asp_cohypo} allows the right-hand side to be negative, allowing \emph{nonmonotonicity of $F+G$} or \emph{nonconvex-nonconcavity of $f(u, v)$}, while the limit of nonmonotonicity is determined by $\rho > 0$. These two assumptions, cohypomonotonicity or weak MVI, are required in the extensive literature cited above.

As the first contribution of this paper, we  extend the range of $\rho$, doubling the upper limit of $\frac{1}{2L}$ considered in the previous works, thus allowing a wider range of nonconvex problems of the form \eqref{eq: jut5} to be solved by first-order algorithms, while ensuring optimal or best-known complexity guarantees.

\textbf{Motivation. } Cohypomonotonicity and weak MVI conditions, defined in \eqref{eq: asp_cohypo} and \eqref{eq: asp_mvi}, allowed progress to be made in understanding the behavior of first-order algorithms for structured nonconvex-nonconcave problems, in a wide variety of works cited at the end of first paragraph.
On the one hand, these assumptions are not as general as one might desire: They have not been shown to hold for problems arising in generative or adversarial ML. 
On the other hand, they have been proven to hold for other relevant problems in ML. 

Examples where cohypomonotonicity holds include the \emph{interaction dominant min-max problems} (\Cref{ex: interaction}) and some stylized worst-case nonconvex-nonconcave instances \citep{hsieh2021limits,pethick2023stable} (see also \citep[Sections 5, 6]{bauschke2021generalized}).
The relaxed assumption of having a weak MVI solution is implied by star (and quasi-strong) monotonicity \citep{loizou2021stochastic} or existence of a solution to MVI \citep{dang2015convergence}, the latter being relevant in the context of policy gradient algorithms for reinforcement learning (RL) \citep{lan2023policy}. 
Weak MVI condition is satisfied in the context of an RL problem described in \Cref{ex: ratio_game}.
\begin{example}\label{ex: interaction}
\emph{Interaction dominant min-max problems \citep{grimmer2023landscape}: } 
We say that $f$ in \eqref{eq: jut5} is
$\alpha(r)$-interaction dominant if it satisfies for all $z=\binom{u}{v}\in\mathbb{R}^{n+m}$ that
\begin{align*}
    \nabla_{uu}^2 f(z) + \nabla^2_{uv} f(z)(r^{-1}\id - \nabla^2_{vv}f(z))^{-1}\nabla^2_{vu} f(z) &\succeq \alpha(r) \id,\\
    -\nabla_{vv}^2 f(z) + \nabla^2_{vu} f(z)(r^{-1}\id + \nabla^2_{uu}f(z))^{-1}\nabla^2_{uv} f(z) &\succeq \alpha(r) \id.
\end{align*}
\emph{Interaction} is captured by the second terms on the left-hand side of each condition.
The problem is called (nonnegative) \emph{interaction dominant} if these terms \emph{dominate} the smallest eigenvalue of $\nabla^2_{uu}f$ and largest eigenvalue of $\nabla_{vv}^2f$, i.e., $\alpha(r)\ge0$.
This is equivalent to the $r$-cohypomonotonicity of $F$ \citep[Proposition 1]{hajizadeh2023linear}.
$\hfill\blacklozenge$
\end{example}
\begin{example}\label{ex: ratio_game}
    \emph{Instances of von Neumann's ratio game:} This is a simple two player stochastic game \citep{neumann1945model,daskalakis2020independent,diakonikolas2021efficient}. Using the standard definition of the simplex $\Delta^d = \{ x\in\mathbb{R}^d\colon, x\geq 0, \sum_{i=1}^d x_i = 1 \}$, the problem is
    \begin{align*}
    \min_{x\in \Delta^m} \max_{y\in \Delta^n} \frac{\langle x, Ry \rangle}{\langle x, Sy\rangle},     
    \end{align*}
    where $R\in\mathbb{R}^{m\times n}$, $S\in\mathbb{R}_+^{m\times n}$ and $\langle x, Sy \rangle > 0 ~\forall(x,y)\in \Delta^m\times \Delta^n$.
    As described in   \cite{diakonikolas2021efficient}, it is easy to construct instances of this problem where it satisfies $\rho$-weakly MVI condition, but not cohypomonotonicity.$\hfill\blacklozenge$
\end{example}
The limit for the parameter $\rho$ in \eqref{eq: asp_cohypo} and \eqref{eq: asp_mvi} for which convergence first-order complexity results are proven seems to have stagnated at $\rho < \frac{1}{2L}$.
Two exceptions exist for a special case of our setting when $G\equiv 0$, which corresponds in view of \eqref{eq: jut5} to an unconstrained problem. First is the recent work \citep{anonymous2023weaker} that claimed to improve the limit of $\rho$ for weak MVI to $\approx \frac{0.63}{L}$ with a rather complicated analysis. The rate obtained is also suboptimal under cohypomonotonicity. This work conjectured (but did not prove) $\frac{1}{L}$ as the maximum limit for $\rho$ and also did not provide any algorithm achieving this. For an unconstrained cohypomonotone problem, \citep[Corollary~4.5]{cai2023variance} also showed possibility of obtaining guarantees with $\rho < \frac{1}{\sqrt{2}L}\approx\frac{0.7}{L}$.
Relevant citations and discussions appear in Table \ref{table:1} and \Cref{sec: app_relwork}.

\textbf{First-order oracles. } As standard in the operator splitting literature (see e.g., \cite{bauschke2017convex}), a \emph{first-order oracle call} for  \eqref{eq: sde4} consists of one evaluation of $F$ and one \emph{resolvent} of $G$ (see \eqref{eq: resolvent_def}). In the context of the min-max problem \eqref{eq: jut5}, this requires computation of gradients $\nabla_u f(u, v), \nabla_v f(u, v)$ together with projections on sets $U, V$. 
(All works in Table~\ref{table:1} have the same oracle access.) See \Cref{asp:4} for the oracles in the stochastic case.

\begin{table*}[t]
\centering
\begin{tabular}{ l|l|c c c } 
\hline
 Assumption & Reference  & \makecell{Upper bound \\ of $\rho$} &Constraints& \makecell{Oracle\\ complexity} \\
\hline
cohypomonotone &\cite{cai2022acceleratedb}&  $\frac{1}{60L}$ & $\checkmark$ &$O(\varepsilon^{-1})$ \\[2mm]\cline{2-2}
&\makecell[l]{ \cite{cai2022accelerateda}, \cite{pethick2023stable} \\
\cite{lee2021fast}, \cite{tran2023sublinear}\\\cite{gorbunov2023convergence} } &  $\frac{1}{2L}$ &$\checkmark$ & $O(\varepsilon^{-1})$ \\ [2mm]\cline{2-2}
& \cite{cai2023variance} &  $\frac{0.7}{L}$ & $\times$ &$\widetilde{O}(\varepsilon^{-1})$ \\ [2mm]
 & Theorem \ref{th: cohypo_det} & $\frac{1}{L}$ &$\checkmark$ &$\widetilde{O}(\varepsilon^{-1})$ \\[.5mm]
\hline
weak MVI &\cite{diakonikolas2021efficient}$^\ddagger$ &  $\frac{1}{8L}$ & $\times$ &$O(\varepsilon^{-2})$ \\[2mm] 
&\cite{bohm2022solving}$^\ddagger$ &  $\frac{1}{2L}$ & $\times$ &$O(\varepsilon^{-2})$ \\[2mm]
&\cite{cai2022acceleratedb}&  $\frac{1}{12\sqrt{3}L}$ &$\checkmark$& $O(\varepsilon^{-2})$ \\[2mm]
& \cite{anonymous2023semianchored}$^\ddagger$ &  $\frac{1}{3L}$ &$\checkmark$ & $\widetilde{O}(\varepsilon^{-2})$ \\ [2mm]
& \cite{pethick2022escaping} &  $\frac{1}{2L}$ &$\checkmark$ & $O(\varepsilon^{-2})$ \\ [2mm]
& \cite{anonymous2023weaker} &  $\frac{0.63}{L}$ & $\times$ &$O(\varepsilon^{-2})$ \\ [2mm]
 & Theorem \ref{th: weakmvi_det} & $\frac{1}{L}$ &$\checkmark$  &$\widetilde{O}(\varepsilon^{-2})$ \\ [.5mm]
\hline
\end{tabular}
\caption{\small Comparison of first-order algorithms for deterministic problems. Complexity refers to the number of oracle calls to get $\text{dist}(0, (F+G)(x))\leq\varepsilon$. See also Remark \ref{rem: opt_meas}. $^\ddagger$These works defined weak MVI as $\langle F(x), x-x^\star \rangle \geq -\frac{\gamma}{2}\|F(x)\|^2$, i.e., $\gamma=2\rho$.}
\label{table:1}
\end{table*}
 
\textbf{Contributions. } We show how to increase the range of the cohypomonotonicity parameter  to $\rho < \frac{1}{L}$ while maintaining first-order oracle complexity $\widetilde{O}(\varepsilon^{-1})$ for finding a point $x$ such that $\dist(0, (F+G)(x)) \leq \varepsilon$, in \Cref{sec: halpern}. 
Such a complexity is optimal (up to a log factor) even for monotone problems \citep[Section~3]{yoon2021accelerated}. 
In \Cref{sec: weak_mvi}, with weak MVI and the improved range of $\rho < \frac{1}{L}$, we show complexity $\widetilde{O}(\varepsilon^{-2})$ for $\dist(0, (F+G)(x)) \leq \varepsilon$ which is the best-known (up to a log factor) under this assumption. 
Table~\ref{table:1} summarizes known results on complexity and the upper bound of $\rho$. 

Thanks to the modularity of our approach, we extend our results to the stochastic case where $F$ is accessed via unbiased oracles $\widetilde{F}(\cdot)$ (that is, $\mathbb{E}[\widetilde{F}(x)] = F(x)$). 
These extensions require the development of further tools for stochastic min-max problems. First, in \Cref{sec:olc1}, we  tighten the analysis of Halpern iteration with inexact resolvent computations. This leads to improvements for the existing complexities even for some classes of convex-concave problems, see \Cref{sec: main_stoc_cohypo}. 
Second, to obtain the best-known complexity for stochastic problems under weak MVI, we incorporate the multilevel Monte Carlo estimator to KM iteration to control the bias in subproblem solutions, see \Cref{sec: main_stoc_wmi}.

\subsection{Preliminaries}\label{sec: prelim}
\paragraph{Notation. } We denote the $\ell_2$ norm as $\|\cdot\|$. Given $G\colon\mathbb{R}^d\rightrightarrows\mathbb{R}^d$, we use standard definitions $\gra G=\{(x, u) \in \mathbb{R}^d\times \mathbb{R}^d\colon u \in G(x)\}$ and $\dist(0, G(x)) = \min_{u \in G(x)}\|u\|$. Domain of an operator is defined as $\dom G = \{ x\in\mathbb{R}^d\colon G(x)\neq \emptyset \}$.
The operator $G$ is \emph{maximally} monotone (resp. cohypomonotone or hypomonotone) if its graph is not strictly contained in the graph of any other monotone (resp. cohypomonotone or hypomonotone) operator.

An operator $F\colon \mathbb{R}^d\rightrightarrows\mathbb{R}^d$, given $(x, u)\in\gra F$ and $(y, v)\in\gra F$, is \textbf{(i)} \emph{$\gamma$-strongly monotone} if $\langle u-v,x-y \rangle \geq \gamma \| x-y\|^2$ with $\gamma>0$ and \emph{monotone} if the inequality holds with $\gamma=0$; {\textbf{(ii)}} \emph{$\rho$-hypomonotone} if $\langle u-v,x-y \rangle \geq -\rho\|x-y\|^2$ with $\rho>0$. An operator  $F\colon\mathbb{R}^d\to\mathbb{R}^d$ is {\textbf{(iii)}} \emph{$L$-Lipschitz} if $\|F(x)-F(y)\|\leq L\|x-y\|$; {\textbf{(iv)}} \emph{nonexpansive} if $F$ is $1$-Lipschitz; \textbf{(v)} \emph{$\gamma$-cocoercive} if $\langle F(x)-F(y),x-y \rangle \geq \gamma \|F(x)-F(y)\|^2$ with $\gamma >0$. We refer to \emph{star} variants of these properties (e.g., \emph{star-cocoercive}) when they are required only at $(y,v)=(x^\star, 0)$ where $0 \in F(x^\star)$. Since it is a standard notion, we use quasi-nonexpansive instead of star-nonexpansive.

The {\em resolvent} of an operator $F\colon \mathbb{R}^d\rightrightarrows\mathbb{R}^d$ is defined as
\begin{equation}\label{eq: resolvent_def}
    J_{F} =(\id+F)^{-1}.
\end{equation}
The resolvent generalizes the well-known \emph{proximal operator} that has been ubiquitous in optimization and ML, where $F$ is typically the subdifferential of a regularizer function, e.g., $\ell_1$ norm.
Favorable properties of the resolvent are well-known when $F$ is monotone \citep{bauschke2017convex}. 
Meanwhile, in our nonmonotone case, immense care must be taken in utilizing this object, as it might even be undefined. 
A comprehensive reference for the properties of  resolvent of a nonmonotone operator is  \citep{bauschke2021generalized}. 
We review and explain the results relevant to our work in the sequel.

The algorithms we analyze are based on the classical Halpern \citep{halpern1967fixed} and Krasnosel'ski\u{\i}-Mann (KM) \citep{krasnosel1955two,mann1953mean} iterations. Given an operator $T\colon\mathbb{R}^d\to\mathbb{R}^d$, Halpern iteration is defined as
\begin{equation}\label{eq: halpern_iter}
    x_{k+1} = \beta_k x_0 + (1-\beta_k) T(x_k),
\end{equation}
for a decreasing sequence $\{ \beta_k\}\in (0, 1)$ and initial point $x_0$. 
The KM iteration, with a fixed $\beta\in(0, 1)$, is defined as
\begin{equation}\label{eq: km_iter}
    x_{k+1} = \beta x_k + (1-\beta) T(x_k).
\end{equation}
\textbf{Conic nonexpansiveness. } The key to relaxing the range of $\rho$ parameter for both assumptions is to harness the algorithmic consequences of \emph{conic nonexpansiveness}, the notion introduced by the influential work of \citet{bauschke2021generalized} that also inspired our developments. We say that $T\colon\mathbb{R}^d \to \mathbb{R}^d$ is $\lambda$-conically nonexpansive with $\lambda >0$ when there exists a nonexpansive operator $N\colon\mathbb{R}^d \to \mathbb{R}^d$ such that $T=(1-\lambda) \id + \lambda N$, see \citep[Def. 3.1]{bauschke2021generalized}. This equivalently means that a particular combination of $\id$ and $T$ is nonexpansive: $\|((1-\lambda^{-1})\id+\lambda^{-1}T)(x-y)\|\leq \|x-y\|$. An important characterization of this property given in \citep[Cor. 3.5(iii)]{bauschke2021generalized} is that $T$ is $\lambda$-conically nonexpansive if and only if $\text{Id}-T$ is $\frac{1}{2\lambda}$-cocoercive.
We also consider the \emph{star} variants (in the sense defined in the Notation paragraph) of these properties  and characterizations,
which are detailed in Appendix \ref{subsec: conic_star}.
\begin{assumption}\label{asp:1}
    The operator $F\colon\mathbb{R}^d\to \mathbb{R}^d$ is $L$-Lipschitz and $G\colon\mathbb{R}^d \rightrightarrows \mathbb{R}^d$ is maximally monotone. The solution set for the problem \eqref{eq: sde4} is nonempty.
\end{assumption}
Assumption \ref{asp:1} is standard, see \cite{facchinei2003finite}, and is required throughout the text. Monotonicity is not assumed for $F$. Lipschitzness of $F$ corresponds to smoothness of $f$ in context of \eqref{eq: jut5} and maximal monotonicity of $G$ is satisfied when we have constraint sets given in \eqref{eq: jut5} but also when we have convex regularizers added on \eqref{eq: jut5} (e.g., $\|\cdot\|_1$).

\begin{assumption}\label{asp:2}
    The operator $F+G$ is maximally $\rho$-cohypomonotone (see \eqref{eq: asp_cohypo} for the definition).
\end{assumption}
Assumption \ref{asp:2} is abundant in the recent literature for nonconvex-nonconcave optimization \citep{lee2021fast,bauschke2021generalized,cai2022accelerateda,cai2022acceleratedb,gorbunov2023convergence,pethick2023stable}. An instance is provided in Example \ref{ex: interaction} with further pointers to related problems given in \Cref{sec: intro}. Assumption \ref{asp:2} is required only for the results in Sections \ref{sec: halpern} and \ref{sec: main_stoc_cohypo}.
\begin{assumption}\label{asp:3}
    There exists a nonempty subset of solutions of \eqref{eq: sde4} whose elements satisfy \eqref{eq: asp_mvi}.
\end{assumption}
Assumption \ref{asp:3} is weaker than Assumption \ref{asp:2} as it is only required with respect to a solution, see also Example \ref{ex: ratio_game}. Assumption \ref{asp:3}, used in Sections \ref{sec: weak_mvi} and \ref{sec: main_stoc_wmi}, is also widespread in the recent literature for nonconvex-nonconcave optimization \citep{diakonikolas2021efficient,pethick2022escaping,pethick2023solving,cai2022accelerateda,anonymous2023semianchored,anonymous2023weaker,bohm2022solving}.

\section{Algorithm and Analysis under Cohypomonotonicity}\label{sec: halpern}
\begin{algorithm*}[h]
\caption{Inexact Halpern iteration for problems with cohypomonotonicity}
\begin{algorithmic}
    \STATE {\bfseries Input:} Parameters $\beta_k=\frac{1}{k+2}, \eta>0, L, \rho$, $\alpha=1-\frac{\rho}{\eta}$, $K\geq 1$, initial iterate $x_0\in\mathbb{R}^d$,
    subroutine $\texttt{FBF}$ in Algorithm \ref{alg:fbf} \\
    \vspace{.2cm}
    \FOR{$k = 0, 1, 2,\ldots, K-1$}
        \STATE $\widetilde{J}_{\eta(F+G)}(x_k) = \texttt{FBF}\left(x_k, N_k, \eta G, \id + \eta F,1+\eta L\right)$ where $N_k=\left\lceil\frac{4(1+\eta L)}{1-\eta L}\log(98\sqrt{k+2}\log(k+2))\right\rceil$
        \STATE $x_{k+1} = \beta_k x_0 + (1-\beta_k)((1-\alpha)x_k + \alpha\widetilde{J}_{\eta(F+G)}(x_k))$
        \ENDFOR
      \end{algorithmic}
\label{alg:cohypo}
\end{algorithm*}

\begin{algorithm}[h]
\caption{$\texttt{FBF}(z_0, N, A, B_{\mathrm{in}}, L_B)$ from \citep{tseng2000modified}}
\begin{algorithmic}
    \STATE {\bfseries Input:} Parameter $\tau=\frac{1}{2L_B}$, initial iterate $z_0\in\mathbb{R}^d$, $B(\cdot)=B_{\mathrm{in}}(\cdot)-z_0$ \\
    \vspace{.2cm}
    \FOR{$t = 0, 1, 2,\ldots, N-1 $}
        \STATE $z_{t+1/2} = J_{\tau A}(z_t - \tau B(z_t)) $ 
        \STATE $z_{t+1} = z_{t+1/2} +\tau B(z_t) - \tau B(z_{t+1/2})$
        \ENDFOR
      \end{algorithmic}
\label{alg:fbf}
\end{algorithm}
\subsection{Algorithm Construction and Analysis Ideas}\label{subsec: high_level_cohypo}

Recall the definitions of resolvent \eqref{eq: resolvent_def} and cohypomonotonicity \eqref{eq: asp_cohypo}.
We sketch the algorithmic construction and analysis ideas which will be expanded on in \Cref{sec:analysis2}.
\begin{enumerate}[(I)]
\item We know that Halpern iteration in \eqref{eq: halpern_iter} with $\beta_k = \frac{1}{k+2}$ has optimal rate when $T$ is nonexpansive, see \citep{sabach2017first,lieder2021convergence,kim2021accelerated}. That is, one gets $\|x_k - T(x_k)\|\leq\varepsilon$ with $O(\varepsilon^{-1})$ evaluations of $T$.\label{item: halpern_rec1}
\item When $F+G$ is maximally $\rho$-cohypomonotone (per Assumption \ref{asp:2}), we know from \cite{bauschke2021generalized} (with precise pointers in \Cref{fact: cohypo}) that $J_{\eta(F+G)}$ is $\frac{1}{2\alpha}$-conically nonexpansive where $\alpha = 1-\frac{\rho}{\eta}$, its domain is $\mathbb{R}^d$ and it is single-valued when $\frac{\rho}{\eta} < 1$.
Consequently, $T=(1-\alpha)\id+\alpha J_{\eta(F+G)}$ is \emph{firmly} nonexpansive (see \Cref{fact: cohypo}). Then, one can use the result in \eqref{item: halpern_rec1}.\label{item: halpern_rec2}

We next see a high level discussion on the approximate computation of $J_{\eta(F+G)}$.
\item Since $F$ is $L$-Lipschitz, we have that $F$ is $L$-hypomonotone by Cauchy-Schwarz inequality, i.e., 
\begin{equation*}
    \langle F(x) - F(y), x-y \rangle \geq -L \| x-y\|^2.
\end{equation*}
Hence, $\id + \eta F$ is $(1-\eta L)$-strongly monotone.

By definition, we have $x_k^\star = J_{\eta(F+G)}(x_k) = (\id +\eta(F+G))^{-1}(x_k)$. Existence and uniqueness of $x_k^\star$ is guaranteed by \eqref{item: halpern_rec2} when $\rho < \eta$ (see \Cref{fact: cohypo}). By definition, $x_k^\star$ is the solution of the problem
\begin{equation*}
    0 \in (\id+\eta(F+G))(x_k^\star) - x_k.
\end{equation*}
Hence, computation of the resolvent is a strongly monotone inclusion problem where $\id+\eta F$ is $(1-\eta L)$-strongly monotone and $(\eta L+1)$-Lipschitz, and $G$ is maximally monotone. In view of \eqref{eq: jut5} this also corresponds to a strongly convex-strongly concave problem.\label{item: halpern_rec3}

\item Any optimal algorithm for monotone inclusions, such as forward-backward-forward (FBF) \cite{tseng2000modified}, 
gives $\hat x_k$ with
$\|\hat x_k - J_{\eta(F+G)}(x_k)\|^2 \leq \varepsilon_k^2$ with complexity $\widetilde{O}\left(\frac{1+\eta L}{1-\eta L}\right)$.
\label{item: halpern_rec4}
\end{enumerate}
In summary, our requirements are $\frac{\rho}{\eta} < 1$ for ensuring well-definedness of the resolvent, as per \eqref{item: halpern_rec2}, and $1-\eta L > 0$ for ensuring strong monotonicity for efficient approximation of the resolvent, as per \eqref{item: halpern_rec3}. 
Hence, we need $\rho < \eta < \frac{1}{L}$, leading to the claimed improved range on $\rho$.

Item \eqref{item: halpern_rec2} refers to the resolvent of $\eta(F+G)$, which cannot be evaluated exactly in general with standard first-order oracles. 
We approximate $J_{\eta(F+G)}$, which leads to the inexact Halpern iteration, similar to \cite{diakonikolas2021efficient,cai2023variance}. 
Note that in the context of problem \eqref{eq: jut5}, approximating the resolvent corresponds to computing approximation of \emph{proximal operator} for function $f$ which is a strongly convex-strongly concave min-max problem.

In the next section, by extending the arguments in \citep[Lemma~12]{diakonikolas2020halpern} and \citep[Lemma~C.3]{cai2023variance} to accommodate conic nonexpansiveness, we show that $\eta^{-1}\|x_k - J_{\eta(F+G)}(x_k)\|\leq\varepsilon$, where the number of (outer) Halpern iterations is $O\left(\frac{\|x_0 - x^\star\|}{(\eta-\rho)\varepsilon}\right)$, when we approximate the resolvent to an accuracy of $\text{poly}\left(\frac{1}{k}\right)$.
To achieve this, we can run a subsolver as per \eqref{item: halpern_rec4}, with $\widetilde{O}\left(\frac{1+\eta L}{1-\eta L}\right)$ calls to evaluations of $F$ and resolvents of $G$.
By combining the complexities at outer and inner levels, we obtain the optimal first-order complexity under $\rho < \frac{1}{L}$.

\textbf{Discussion.} From the construction \eqref{item: halpern_rec1}-\eqref{item: halpern_rec4}, we see that the ingredients of our approach are based on known results. 
This raises the question: {\em what insight makes it possible to go beyond the $\rho < \frac{1}{2L}$ barrier?} 
The key is \emph{conic nonexpansiveness}, the critical notion introduced by \citet{bauschke2021generalized}. 
In particular, previous results on first-order complexity for nonmonotone problems
(including \cite{pethick2023stable} who utilized a similar algorithmic construction based on KM as ours in \Cref{sec: weak_mvi}) used \emph{nonexpansiveness} of the resolvent, which asks for the stringent requirement $\rho \leq \frac{\eta}{2}< \frac{1}{2L}$. This allows Halpern or KM iteration to be analyzed in a standard way. 

Our main starting insight is that, from the viewpoint of the analysis of Halpern iteration, we do not necessarily need nonexpansiveness of $J_{\eta(F+G)}$. We can apply the Halpern iteration to the operator $T=(1-\alpha)\id+\alpha J_{\eta(F+G)}$ where $\alpha = 1-\frac{\rho}{\eta}$, which is \emph{firmly nonexpansive} for $\rho < \eta$ (see \Cref{fact: cohypo}). 
Hence, as long as $\rho < \eta$,  Halpern iteration can be analyzed with $\rho < \eta < \frac{1}{L}$, at essentially no cost. We see later how \emph{firm nonexpansiveness} is essential for improving the inexactness criterion in approximating the resolvent.

\subsection{Analysis}  \label{sec:analysis2}

We now analyze the construction described in the previous section, given as Algorithm \ref{alg:cohypo}. We start with the main result, see \Cref{subsubsec: halpern_total} and \Cref{sec: total_compl_app} for its proof.
\begin{theorem}\label{th: cohypo_det}
    Let Assumptions \ref{asp:1} and \ref{asp:2} hold. Let $\eta < \frac{1}{L}$ in Algorithm \ref{alg:cohypo} and suppose $\rho < \eta$. 
    For any $k=1,\dotsc,K$, we have that $(x_k)$ from Algorithm \ref{alg:cohypo} satisfies
    \begin{equation*}
        \frac{1}{\eta^2}\|x_k - J_{\eta(F+G)}(x_k)\|^2 \leq \frac{16\|x_0-x^\star\|^2}{(\eta-\rho)^2 (k+1)^2}.
    \end{equation*}
    The number of first-order oracles used at iteration $k$ is upper bounded by $2N_k$ where $N_k$ is defined in \Cref{alg:cohypo}.
\end{theorem}
\begin{corollary}\label{cor: cohypo_det}
Under the setting of \Cref{th: cohypo_det}, for any $\varepsilon>0$, we have $\eta^{-1}\| (\id-J_{\eta(F+G)})(x_K)\| \leq \varepsilon$, for $K\leq \left\lceil \frac{4 \|x_0-x^*\|}{(\eta-\rho) \varepsilon} \right\rceil$ and
first-order oracle complexity
    \begin{equation*}
        \widetilde{O} \left( \frac{(1+\eta L)\|x_0 - x^\star\|}{\varepsilon  \left(\eta-\rho\right)(1-\eta L)} \right).
    \end{equation*}
\end{corollary}
\begin{remark}\label{rem: opt_meas}
    The definition of $x^\star$ gives that $(\id-J_{\eta(F+G)})(x^\star) = 0$ and $(\id-J_{\eta(F+G)})(x_k)$ is indeed the fixed point residual, which is a standard way to measure optimality for fixed point iterations, see e.g., \citep[Section 2.4.2]{ryu2022large}. Based on Cor. \ref{cor: cohypo_det}, it is straightforward to produce $x^{\out}$ with $\dist(0, (F+G)(x^{\out}))\leq\varepsilon$ as claimed in Table \ref{table:1}, with no change in the worst-case complexity. This is clear when $G\equiv 0$. In the general case, see \citep[Lemma C.4]{cai2023variance}. 
\end{remark}
\begin{remark}
    The constant in our complexity deteriorates as $\rho$ gets close to $\eta$ which is the same as most of the works included in Table \ref{table:1}. It is straightforward to make our  bound \emph{$\rho$-independent} in view of \cite{pethick2023stable} by simply
    expressing $\rho$
    as a fraction of $\eta$, e.g. assume $\rho < \frac{9\eta}{10}$. Then, at the expense of a constant multiple of $10$, we have the complexity
        $\widetilde{O}\left( \frac{(1+\eta L)\|x_0-x^\star\|}{\varepsilon\eta(1-\eta L)} \right),$ valid for the range $\rho < \frac{9}{10L}$.
    In comparison, the $\rho$-independent complexity result in 
    \cite{pethick2023stable}
    had $\widetilde{O}(\varepsilon^{-2})$ for $\rho < \frac{1}{2L}$.
    A similar reasoning by slightly restricting the range of $\rho$ can also make the algorithms agnostic to the knowledge of $\rho$.
\end{remark}
\textbf{Outline of the analysis. } We follow the steps sketched in \Cref{subsec: high_level_cohypo}. First, we analyze Halpern iteration with inexactness using the tools mentioned in \eqref{item: halpern_rec1}, \eqref{item: halpern_rec2}. Second, we analyze the inner loop (\Cref{alg:fbf}) as mentioned in \eqref{item: halpern_rec4}. Finally we piece together these ingredients.
\subsubsection{Outer-Loop Complexity} \label{sec:olc1}
We now analyze Halpern iteration with inexactness in the resolvent computation. 
See Appendix~\ref{sec: halpern_outer_app} for the proof.
\begin{lemma}\label{lem: cohypo_outer_mainlem}
Let Assumptions \ref{asp:1} and \ref{asp:2} hold. Suppose that the iterates $(x_k)$ of \Cref{alg:cohypo}  
satisfy $\| J_{\eta(F+G)}(x_i) - \widetilde{J}_{\eta(F+G)}(x_i)\| \le \varepsilon_i$ for some $\varepsilon_i>0$ and $\rho < \eta$. Let $R=\id-J_{\eta(F+G)}$.
    Then, we have for any $K\geq 1$ that
    \begin{align*}
        \frac{ K(K+1)}{4} \| R(x_K)\|^2 - \frac{K+1}{K\alpha^2} \| x^\star - x_0\|^2 \leq   \sum_{k=0}^{K-1} \left(\frac{(k+1)(k+2)\varepsilon_k^2}{2}+(k+1)\|R(x_k)\|\varepsilon_k\right),
    \end{align*}
    where $\|x_{k}-x^\star \| \leq \|x_0 - x^\star \| + \frac{\alpha}{k+1}\sum_{i=0}^{k-1} (i+1) \varepsilon_i$.
\end{lemma}
In \eqref{eq:ci1} below, we define appropriate values for $\varepsilon_k$, and show that the number of inner iterations $N_k$ selected for $\texttt{FBF}$ in \Cref{alg:cohypo} suffices to achieve the inexactness level $\varepsilon_k$.

This analysis extends \citet{diakonikolas2020halpern}, who studied monotone inclusions, in two aspects.
First, we analyze the convergence of the method under conic nonexpansiveness which is the relevant property when the parameter $\rho$ lies in the range $[\frac{1}{2L},\frac{1}{L})$. 
Second, and more importantly,  we conduct a tighter error analysis that allows the inexactness on the error in resolvent computation ($\varepsilon_k$) to be $\widetilde{O}(k^{-3/2})$ instead of the tolerance $\widetilde{O}(k^{-3})$ used in \citep{diakonikolas2020halpern,yoon2022accelerated,cai2023variance}. Even though it is not immediately obvious, this is because the bottleneck term on the bound in \Cref{lem: cohypo_outer_mainlem} is $\sum_{k=0}^{K-1}(k+1)(k+2)\varepsilon_k^2$ which sums to a $\log$ with $\varepsilon_k=\widetilde{O}(k^{-3/2})$.\footnote{A similar insight appeared in a different context in the independent work \cite{liang2024inexact}, which came out on arXiv at the same time as our paper.}
This tightening becomes important in the stochastic case in \Cref{sec: extensions}, where the inner loop does not have a linear rate of convergence.

The improvement derives from applying Halpern to the \emph{firmly nonexpansive} operator $(1-\alpha)\id+\alpha J_{\eta(F+G)}$, which helps avoid the main source of \emph{looseness} in the previous analysis which only uses nonexpansiveness. 
We discuss this further following \eqref{eq: byu4}.
See \Cref{rem: reflection} for a discussion from the viewpoint of nonexpansive operators.
\begin{remark}\label{rem: reflection}
By $\frac{1}{2\alpha}$-conic nonexpansiveness of $J_{\eta(F+G)}$ (see \Cref{fact: cohypo}(ii)), we have nonexpansiveness of $T'=(1-2\alpha) \id + 2\alpha J_{\eta(F+G)}$. If we were to apply Halpern iteration to this operator, we would still obtain results with $\rho < \frac{1}{L}$  but we would need a stricter inexactness requirement as the analyses in \cite{diakonikolas2020halpern,cai2023variance} dictate.

This can be viewed
as a Cayley (or reflection) operator of a firmly nonexpansive operator $T = (1-\alpha)\id + \alpha J_{\eta(F+G)}$, 
since $T' = 2T - \id$.
Our algorithm applies Halpern iteration to $T$ which helps us  relax the inexactness requirement.

On the other hand, as shown in \citep[Section 12.2]{ryu2022large}, while solving monotone inclusions with \emph{exact} evaluations of the resolvent, applying Halpern to the Cayley operator of the resolvent gives a better constant, by a factor of $4$. Our analysis brings to light a tradeoff between the constant in the convergence bound and the allowed inexactness in the computation of the resolvent.
\end{remark}
\subsubsection{Inner-Loop Complexity}\label{subsubsec: cohypo_inner}
The seminal FBF algorithm of \cite{tseng2000modified} is optimal for solving the resolvent subproblem, which is a strongly monotone inclusion. 
We provide the derivation of the precise constants appearing in the statement in \Cref{subsubsec: cohypo_inner}.
\begin{theorem}{(See \citep[Theorem 3.4]{tseng2000modified})}\label{th: fbf}
    Let $B$ be $\mu$-strongly monotone with $\mu > 0$ and $L_B$-Lipschitz; $A$ be maximally monotone, and $z^\star=(A+B)^{-1}(0)\neq \emptyset$. For any $\zeta>0$, running Algorithm~\ref{alg:fbf} with $\tau = \frac{1}{2L_B}$ and initial point $z_0$ for $N=\left\lceil \frac{4L_B}{\mu}\log\frac{\|z_0-z^\star\|}{\zeta} \right\rceil$ iterations give
    \begin{equation*}
        \|z_N - z^\star\| \leq\zeta,
    \end{equation*}
    where the number of calls to evaluations of $B$ and resolvents of $A$ is upper bounded by $2N$.
\end{theorem}
 
\subsubsection{Total complexity}\label{subsubsec: halpern_total}
\Cref{subsec: high_level_cohypo} already shows the key steps in our analysis, but we combine the preliminary results above into a  proof sketch here, to highlight the simplicity of our approach. Full proof is given in \Cref{sec: total_compl_app}.
\begin{proof}[Proof sketch of \Cref{th: cohypo_det}]
Denote $R=\id-J_{\eta(F+G)}$ for brevity.
Suppose that $\varepsilon_k$ in \Cref{lem: cohypo_outer_mainlem} satisfies 
\begin{equation} \label{eq:ci1}
    \varepsilon_k = \frac{\gamma \|R(x_k)\|}{ \sqrt{k+2}\log(k+2)}, \text{~with~} \gamma = \frac{1}{98}.
\end{equation}
We justify this supposition further below. Then we have by \Cref{lem: cohypo_outer_mainlem} (after multiplying both sides by $\alpha$) that
\begin{align*}
    \frac{\alpha K(K+1)}{4}\|R(x_K)\|^2 - \frac{K+1}{K\alpha} \| x_0 - x^\star\|^2 \leq \sum_{k=0}^{K-1} \|R(x_k)\|^2\left( \frac{\alpha \gamma^2(k+1)}{2 \log^2(k+2)} + \frac{\alpha \gamma \sqrt{k+2}}{\log (k+2)}\right).
\end{align*}
We can show by induction from this bound 
that 
\begin{equation*}
    \|R(x_k)\| \leq \frac{4\|x_0 - x^\star\|}{\alpha(k+1)}~~~\forall k \geq 1.
\end{equation*}
We see that for $K \leq \lceil \frac{4\|x_0-x^\star\|}{\eta\alpha\varepsilon} \rceil$, we are guaranteed to have $\eta^{-1}\|R(x_K) \| \leq \varepsilon$.

We now calculate the number of inner iterations to reach the accuracy $\varepsilon_k$ (see \eqref{eq:ci1}). 
At iteration $k$, as per the setup in Theorem \ref{th: fbf}, we set 
\begin{equation*}
\begin{aligned}
    &A\equiv\eta G, ~~~B(\cdot)\equiv(\id+\eta F)(\cdot)-x_k, ~~~z_0 \equiv x_k, \\
    &z_N\equiv \widetilde{J}_{\eta(F+G)}(x_k), ~~~ z^\star \equiv J_{\eta(F+G)}(x_k), ~~~ \zeta\equiv\varepsilon_k,
\end{aligned}
\end{equation*}
hence $z_0 - z^\star = (\id-J_{\eta(F+G)})(x_k) = R(x_k)$. $B$ is $L_B\equiv(1+\eta L)$-Lipschitz and $(1-\eta L)$-strongly monotone due to \Cref{fact: cohypo}(iv). Existence of $z^\star$ is guaranteed by \Cref{fact: cohypo}(i).

By matching these definitions with \Cref{alg:cohypo}, we see by invoking \Cref{th: fbf}  that the number of inner iterations used at step $k$ to obtain  $\|J_{\eta(F+G)}(x_k)-\widetilde{J}_{\eta(F+G)}(x_k)\|\leq\varepsilon_k$ is 
\begin{equation*}
N_k\equiv\left\lceil \frac{4(1+\eta L)}{1-\eta L}\log\frac{\|R(x_k)\|}{\varepsilon_k}\right\rceil,    
\end{equation*}
by the settings of $z_0$, $z^\star$, $R(x_k)$, and $\zeta$  above, along with $\varepsilon_k$ defined in \eqref{eq:ci1}. This value is precisely $N_k$ used in \Cref{alg:cohypo} (by the definition of $\varepsilon_k$), which justifies our application of \Cref{lem: cohypo_outer_mainlem} and the cost at iteration $k$.
\end{proof}

\section{Algorithm and Analysis under weak MVI}\label{sec: weak_mvi}
\subsection{Algorithm Construction and Analysis Ideas}\label{subsec: cohypo_ideas}
\begin{algorithm*}[h]
\caption{Inexact KM iteration for problems with weak MVI}
\begin{algorithmic}
    \STATE {\bfseries Input:} Parameters $\eta>0, L, \rho$, $\alpha_k=\alpha=1-\frac{\rho}{\eta}$, $K>0$, initial iterate $x_0\in\mathbb{R}^d$, subroutine $\texttt{FBF}$ in Algorithm \ref{alg:fbf} \\
    \vspace{.2cm}
    \FOR{$k = 0, 1, 2,\ldots, K-1 $}
        \STATE $\widetilde{J}_{\eta(F+G)}(x_k) = \texttt{FBF}\left(x_k, N_k, \eta G, \id + \eta F,1+\eta L\right)$, where $N_k=\left\lceil \frac{4(1+\eta L)}{1-\eta L}\log(8(k+1)\log^2(k+2))\right\rceil$ 
        \STATE $x_{k+1} = (1-\alpha_k)x_k + \alpha_k \widetilde{J}_{\eta(F+G)}(x_k)$
        \ENDFOR
      \end{algorithmic}
\label{alg:weakmvi}
\end{algorithm*}

We turn to the \emph{weak MVI condition} of Assumption \ref{asp:3}, which (as mentioned in \Cref{sec: prelim}) is weaker than cohypomonotonicity.
The best-known complexity under this assumption is $O(\varepsilon^{-2})$:
the lower part of \Cref{table:1} outlines existing results.
Our aim 
is to obtain $\widetilde{O}(\varepsilon^{-2})$ complexity for the extended range $\rho < \frac{1}{L}$.
The steps of our construction are as follows.
\begin{enumerate}[(i)]
    \item  KM iteration \eqref{eq: km_iter}, when $\id-T$ is star-cocoercive, gets $\eta^{-1}\|x_k - T(x_k)\|\leq\varepsilon$ with $O(\varepsilon^{-2})$ evaluations of $T$ \citep{groetsch1972note,browder1967construction}.\label{item: km_rec1}
    \item We get from \cite{bauschke2021generalized} 
    that $J_{\eta (F+G)}$ has domain $\mathbb{R}^d$ and is single-valued when $F$ is $L$-Lipschitz and $\eta < \frac{1}{L}$. 
    Lemma \ref{lem:conic-star} gives that $J_{\eta(F+G)}$ is $\frac{1}{2\alpha}$-conically quasi-nonexpansive, with $\alpha=1-\frac{\rho}{\eta}$, leading to $\id - J_{\eta(F+G)}$ being $\alpha$-star-cocoercive. 

    Thus, we require $\rho < \eta$. As per \eqref{item: km_rec1}, KM applied to $\id-J_{\eta(F+G)}$ requires 
    $O(\varepsilon^{-2})$ evaluations of $J_{\eta(F+G)}$ to find $x$ such that $\eta^{-1}\|x - J_{\eta(F+G)}(x)\|\leq\varepsilon$.\label{item: km_rec2}
    \item Since $F$ is Lipschitz and $G$ is maximally monotone, we can estimate $J_{\eta (F+G)}$ as before (via \eqref{item: halpern_rec3} and \eqref{item: halpern_rec4} of \Cref{sec: halpern}), with a linear rate of convergence when $\eta < \frac{1}{L}$. The existence of a solution to the subproblem is guaranteed by item \eqref{item: km_rec2}. The inner iterations introduce a logarithmic factor into the total complexity. As a result, the range for $\rho$ is again $\rho < \eta < \frac{1}{L}$. 
\end{enumerate}

Even with inexactness,  Alg. \ref{alg:weakmvi} is  classical; see \citep[Theorem 12.3.7]{facchinei2003finite}, \cite{combettes2001quasi} and \cite{combettes2002generalized}. 
We analyze this scheme for problems with weak MVI solutions and characterize the  first-order oracle complexity. \citet{pethick2023stable} recently analyzed a similar scheme under cohypomonotonicity, by using quasi-nonexpansiveness of the resulting operator.\footnote{This work claimed that some of their results extend to accommodate weak MVI condition as well.} Our main difference regarding the results in this section is that we harness the milder property of conic quasi-nonexpansiveness to improve the range of $\rho$ (see also \cite{bartz2022conical} for a similar idea by using exact resolvent).
We also approximate the resolvent slightly differently.
FBF can be replaced with other optimal algorithms like \citep{malitsky2020forward}, showing the modularity of our approach.

The key insight for extending the upper bound of $\rho$ to $\frac{1}{L}$ is similar to that of \Cref{sec: halpern}. The difference is that the analysis of Halpern iteration requires conic nonexpansiveness between any pair of points in the space, making it unsuitable with weak MVI. 
In contrast, the KM iteration can be analyzed with conic nonexpansiveness holding only with respect to a solution, a property that is a consequence of weak MVI.
Conic quasi-nonexpansiveness, while not defined explicitly in \cite{bauschke2021generalized}, directly follows by adapting the corresponding results therein by using $\rho$-weak MVI condition instead of cohypomonotonicity; see \Cref{subsec: conic_star} for the details.

\subsection{Analysis}
Similar to Section \ref{sec: halpern}, we start with the main complexity result, under weak MVI. 
Its proof appears in \Cref{subsec: weakmvi_total_app}.
\begin{theorem}\label{th: weakmvi_det}
Let Assumptions \ref{asp:1} and \ref{asp:3} hold. Let $\eta < \frac{1}{L}$ in Algorithm \ref{alg:weakmvi} and suppose $\rho < \eta$. For any $K\geq 1$, we have 
    \begin{equation*}
        \frac{1}{K}\sum_{k=0}^{K-1} \frac{1}{\eta^2}\|x_k - J_{\eta(F+G)}(x_k) \|^2 \leq \frac{11\|x_0-x^\star\|^2}{(\eta-\rho)^2 K}.
    \end{equation*}
    The number of first-order oracles used at iteration $k$ is upper bounded by $2N_k$ where $N_k$ is defined in \Cref{alg:weakmvi}.
\end{theorem}
\begin{corollary}\label{cor: weakmvi_det}
    Under the setting of \Cref{th: weakmvi_det}, for any $\varepsilon>0$, we have for some $x^{\out}\in\{x_0, \dots, x_{K-1}\}$ that $\eta^{-1}\| (\id-J_{\eta(F+G)})(x^{\out})\| \leq \varepsilon$ for $K\leq \left\lceil \frac{11 \|x_0-x^*\|}{(\eta-\rho)^2 \varepsilon^2} \right\rceil$ with first-order oracle complexity
    \begin{equation*}
        \widetilde{O}\left( \frac{(1+\eta L)\|x_0 - x^\star\|^2}{\varepsilon^2  \left(\eta-\rho\right)^2(1-\eta L)} \right).
    \end{equation*}
\end{corollary}
    See Remark \ref{rem: opt_meas} for details to convert this result to produce a point with $\dist(0, (F+G)(x^{\text{out}}))\leq\varepsilon$ as in Table \ref{table:1}.
    \begin{remark}\label{eq: weakmvi_rem1}
        This result is for the \emph{best iterate}, that is, $x^{\out}=\arg\min_{x\in\{x_0,\dots,x_{k-1}\}} \|(\id-J_{\eta(F+G)})(x)\|$, consistent with existing results for weak MVI, see \cite{diakonikolas2021efficient,pethick2022escaping,cai2022acceleratedb}.
    \end{remark}
    \begin{remark}\label{rem: xout}
        Note that $x^{\out}$ as defined in \Cref{eq: weakmvi_rem1} is not computable since we do not have access to $J_{\eta(F+G)}(x_k)$. For the unconstrained case, i.e., $G\equiv 0$, we can show the result with $x^{\out}  = \arg\min_{x\in\{x_0,\dots,x_{K-1}\}} \|Fx\|^2$, which is computable. For the constrained problem \eqref{eq: jut5}, we can handle this issue by slightly changing how $\widetilde{J}_{\eta(F+G)}$ is calculated and requiring the knowledge of the target accuracy $\varepsilon$, with no change in the order of complexity bounds. We present \Cref{alg:weakmvi} in its current form so that it is \emph{anytime}, not requiring the target accuracy as an input. 
        The details for making $x^{\out}$ computable are in \Cref{sec: weakmvi_remark_exp}. 
        We can also present this result as an {\em expected} bound for a \emph{randomly selected $x^{\out}$}, like \citep[Thm. 3.2(ii)]{diakonikolas2021efficient}. 
    \end{remark}
\textbf{Outer-loop complexity. } We analyze the iteration complexity of the outer loop; see \Cref{sec: weakmvi_outer} for a proof which is a modification of \cite{combettes2001quasi} and \cite{bartz2022conical} to accommodate conic quasi-nonexpansiveness and inexact resolvent computations.
\begin{lemma}\label{lem: rate_outer}
 Let Assumptions \ref{asp:1} and \ref{asp:3} hold. Suppose that the iterates $(x_k)$ of \Cref{alg:weakmvi} satisfy
$\| J_{\eta(F+G)}(x_k) - \widetilde{J}_{\eta(F+G)}(x_k)\| \le \varepsilon_k$ for some $\varepsilon_k >0$ and $\rho<\eta$.
Then, we have for $K\geq 1$ that
\begin{align*}
&\sum_{k=0}^{K-1}\|(\id - J_{\eta(F+G)})(x_k)\|^2 -\frac{2\eta^2}{(\eta-\rho)^2}\|x_0 - x^\star\|^2 \leq 6\sum_{k=0}^{K-1} \varepsilon_k^2+\frac{4\eta}{\eta-\rho} \sum_{k=0}^{K-1} \|x_k-x^\star\|\varepsilon_k,
\end{align*}
where $\|x_{k}-x^\star \| \leq \| x_{k-1} - x^\star \| + \alpha \varepsilon_{k-1}$.
\end{lemma}
\textbf{Total Complexity. }
The sketch of the proof of \Cref{th: weakmvi_det} follows \Cref{subsubsec: halpern_total} closely.
We use \Cref{lem: rate_outer} instead of \Cref{lem: cohypo_outer_mainlem}. The choice of $\varepsilon_k$ is slightly different, as can be noticed by the number of inner iterations $N_k$ in Algorithm \ref{alg:weakmvi}. However, with the same argument in \Cref{subsubsec: halpern_total}, we can show that this $N_k$  is sufficient to attain the inexactness required by $\varepsilon_k$.
\section{Algorithms and Analyses with Stochasticity}\label{sec: extensions}
In this case, $F$ in \eqref{eq: sde4} is accessed via unbiased oracles. 
\begin{assumption}\label{asp:4}
    The \emph{stochastic first-order oracle (SFO)} $\widetilde{F}\colon\mathbb{R}^d\to\mathbb{R}^d$ satisfies
    \begin{equation*}
        F(x) = \mathbb{E} [\widetilde{F}(x)] \text{~~and~~} \mathbb{E}\|\widetilde{F}(x) - F(x)\|^2\leq\sigma^2. 
    \end{equation*}
\end{assumption}
In view of \eqref{eq: jut5}, this corresponds to using \emph{stochastic gradients} $\widetilde{F}(x)=\binom{\widetilde{\nabla}_u f(u, v)}{-\widetilde{\nabla}_v f(u, v)}$ where $\mathbb{E}[\widetilde{\nabla}_u f(u, v)]=\nabla_u f(u,v)$ (and similarly for the $v$ component).
\Cref{table:2}, with comparisons for stochastic problems, is in \Cref{sec: extensions_app}. The variance assumption could be relaxed by using, e.g., an argument similar to \citep[Section 5.4.3]{wright2022optimization}.

\subsection{Cohypomonotone Case}\label{sec: main_stoc_cohypo}
For this setup, \Cref{alg:cohypo} will call $\texttt{FBF}$ with stochastic oracles $\widetilde{F}(x_t)$ as per \Cref{asp:4} to approximate $\widetilde{J}_{\eta(F+G)}$:
\begin{equation}\label{eq: stoc_cohypo_step}
    \widetilde{J}_{\eta(F+G)}(x_k) = \texttt{FBF}(x_k, N_k, \eta G, \id + \eta \widetilde{F},1+\eta L),
\end{equation}
where $N_k= \lceil 1734(k+2)^3\log^2(k+2)(1-\eta L)^{-2}\rceil$. 
\begin{corollary}\label{cor: cohypo_stoc_main}
Let Assumptions \ref{asp:1}, \ref{asp:2} and \ref{asp:4} hold. Let $\eta < \frac{1}{L}$ in Alg. \ref{alg:cohypo}, $\rho < \eta$ and use \eqref{eq: stoc_cohypo_step} for computing $\widetilde{J}_{\eta(F+G)}$ (see Alg. \ref{alg:cohypo_stoc_app}). Then we have for $k\geq 1$ that
\begin{equation*}
    \eta^{-2}\mathbb{E}\|x_k - J_{\eta(F+G)}(x_k)\|^2 = O\left( k^{-2}\right).
\end{equation*}
For any $\varepsilon>0$, we have $\eta^{-1}\mathbb{E}\| (\id-J_{\eta(F+G)})(x_K)\| \leq \varepsilon$ for the last iterate, with SFO complexity $\widetilde{O} (\varepsilon^{-4})$.
\end{corollary}
The proof, provided in \Cref{subsubsec: app_halpern_stoc} is the stochastic adaptation of \Cref{sec: halpern}. Our tighter analysis for the level of inexactness (which is highlighted after \Cref{lem: cohypo_outer_mainlem}) is the main reason we could get the $\widetilde{O}(\varepsilon^{-4})$ complexity. The inexactness level required by following the existing analyses in \cite{diakonikolas2020halpern,cai2023variance} would instead result in a $\widetilde{O}(\varepsilon^{-7})$ complexity.
\begin{remark}
    The previous \emph{last iterate} result for constrained, cohypomonotone, stochastic problems by \citep[Corollary E.3(ii)]{pethick2023stable} was $\widetilde{O}(\varepsilon^{-16})$ (in fact we are not aware of another last iterate result even for stochastic and constrained convex-concave problems). This result also required increasing batch sizes in the inner loop and $\rho < \frac{1}{2L}$. 
    For unconstrained problems, \citet{chen2022near} showed an improved $\widetilde{O}(\varepsilon^{-2})$ 
    expected complexity for $\rho < \frac{1}{2L}$ with some drawbacks described in \Cref{sec: app_relwork}. 
    It is an open question to get a similar complexity improvement in our constrained setup with a wider range for $\rho$. 
\end{remark}
\begin{remark}\label{rem: stoc_eps4}
    \citet{pethick2023solving} has complexity $\widetilde{O}(\varepsilon^{-4})$ for a constrained problem with weak MVI.
    However, this work additionally assumed a stronger oracle model and Lipschitzness assumptions. In particular, denoting $F(\cdot)=\mathbb{E}_{\xi\sim \Xi}[F_\xi(\cdot)]$ for an unknown $\Xi$ that we can sample from, this work assumes  mean-square (MS)-Lipschitzness: $\mathbb{E}_{\xi\sim\Xi}\|F_\xi(x) - F_\xi(y)\|^2 \leq L^2 \| x-y\|^2$. This work also needs to query the operator for the same seed for two different points: $F_\xi(x_k), F_\xi(x_{k-1})$. These two assumptions define a different template. For nonconvex minimization, for example, lower bounds improve with these assumptions compared to our standard \emph{stochastic approximation} setting in \Cref{asp:4}, see \cite{arjevani2023lower}. 
    Moreover, the additional assumption might not hold even for trivial problems: $F_1(x) = x^2$, $F_2(x) = -x^2$ where $F=F_1+F_2$
    is clearly Lipschitz but not MS-Lipschitz. 
\end{remark}

\subsection{Weak MVI Case}\label{sec: main_stoc_wmi}
We next modify \Cref{alg:weakmvi} for the stochastic case. The main observation from the analysis (see \Cref{lem: stoc_mvi_iter_compl}) is that bounding the \emph{bias} $\|\mathbb{E}[\widetilde{J}_{\eta(F+G)}(x_k)] - J_{\eta(F+G)}(x_k)\|$ with square root of \emph{variance} $\mathbb{E}\|\widetilde{J}_{\eta(F+G)}(x_k) - J_{\eta(F+G)}(x_k)\|^2$ by Jensen's inequality is loose and would give complexity $\widetilde{O}(\varepsilon^{-6})$, like \citep[Cor. E.3(i)]{pethick2023stable}. 

A natural candidate for a careful bias analysis is the multilevel Monte Carlo (MLMC) technique which helps control the bias-variance tradeoff \citep{giles2008multilevel,blanchet2015unbiased,asi2021stochastic,hu2021bias}. The high level idea is that stochastic KM iteration, in our setting would give $O(\varepsilon^{-4})$ complexity if we had unbiased samples of $J_{\eta(F+G)}$ (see, e.g., \cite{bravo2024stochastic}). Obtaining such unbiased samples is highly non-trivial since $J_{\eta(F+G)}$ is an inclusion problem in itself. Fortunately, MLMC is a way to get an estimator with bias $O(\varepsilon')$ and variance $\widetilde{O}(1)$ by making, in expectation,  $\widetilde{O}(1)$ calls to the oracle defined in \Cref{asp:4}. MLMC is used in \cite{asi2021stochastic} for the related proximal point algorithm.

\begin{estimator}[MLMC]\label{estimator}
We set $\widetilde{J}_{\eta(F+G)}$ as follows.
\begin{enumerate}
    \item
Given $N_k\geq 1$, $M_k\geq 1$, set for $m=1,\dots, M_k$,
\begin{equation*}
\begin{aligned}
    &\widetilde{J}^{(m)}_{\eta(F+G)}(x_k) = \begin{cases} y^0 + 2^I(y^I - y^{I-1}) \text{~if~} I \leq N_k, \\ y^0, \text{~~~otherwise,}\end{cases} \\ &\text{where~} 
    I\sim \mathrm{Geom}(1/2) \\
    &\text{and~} y^i = \texttt{FBF}(x_k, 2^i, G, \id+\eta \widetilde{F},  1+\eta L)~\forall i\geq 0.
\end{aligned}
\end{equation*}
\item Given $M_k$ independent draws of this estimator, we define $\widetilde{J}_{\eta(F+G)}(x_k) = \frac{1}{M_k}\sum_{m=1}^{M_k} \widetilde{J}^{(m)}_{\eta(F+G)}(x_k)$.
\end{enumerate}
\end{estimator}
To show that the scheme is \emph{implementable} we give the (non-optimized) values of $M_k, N_k$. This is to ensure that they are agnostic to unknown quantities $\{\|x_0-x^\star\|^2, \sigma^2\}$, unlike some MLMC methods \citep{chen2022near}. 
\begin{corollary}\label{cor: weakmvi_stoc_main}
Let Assumptions \ref{asp:1}, \ref{asp:3} and \ref{asp:4} hold. In Algorithm \ref{alg:weakmvi}, set $\eta < \frac{1}{L}$, $\alpha_k \equiv \frac{\alpha}{\sqrt{k+2}\log(k+3)}$, suppose that $\rho < \eta$ and use \Cref{estimator} for computing $\widetilde{J}_{\eta(F+G)}$ (see \Cref{alg:weakmvi_stoc}) with $N_k \equiv \lceil \frac{96(1-\eta L)^{-2}}{\min\{\frac{\alpha_k}{120\alpha(k+1)}, \frac{1}{120}\}} \rceil$ and $M_k \equiv\lceil \frac{672\times 120(\log_2 N_k)}{(1-\eta L)^2} \rceil$. For any $\varepsilon>0$, we have that $\eta^{-1}\mathbb{E}\| (\id-J_{\eta(F+G)})(x^{\out})\| \leq \varepsilon$, with \emph{expected} SFO complexity $\widetilde{O} (\varepsilon^{-4})$ where $x^\out$ is selected uniformly at random from $\{x_0, \dots, x_{K-1}\}$.
\end{corollary}
Proof of this corollary appears in \Cref{subsubsec: mlmc_main_app}.
This result is an alternative to \cite{pethick2023solving} that required additional assumptions as explained in \Cref{rem: stoc_eps4}.
In our setting under \Cref{asp:4}, the only $O(\varepsilon^{-4})$ complexity was known in the special case of unconstrained problems ($G\equiv 0$), due to \cite{diakonikolas2021efficient} (see also \cite{choudhury2023singlecall}). 
Because of the use of MLMC, our complexity result is \emph{expected} number of stochastic oracle calls and hence the results mentioned in this paragraph complement each other. See also \Cref{table:2}.

MLMC is used in conditional/compositional stochastic minimization \citep{hu2021bias}, distributionally robust optimization \citep{levy2020large}, and stochastic minimization with non-i.i.d. data \citep{dorfman2022adapting}. Our development of the KM iteration with MLMC can provide the potential to extend some of these results to stochastic min-max setting.

\section{Conclusions}
We conclude with some open questions. Even though our results for nonmonotone problems, either with cohypomonotonicity or weak MVI conditions, can go beyond the existing barrier for the $\rho$ parameter, our algorithms rely on a double loop strategy, alternating between a Halpern or KM step and resolvent approximation. This strategy also results in an additional $\log$ factor in the final complexity bound. Two worthwhile directions in this context are: \emph{(i)} developing a single loop algorithm with the extended $\rho$ range \emph{(ii)} obtaining first-order complexities without spurious $\log$ terms and extended range for $\rho$. It is also critical to find more problems in ML that satisfy these nonmonotonicity assumptions.

We next highlight a direction for stochastic problems. As mentioned before, the $O(\varepsilon^{-4})$ first-order complexity seems to be the best-known for even constrained, convex-concave stochastic problems. However, for unconstrained problems, better complexities are known, see, e.g., \cite{chen2022near,cai2022stochastic}. The tools in these works may be combined with the ideas in our paper to develop better complexity results for constrained stochastic min-max problems with convex-concave or nonconvex-nonconcave functions.

\section*{Acknowledgements}
This work was supported in part by the NSF grant 2023239, the NSF grant 2224213, the AFOSR award FA9550-21-1-0084, National Research Foundation of Korea (NRF) grant funded by the Korea government (MSIT) (No. 2019R1A5A1028324, 2022R1C1C1003940), and the Samsung Science \& Technology Foundation grant (No. SSTF-BA2101-02).

A. Alacaoglu is thankful to Vidya Muthukumar and Panayotis Mertikopoulos for helpful discussions about the implications of the results in \Cref{sec: main_stoc_cohypo}. We also thank a reviewer of our paper who helped us improve the presentation in \Cref{subsec: high_level_cohypo}. 

This work was done while A. Alacaoglu was at the University of Wisconsin--Madison.

\bibliography{nonmonot_arxiv.bib}

\begin{thebibliography}{69}
\providecommand{\natexlab}[1]{#1}
\providecommand{\url}[1]{\texttt{#1}}
\expandafter\ifx\csname urlstyle\endcsname\relax
  \providecommand{\doi}[1]{doi: #1}\else
  \providecommand{\doi}{doi: \begingroup \urlstyle{rm}\Url}\fi

\bibitem[Alacaoglu et~al.(2023)Alacaoglu, B{\"o}hm, and
  Malitsky]{alacaoglu2023beyond}
A.~Alacaoglu, A.~B{\"o}hm, and Y.~Malitsky.
\newblock Beyond the golden ratio for variational inequality algorithms.
\newblock \emph{Journal of Machine Learning Research}, 24\penalty0
  (172):\penalty0 1--33, 2023.

\bibitem[Allen-Zhu(2018)]{allen2018make}
Z.~Allen-Zhu.
\newblock How to make the gradients small stochastically: Even faster convex
  and nonconvex sgd.
\newblock \emph{Advances in Neural Information Processing Systems}, 31, 2018.

\bibitem[Arjevani et~al.(2023)Arjevani, Carmon, Duchi, Foster, Srebro, and
  Woodworth]{arjevani2023lower}
Y.~Arjevani, Y.~Carmon, J.~C. Duchi, D.~J. Foster, N.~Srebro, and B.~Woodworth.
\newblock Lower bounds for non-convex stochastic optimization.
\newblock \emph{Mathematical Programming}, 199\penalty0 (1-2):\penalty0
  165--214, 2023.

\bibitem[Asi et~al.(2021)Asi, Carmon, Jambulapati, Jin, and
  Sidford]{asi2021stochastic}
H.~Asi, Y.~Carmon, A.~Jambulapati, Y.~Jin, and A.~Sidford.
\newblock Stochastic bias-reduced gradient methods.
\newblock \emph{Advances in Neural Information Processing Systems},
  34:\penalty0 10810--10822, 2021.

\bibitem[Bartz et~al.(2022)Bartz, Dao, and Phan]{bartz2022conical}
S.~Bartz, M.~N. Dao, and H.~M. Phan.
\newblock Conical averagedness and convergence analysis of fixed point
  algorithms.
\newblock \emph{Journal of Global Optimization}, 82\penalty0 (2):\penalty0
  351--373, 2022.

\bibitem[Bauschke and Combettes(2017)]{bauschke2017convex}
H.~H. Bauschke and P.~L. Combettes.
\newblock Convex analysis and monotone operator theory in hilbert spaces.
\newblock \emph{CMS Books in Mathematics}, 2017.

\bibitem[Bauschke et~al.(2021)Bauschke, Moursi, and
  Wang]{bauschke2021generalized}
H.~H. Bauschke, W.~M. Moursi, and X.~Wang.
\newblock Generalized monotone operators and their averaged resolvents.
\newblock \emph{Mathematical Programming}, 189:\penalty0 55--74, 2021.

\bibitem[Blanchet and Glynn(2015)]{blanchet2015unbiased}
J.~H. Blanchet and P.~W. Glynn.
\newblock Unbiased monte carlo for optimization and functions of expectations
  via multi-level randomization.
\newblock In \emph{2015 Winter Simulation Conference (WSC)}, pages 3656--3667.
  IEEE, 2015.

\bibitem[B{\"o}hm(2022)]{bohm2022solving}
A.~B{\"o}hm.
\newblock Solving nonconvex-nonconcave min-max problems exhibiting weak minty
  solutions.
\newblock \emph{Transactions on Machine Learning Research}, 2022.

\bibitem[B{\"o}hm et~al.(2022)B{\"o}hm, Sedlmayer, Csetnek, and
  Bot]{bohm2022two}
A.~B{\"o}hm, M.~Sedlmayer, E.~R. Csetnek, and R.~I. Bot.
\newblock Two steps at a time---taking gan training in stride with tseng's
  method.
\newblock \emph{SIAM Journal on Mathematics of Data Science}, 4\penalty0
  (2):\penalty0 750--771, 2022.

\bibitem[Bravo and Cominetti(2024)]{bravo2024stochastic}
M.~Bravo and R.~Cominetti.
\newblock Stochastic fixed-point iterations for nonexpansive maps: Convergence
  and error bounds.
\newblock \emph{SIAM Journal on Control and Optimization}, 62\penalty0
  (1):\penalty0 191--219, 2024.

\bibitem[Bravo and Contreras(2024)]{bravo2024stochastic_halpern}
M.~Bravo and J.~P. Contreras.
\newblock Stochastic halpern iteration in normed spaces and applications to
  reinforcement learning.
\newblock \emph{arXiv:2403.12338}, 2024.

\bibitem[Browder and Petryshyn(1967)]{browder1967construction}
F.~E. Browder and W.~V. Petryshyn.
\newblock Construction of fixed points of nonlinear mappings in hilbert space.
\newblock \emph{Journal of Mathematical Analysis and Applications}, 20\penalty0
  (2):\penalty0 197--228, 1967.

\bibitem[Cai et~al.(2022{\natexlab{a}})Cai, Song, Guzm{\'a}n, and
  Diakonikolas]{cai2022stochastic}
X.~Cai, C.~Song, C.~Guzm{\'a}n, and J.~Diakonikolas.
\newblock Stochastic halpern iteration with variance reduction for stochastic
  monotone inclusions.
\newblock \emph{Advances in Neural Information Processing Systems},
  35:\penalty0 24766--24779, 2022{\natexlab{a}}.

\bibitem[Cai et~al.(2023)Cai, Alacaoglu, and Diakonikolas]{cai2023variance}
X.~Cai, A.~Alacaoglu, and J.~Diakonikolas.
\newblock Variance reduced halpern iteration for finite-sum monotone
  inclusions.
\newblock In \emph{International Conference on Learning Representations}, 2023.

\bibitem[Cai and Zheng(2022)]{cai2022acceleratedb}
Y.~Cai and W.~Zheng.
\newblock Accelerated single-call methods for constrained min-max optimization.
\newblock In \emph{The Eleventh International Conference on Learning
  Representations}, 2022.

\bibitem[Cai et~al.(2022{\natexlab{b}})Cai, Oikonomou, and
  Zheng]{cai2022accelerateda}
Y.~Cai, A.~Oikonomou, and W.~Zheng.
\newblock Accelerated algorithms for monotone inclusions and constrained
  nonconvex-nonconcave min-max optimization.
\newblock \emph{arXiv:2206.05248}, 2022{\natexlab{b}}.

\bibitem[Chen and Luo(2022)]{chen2022near}
L.~Chen and L.~Luo.
\newblock Near-optimal algorithms for making the gradient small in stochastic
  minimax optimization.
\newblock \emph{arXiv:2208.05925}, 2022.

\bibitem[Choudhury et~al.(2023)Choudhury, Gorbunov, and
  Loizou]{choudhury2023singlecall}
S.~Choudhury, E.~Gorbunov, and N.~Loizou.
\newblock Single-call stochastic extragradient methods for structured
  non-monotone variational inequalities: Improved analysis under weaker
  conditions.
\newblock In \emph{Advances in Neural Information Processing Systems}, 2023.

\bibitem[Combettes(2001)]{combettes2001quasi}
P.~L. Combettes.
\newblock Quasi-fej{\'e}rian analysis of some optimization algorithms.
\newblock In \emph{Studies in Computational Mathematics}, volume~8, pages
  115--152. Elsevier, 2001.

\bibitem[Combettes and Pennanen(2002)]{combettes2002generalized}
P.~L. Combettes and T.~Pennanen.
\newblock Generalized mann iterates for constructing fixed points in hilbert
  spaces.
\newblock \emph{Journal of Mathematical Analysis and Applications},
  275\penalty0 (2):\penalty0 521--536, 2002.

\bibitem[Combettes and Pennanen(2004)]{combettes2004proximal}
P.~L. Combettes and T.~Pennanen.
\newblock Proximal methods for cohypomonotone operators.
\newblock \emph{SIAM journal on control and optimization}, 43\penalty0
  (2):\penalty0 731--742, 2004.

\bibitem[Dang and Lan(2015)]{dang2015convergence}
C.~D. Dang and G.~Lan.
\newblock On the convergence properties of non-euclidean extragradient methods
  for variational inequalities with generalized monotone operators.
\newblock \emph{Computational Optimization and applications}, 60:\penalty0
  277--310, 2015.

\bibitem[Dao and Phan(2019)]{dao2019adaptive}
M.~N. Dao and H.~M. Phan.
\newblock Adaptive douglas--rachford splitting algorithm for the sum of two
  operators.
\newblock \emph{SIAM Journal on Optimization}, 29\penalty0 (4):\penalty0
  2697--2724, 2019.

\bibitem[Daskalakis et~al.(2020)Daskalakis, Foster, and
  Golowich]{daskalakis2020independent}
C.~Daskalakis, D.~J. Foster, and N.~Golowich.
\newblock Independent policy gradient methods for competitive reinforcement
  learning.
\newblock \emph{Advances in neural information processing systems},
  33:\penalty0 5527--5540, 2020.

\bibitem[Daskalakis et~al.(2021)Daskalakis, Skoulakis, and
  Zampetakis]{daskalakis2021complexity}
C.~Daskalakis, S.~Skoulakis, and M.~Zampetakis.
\newblock The complexity of constrained min-max optimization.
\newblock In \emph{Proceedings of the 53rd Annual ACM SIGACT Symposium on
  Theory of Computing}, pages 1466--1478, 2021.

\bibitem[Diakonikolas(2020)]{diakonikolas2020halpern}
J.~Diakonikolas.
\newblock Halpern iteration for near-optimal and parameter-free monotone
  inclusion and strong solutions to variational inequalities.
\newblock In \emph{Conference on Learning Theory}, pages 1428--1451. PMLR,
  2020.

\bibitem[Diakonikolas et~al.(2021)Diakonikolas, Daskalakis, and
  Jordan]{diakonikolas2021efficient}
J.~Diakonikolas, C.~Daskalakis, and M.~I. Jordan.
\newblock Efficient methods for structured nonconvex-nonconcave min-max
  optimization.
\newblock In \emph{International Conference on Artificial Intelligence and
  Statistics}, pages 2746--2754. PMLR, 2021.

\bibitem[Dorfman and Levy(2022)]{dorfman2022adapting}
R.~Dorfman and K.~Y. Levy.
\newblock Adapting to mixing time in stochastic optimization with markovian
  data.
\newblock In \emph{International Conference on Machine Learning}, pages
  5429--5446. PMLR, 2022.

\bibitem[Facchinei and Pang(2003)]{facchinei2003finite}
F.~Facchinei and J.-S. Pang.
\newblock \emph{Finite-dimensional variational inequalities and complementarity
  problems}.
\newblock Springer, 2003.

\bibitem[Fan et~al.(2024)Fan, Li, and Chen]{anonymous2023weaker}
Y.~Fan, Y.~Li, and B.~Chen.
\newblock Weaker {MVI} condition: Extragradient methods with multi-step
  exploration.
\newblock In \emph{The Twelfth International Conference on Learning
  Representations}, 2024.

\bibitem[Giles(2008)]{giles2008multilevel}
M.~B. Giles.
\newblock Multilevel monte carlo path simulation.
\newblock \emph{Operations research}, 56\penalty0 (3):\penalty0 607--617, 2008.

\bibitem[Giselsson and Moursi(2021)]{giselsson2021compositions}
P.~Giselsson and W.~M. Moursi.
\newblock On compositions of special cases of lipschitz continuous operators.
\newblock \emph{Fixed Point Theory and Algorithms for Sciences and
  Engineering}, 2021\penalty0 (1):\penalty0 1--38, 2021.

\bibitem[Goodfellow et~al.(2014)Goodfellow, Pouget-Abadie, Mirza, Xu,
  Warde-Farley, Ozair, Courville, and Bengio]{goodfellow2014generative}
I.~Goodfellow, J.~Pouget-Abadie, M.~Mirza, B.~Xu, D.~Warde-Farley, S.~Ozair,
  A.~Courville, and Y.~Bengio.
\newblock Generative adversarial nets.
\newblock \emph{Advances in neural information processing systems}, 27, 2014.

\bibitem[Gorbunov et~al.(2023)Gorbunov, Taylor, Horv{\'a}th, and
  Gidel]{gorbunov2023convergence}
E.~Gorbunov, A.~Taylor, S.~Horv{\'a}th, and G.~Gidel.
\newblock Convergence of proximal point and extragradient-based methods beyond
  monotonicity: the case of negative comonotonicity.
\newblock In \emph{International Conference on Machine Learning}, pages
  11614--11641. PMLR, 2023.

\bibitem[Grimmer et~al.(2023)Grimmer, Lu, Worah, and
  Mirrokni]{grimmer2023landscape}
B.~Grimmer, H.~Lu, P.~Worah, and V.~Mirrokni.
\newblock The landscape of the proximal point method for nonconvex--nonconcave
  minimax optimization.
\newblock \emph{Mathematical Programming}, 201\penalty0 (1-2):\penalty0
  373--407, 2023.

\bibitem[Groetsch(1972)]{groetsch1972note}
C.~Groetsch.
\newblock A note on segmenting mann iterates.
\newblock \emph{Journal of Mathematical Analysis and Applications}, 40\penalty0
  (2):\penalty0 369--372, 1972.

\bibitem[Hajizadeh et~al.(2023)Hajizadeh, Lu, and Grimmer]{hajizadeh2023linear}
S.~Hajizadeh, H.~Lu, and B.~Grimmer.
\newblock On the linear convergence of extragradient methods for
  nonconvex--nonconcave minimax problems.
\newblock \emph{INFORMS Journal on Optimization}, 2023.

\bibitem[Halpern(1967)]{halpern1967fixed}
B.~Halpern.
\newblock Fixed points of nonexpanding maps.
\newblock \emph{Bulletin of the American Mathematical Society}, 73\penalty0
  (6):\penalty0 957--961, 1967.

\bibitem[Hsieh et~al.(2019)Hsieh, Iutzeler, Malick, and
  Mertikopoulos]{hsieh2019convergence}
Y.-G. Hsieh, F.~Iutzeler, J.~Malick, and P.~Mertikopoulos.
\newblock On the convergence of single-call stochastic extra-gradient methods.
\newblock \emph{Advances in Neural Information Processing Systems}, 32, 2019.

\bibitem[Hsieh et~al.(2021)Hsieh, Mertikopoulos, and Cevher]{hsieh2021limits}
Y.-P. Hsieh, P.~Mertikopoulos, and V.~Cevher.
\newblock The limits of min-max optimization algorithms: Convergence to
  spurious non-critical sets.
\newblock In \emph{International Conference on Machine Learning}, pages
  4337--4348. PMLR, 2021.

\bibitem[Hu et~al.(2021)Hu, Chen, and He]{hu2021bias}
Y.~Hu, X.~Chen, and N.~He.
\newblock On the bias-variance-cost tradeoff of stochastic optimization.
\newblock \emph{Advances in Neural Information Processing Systems},
  34:\penalty0 22119--22131, 2021.

\bibitem[Kim(2021)]{kim2021accelerated}
D.~Kim.
\newblock Accelerated proximal point method for maximally monotone operators.
\newblock \emph{Mathematical Programming}, 190\penalty0 (1-2):\penalty0 57--87,
  2021.

\bibitem[Kohlenbach(2022)]{kohlenbach2022proximal}
U.~Kohlenbach.
\newblock On the proximal point algorithm and its halpern-type variant for
  generalized monotone operators in hilbert space.
\newblock \emph{Optimization Letters}, 16\penalty0 (2):\penalty0 611--621,
  2022.

\bibitem[Kotsalis et~al.(2022)Kotsalis, Lan, and Li]{kotsalis2022simple}
G.~Kotsalis, G.~Lan, and T.~Li.
\newblock Simple and optimal methods for stochastic variational inequalities,
  i: operator extrapolation.
\newblock \emph{SIAM Journal on Optimization}, 32\penalty0 (3):\penalty0
  2041--2073, 2022.

\bibitem[Krasnosel'skii(1955)]{krasnosel1955two}
M.~A. Krasnosel'skii.
\newblock Two remarks on the method of successive approximations.
\newblock \emph{Uspekhi matematicheskikh nauk}, 10\penalty0 (1):\penalty0
  123--127, 1955.

\bibitem[Lan(2023)]{lan2023policy}
G.~Lan.
\newblock Policy mirror descent for reinforcement learning: Linear convergence,
  new sampling complexity, and generalized problem classes.
\newblock \emph{Mathematical programming}, 198\penalty0 (1):\penalty0
  1059--1106, 2023.

\bibitem[Lee and Kim(2021)]{lee2021fast}
S.~Lee and D.~Kim.
\newblock Fast extra gradient methods for smooth structured
  nonconvex-nonconcave minimax problems.
\newblock \emph{Advances in Neural Information Processing Systems},
  34:\penalty0 22588--22600, 2021.

\bibitem[Lee and Kim(2024)]{anonymous2023semianchored}
S.~Lee and D.~Kim.
\newblock Semi-anchored gradient methods for nonconvex-nonconcave minimax
  problems, 2024.
\newblock URL \url{https://openreview.net/forum?id=rmLTwKGiSP}.

\bibitem[Leu{\c{s}}tean and Pinto(2021)]{leucstean2021quantitative}
L.~Leu{\c{s}}tean and P.~Pinto.
\newblock Quantitative results on a halpern-type proximal point algorithm.
\newblock \emph{Computational Optimization and Applications}, 79\penalty0
  (1):\penalty0 101--125, 2021.

\bibitem[Levy et~al.(2020)Levy, Carmon, Duchi, and Sidford]{levy2020large}
D.~Levy, Y.~Carmon, J.~C. Duchi, and A.~Sidford.
\newblock Large-scale methods for distributionally robust optimization.
\newblock \emph{Advances in Neural Information Processing Systems},
  33:\penalty0 8847--8860, 2020.

\bibitem[Liang et~al.(2024)Liang, Toh, and Zhu]{liang2024inexact}
L.~Liang, K.-C. Toh, and J.-J. Zhu.
\newblock An inexact halpern iteration for with application to distributionally
  robust optimization.
\newblock \emph{arXiv:2402.06033}, 2024.

\bibitem[Lieder(2021)]{lieder2021convergence}
F.~Lieder.
\newblock On the convergence rate of the halpern-iteration.
\newblock \emph{Optimization letters}, 15\penalty0 (2):\penalty0 405--418,
  2021.

\bibitem[Loizou et~al.(2021)Loizou, Berard, Gidel, Mitliagkas, and
  Lacoste-Julien]{loizou2021stochastic}
N.~Loizou, H.~Berard, G.~Gidel, I.~Mitliagkas, and S.~Lacoste-Julien.
\newblock Stochastic gradient descent-ascent and consensus optimization for
  smooth games: Convergence analysis under expected co-coercivity.
\newblock \emph{Advances in Neural Information Processing Systems},
  34:\penalty0 19095--19108, 2021.

\bibitem[Madry et~al.(2018)Madry, Makelov, Schmidt, Tsipras, and
  Vladu]{madry2018towards}
A.~Madry, A.~Makelov, L.~Schmidt, D.~Tsipras, and A.~Vladu.
\newblock Towards deep learning models resistant to adversarial attacks.
\newblock In \emph{International Conference on Learning Representations}, 2018.

\bibitem[Malitsky and Tam(2020)]{malitsky2020forward}
Y.~Malitsky and M.~K. Tam.
\newblock A forward-backward splitting method for monotone inclusions without
  cocoercivity.
\newblock \emph{SIAM Journal on Optimization}, 30\penalty0 (2):\penalty0
  1451--1472, 2020.

\bibitem[Mann(1953)]{mann1953mean}
W.~R. Mann.
\newblock Mean value methods in iteration.
\newblock \emph{Proceedings of the American Mathematical Society}, 4\penalty0
  (3):\penalty0 506--510, 1953.

\bibitem[Neumann(1945)]{neumann1945model}
J.~v. Neumann.
\newblock A model of general economic equilibrium.
\newblock \emph{The Review of Economic Studies}, 13\penalty0 (1):\penalty0
  1--9, 1945.

\bibitem[Pethick et~al.(2022)Pethick, Patrinos, Fercoq, Cevher, and
  Latafat]{pethick2022escaping}
T.~Pethick, P.~Patrinos, O.~Fercoq, V.~Cevher, and P.~Latafat.
\newblock Escaping limit cycles: Global convergence for constrained
  nonconvex-nonconcave minimax problems.
\newblock In \emph{International Conference on Learning Representations}, 2022.

\bibitem[Pethick et~al.(2023{\natexlab{a}})Pethick, Fercoq, Latafat, Patrinos,
  and Cevher]{pethick2023solving}
T.~Pethick, O.~Fercoq, P.~Latafat, P.~Patrinos, and V.~Cevher.
\newblock Solving stochastic weak minty variational inequalities without
  increasing batch size.
\newblock In \emph{International Conference on Learning Representations},
  2023{\natexlab{a}}.

\bibitem[Pethick et~al.(2023{\natexlab{b}})Pethick, Xie, and
  Cevher]{pethick2023stable}
T.~Pethick, W.~Xie, and V.~Cevher.
\newblock Stable nonconvex-nonconcave training via linear interpolation.
\newblock In \emph{Thirty-seventh Conference on Neural Information Processing
  Systems}, 2023{\natexlab{b}}.

\bibitem[Ryu and Yin(2022)]{ryu2022large}
E.~K. Ryu and W.~Yin.
\newblock \emph{Large-scale convex optimization: algorithms \& analyses via
  monotone operators}.
\newblock Cambridge University Press, 2022.

\bibitem[Sabach and Shtern(2017)]{sabach2017first}
S.~Sabach and S.~Shtern.
\newblock A first order method for solving convex bilevel optimization
  problems.
\newblock \emph{SIAM Journal on Optimization}, 27\penalty0 (2):\penalty0
  640--660, 2017.

\bibitem[Tran-Dinh(2023)]{tran2023sublinear}
Q.~Tran-Dinh.
\newblock Sublinear convergence rates of extragradient-type methods: A survey
  on classical and recent developments.
\newblock \emph{arXiv:2303.17192}, 2023.

\bibitem[Tran-Dinh and Luo(2023)]{tran2023randomized}
Q.~Tran-Dinh and Y.~Luo.
\newblock Randomized block-coordinate optimistic gradient algorithms for
  root-finding problems.
\newblock \emph{arXiv:2301.03113}, 2023.

\bibitem[Tseng(2000)]{tseng2000modified}
P.~Tseng.
\newblock A modified forward-backward splitting method for maximal monotone
  mappings.
\newblock \emph{SIAM Journal on Control and Optimization}, 38\penalty0
  (2):\penalty0 431--446, 2000.

\bibitem[Wright and Recht(2022)]{wright2022optimization}
S.~J. Wright and B.~Recht.
\newblock \emph{Optimization for data analysis}.
\newblock Cambridge University Press, 2022.

\bibitem[Yoon and Ryu(2021)]{yoon2021accelerated}
T.~Yoon and E.~K. Ryu.
\newblock Accelerated algorithms for smooth convex-concave minimax problems
  with o (1/k\^{} 2) rate on squared gradient norm.
\newblock In \emph{International Conference on Machine Learning}, pages
  12098--12109. PMLR, 2021.

\bibitem[Yoon and Ryu(2022)]{yoon2022accelerated}
T.~Yoon and E.~K. Ryu.
\newblock Accelerated minimax algorithms flock together.
\newblock \emph{arXiv:2205.11093}, 2022.

\end{thebibliography}
\bibliographystyle{abbrvnat}

\newpage
\appendix
\section{Proofs for Section \ref{sec: halpern}}\label{sec: halpern_app}

\subsection{Preliminary Results} \label{sec: halpern_app_prelim}
We start with the properties of the resolvent of a cohypomonotone operator and the properties of the subproblem for approximating this resolvent. These important points are also sketched in Section \ref{subsec: high_level_cohypo}. We present this preliminary result here for the ease of reference throughout the proofs. Most of the conclusions follow from the results of \cite{bauschke2021generalized}. Note that $\rho$-cohypomonotone in our notation is $-\rho$-comonotone in the notation of \cite{bauschke2021generalized}. See also \citep[Remark 2.5]{bauschke2021generalized} for these two conventions.
\begin{fact}\label{fact: cohypo}
    Let Assumptions \ref{asp:1} and \ref{asp:2} hold and let $\eta >0$. 
    Then, we have
    \begin{enumerate}[(i)]
        \item The operator $J_{\eta(F+G)}$ is single-valued and $\dom J_{\eta(F+G)}=\mathbb{R}^d$ when $\rho < \eta$. \label{fact1: item1}
        \item The operator $J_{\eta(F+G)}$ is $\frac{1}{2\left(1-\frac{\rho}{\eta}\right)}$-conically nonexpansive, $\id - J_{\eta(F_G)}$ is $\left(1-\frac{\rho}{\eta}\right)$-cocoercive, and $(1-\alpha)\id+\alpha J_{\eta(F+G)}$ is firmly nonexpansive when $\rho < \eta$.
        \item For any $\bar x \in \mathbb{R}^d$, computing $J_{\eta(F+G)}(\bar x)$ is equivalent to solving the problem:        \label{fact1: item3}
        \begin{equation}\label{eq: str_inc}
            \text{Find~} x\in\mathbb{R}^d \text{~such that~} 0 \in (\id+\eta(F+G))(x) - \bar x.
        \end{equation}
        The problem \eqref{eq: str_inc} has a unique solution when $\rho < \eta$.
        \item The operator $\id+\eta F$ is $(1+\eta L)$-Lipschitz and $(1-\eta L)$-strongly monotone when $\eta < \frac{1}{L}$.
    \end{enumerate}
\end{fact}
\begin{proof}
    \begin{enumerate}[(i)]
        \item By Assumption \ref{asp:2} and the definition of cohypomonotonicity in \eqref{eq: asp_cohypo}, we have that $\eta(F+G)$ is maximally $\frac{\rho}{\eta}$-cohypomonotone. Then for $\frac{\rho}{\eta} < 1$, \citep[Corollary 2.14]{bauschke2021generalized} gives the result.
        \item Since $\eta(F+G)$ is maximally $\frac{\rho}{\eta}$-cohypomonotone, \citep[Prop. 3.11(ii)]{bauschke2021generalized} gives $\frac{1}{2\left(1-\frac{\rho}{\eta}\right)}$-conic nonexpansiveness. Cocoercivity of $\id - J_{\eta(F+G)}$ then follows from \citep[Corollary 3.5(iii)]{bauschke2021generalized}.
        
By the definition of conic nonexpansiveness, we have that $T'=(1-2\alpha) \id + 2\alpha J_{\eta(F+G)}$ is nonexpansive. By definition, a firmly nonexpansive operator is one that can be written as $\frac{1}{2}\id + \frac{1}{2} N$ for a nonexpansive operator $N$ \citep[Remark 4.34(iii)]{bauschke2017convex}. Since $(1-\alpha)\id+\alpha J_{\eta(F+G)} = \frac{1}{2} \id+ \frac{1}{2} T'$ for the nonexpansive $T'$ specified in this paragraph,  we conclude.        
\item Let us denote $\bar x^\star = J_{\eta(F+G)}(\bar x)$ and use the definition of a resolvent to obtain
        \begin{equation*}
            \bar x^\star = J_{\eta(F+G)}(\bar x) = (\id + \eta(F+G))^{-1}(\bar x) \iff \bar x^\star + \eta(F+G)(\bar x^\star) \ni \bar x,
        \end{equation*}
        where the existence of $\bar x^\star$ is guaranteed by \eqref{fact1: item1}. Rearranging the inclusion gives \eqref{eq: str_inc}. Uniqueness of the solution is due to \eqref{fact1: item1}.
        \item By Lipschitzness of $F$ and Cauchy-Schwarz inequality, we have
        \begin{align*}
            \langle \eta F(x) - \eta F(y), x-y \rangle \geq - \eta \| F(x) - F(y) \| \|x-y\| \geq - \eta L\| x-y\|^2.
        \end{align*}
        As a result, we have that  $\id+ \eta F$ is $(1-\eta L)$-strongly monotone. 
        We also have by triangle inequality that
        \begin{equation*}
            \|(\id + \eta F)(x) - (\id+\eta F) y \| \leq \| x-y\| + \eta\|F(x) - F(y)\| \leq (1+\eta L) \| x-y\|,
        \end{equation*}
        completing the proof.\qedhere
    \end{enumerate}
\end{proof}

\subsection{Complexity of the Outer loop}\label{sec: halpern_outer_app}

\paragraph{Bounding the norm of the iterates. } 
\begin{lemma}\label{lem: wlk3}
     Let Assumptions \ref{asp:1} and \ref{asp:2} hold. Suppose that the iterates $(x_k)$ of \Cref{alg:cohypo} satisfy
$\| J_{\eta(F+G)}(x_k) - \widetilde{J}_{\eta(F+G)}(x_k)\| \le \varepsilon_k$ for some $\varepsilon_k >0$ and $\rho<\eta$. Then, we have for $k \geq 0$ that
    \begin{equation*}
        \|x_{k+1}-x^\star \| \leq \|x_0 - x^\star \| + \left(1-\frac{\rho}{\eta}\right)\frac{1}{k+2}\sum_{i=0}^k (i+1) \varepsilon_i.
    \end{equation*}
\end{lemma}
\begin{proof}
Recall the following notation from \Cref{alg:cohypo}:
\begin{equation*}
\alpha = 1-\frac{\rho}{\eta} = \frac{\eta-\rho}{\eta}.
\end{equation*}
Then, by \Cref{fact: cohypo}(ii), we know that $J_{\eta(F+G)}$ is $\frac{1}{2\alpha}$-conically nonexpansive.
This means that we can write $J_{\eta(F+G)}=(1-\frac{1}{2\alpha})\id+\frac{1}{2\alpha} N$ for a  nonexpansive operator $N$. 

Adding and subtracting $\alpha(1-\beta_k)J_{\eta(F+G)}(x_k)$ in the definition of $x_{k+1}$ in \Cref{alg:cohypo}, using conic nonexpansiveness of $J_{\eta(F+G)}$, and rearranging gives
    \begin{align*}
        x_{k+1} &= \beta_k x_0 + (1-\beta_k)\left( (1-\alpha)x_k + \alpha \widetilde{J}_{\eta(F+G)}(x_k) \right) \\
        &= \beta_k x_0 + (1-\beta_k)\left( (1-\alpha)x_k + \alpha J_{\eta(F+G)}(x_k) \right) + \alpha(1-\beta_k)\left( \widetilde{J}_{\eta(F+G)}(x_k) - J_{\eta(F+G)}(x_k) \right) \\
        &= \beta_k x_0 + \frac{1-\beta_k}{2}x_k+ \frac{1-\beta_k}{2}N(x_k) + \alpha(1-\beta_k)\left( \widetilde{J}_{\eta(F+G)}(x_k) - J_{\eta(F+G)}(x_k) \right),
    \end{align*}
    where the last step is because $J_{\eta(F+G)} = \frac{2\alpha-1}{2\alpha} \id + \frac{1}{2\alpha} N$ for a nonexpansive operator $N$.

    We now use triangle inequality, nonexpansiveness of $N$, the definition of $\varepsilon_k$, and the last equality to obtain
    \begin{align}
        \| x_{k+1}  -x^\star\| &\leq \beta_k \|x_0 - x^\star \| +\frac{1-\beta_k}{2} \|x_k-x^\star\| + \frac{1-\beta_k}{2} \| N(x_k) - x^\star\| \notag\\
        &\quad+ \alpha(1-\beta_k) \|\widetilde{J}_{\eta(F+G)}(x_k) - J_{\eta(F+G)}(x_k)\| \notag \\
        &\leq \beta_k \|x_0 - x^\star \| + (1-\beta_k) \| x_k - x^\star\| + \alpha(1-\beta_k) \varepsilon_k,\label{eq: hoi4}
    \end{align}
    where the inequality used that $Nx^\star = x^\star$ since $N=2\alpha J_{\eta(F+G)}+(1-2\alpha)\id$ and that $J_{\eta(F+G)}(x^\star) = x^\star$ by the definition of $x^\star$ in \eqref{eq: sde4}, and \Cref{fact: cohypo}(i).
    
    The result of the lemma now follows by induction after using the definition $\beta_k = \frac{1}{k+2}$. In particular, the assertion is true for $k=0$ by inspection. Assume the assertion holds for $k=K-1$, then \eqref{eq: hoi4} gives
    \begin{align*}
        \|x_{K+1} - x^\star\| &\leq \frac{1}{K+2} \|x_0 - x^\star\| + \frac{K+1}{K+2} \| x_K - x^\star \| + \frac{\alpha (K+1)}{K+2}\varepsilon_K \\
        &\leq \frac{1}{K+2} \|x_0 - x^\star\| + \frac{K+1}{K+2} \left( \| x_0 - x^\star\| + \frac{\alpha}{K+1}\sum_{i=0}^{K-1}  (i+1)\varepsilon_i \right) + \frac{\alpha (K+1)}{K+2}\varepsilon_K \\
        &= \| x_0 - x^\star\|  + \frac{\alpha}{K+2}\sum_{i=0}^K (i+1)\varepsilon_i,
    \end{align*}
    which completes the induction. The statement follows after using $\alpha = 1-\frac{\rho}{\eta}$.
\end{proof}

\paragraph{Iteration complexity}

\begin{replemma}{lem: cohypo_outer_mainlem}
Let Assumptions \ref{asp:1} and \ref{asp:2} hold. Suppose that the iterates $(x_k)$ of \Cref{alg:cohypo}  
satisfy $\| J_{\eta(F+G)}(x_i) - \widetilde{J}_{\eta(F+G)}(x_i)\| \le \varepsilon_i$ for some $\varepsilon_i>0$ and $\rho < \eta$. Let $R=\id-J_{\eta(F+G)}$.
    Then, we have for any $K\geq 1$ that
    \begin{align*}
        \frac{\alpha K(K+1)}{4} \| R(x_K)\|^2 -\frac{K+1}{K\alpha} \| x^\star - x_0\|^2 \leq  \sum_{k=0}^{K-1} \left(\frac{\alpha}{2}(k+1)(k+2)\varepsilon_k^2+\alpha(k+1)\|R(x_k)\|\varepsilon_k\right),
    \end{align*}
    where $\alpha=1-\frac{\rho}{\eta}$, as defined in \Cref{alg:cohypo} and $\|x_{k}-x^\star \| \leq \|x_0 - x^\star \| + \frac{\alpha}{k+1}\sum_{i=0}^{k-1} (i+1) \varepsilon_i$.
\end{replemma}
\begin{proof}[Proof of \Cref{lem: cohypo_outer_mainlem}]
By \Cref{fact: cohypo}(ii), we have
that $\id - J_{\eta(F+G)}$ is $\left(1-\frac{\rho}{\eta}\right)$ cocoercive. 
Recall the definition of $\alpha$ from \Cref{alg:cohypo} and the notation for $\id-J_{\eta(F+G)}$ as:
\begin{equation*}
 \alpha = 1-\frac{\rho}{\eta} \text{~~~~and~~~} R = \id - J_{\eta(F+G)}.
\end{equation*}
With these, we use $\alpha$-cocoercivity of $R$:
\begin{equation}\label{eq: oog4}
    \langle R(x_{k+1}) - R(x_k), x_{k+1} - x_k \rangle \geq \alpha \| R(x_{k+1}) - R(x_k)\|^2.
\end{equation}
By rearranging the update rule of $x_{k+1}$ in Algorithm \ref{alg:cohypo}, we have for $k\geq 0$ that
\begin{align}
    x_{k+1} &= \beta_k x_0 + (1-\beta_k) x_k - \alpha(1-\beta_k)( \id - \widetilde{J}_{\eta(F+G)} )(x_k)\notag \\
    &= \beta_k x_0 + (1-\beta_k) x_k - \alpha(1-\beta_k)R(x_k) + \alpha(1-\beta_k)( \widetilde{J}_{\eta(F+G)} - J_{\eta(F+G)} )(x_k),\label{eq: niu4}
\end{align}
where we added and subtracted $\alpha(1-\beta_k)J_{\eta(F+G)}(x_k)$ and used the definition $R=\id-J_{\eta(F+G)}$.

We now use a step that is common in the rate analysis of Halpern-type methods, which can be seen for example in \cite{diakonikolas2020halpern} or \cite{yoon2021accelerated}. In particular, from \eqref{eq: niu4}, we obtain two identical representations for $x_{k+1} - x_k$:
\begin{subequations}
\begin{align}
    x_{k+1} - x_k &= \beta_k(x_0 - x_k) - \alpha(1-\beta_k) R(x_k) + \alpha(1-\beta_k)( \widetilde{J}_{\eta(F+G)} - J_{\eta(F+G)} )(x_k),\label{eq: bbn1} \\
    x_{k+1} - x_k&= \frac{\beta_k}{1-\beta_k}(x_0 - x_{k+1}) -  \alpha R(x_k) + \alpha( \widetilde{J}_{\eta(F+G)} - J_{\eta(F+G)} )(x_k),\label{eq: bbn2}
\end{align}
\end{subequations}
where the second representation follows from subtracting $\beta_k x_{k+1}$ from both sides of \eqref{eq: niu4} and rearranging.
With these at hand, we develop the left-hand side of \eqref{eq: oog4}. First, by using \eqref{eq: bbn2}, we have that
\begin{align}
    \langle R(x_{k+1}), x_{k+1} - x_k \rangle &= \frac{\beta_k}{1-\beta_k}\langle R(x_{k+1}), x_0 - x_{k+1} \rangle - \alpha\langle R(x_{k+1}), R(x_k) \rangle \notag \\
    &\quad+ \alpha\langle R(x_{k+1}),  ( \widetilde{J}_{\eta(F+G)} - J_{\eta(F+G)} )(x_k)\rangle \notag \\
    &= \frac{\beta_k}{1-\beta_k}\langle R(x_{k+1}), x_0 - x_{k+1} \rangle - \frac{\alpha}{2} \left( \|R(x_{k+1})\|^2 + \|R(x_k)\|^2 - \|R(x_{k+1}) - R(x_k)\|^2 \right) \notag \\
    &\quad+ \alpha\langle R(x_{k+1}),  ( \widetilde{J}_{\eta(F+G)} - J_{\eta(F+G)} )(x_k)\rangle,\label{eq: oog5}
    \end{align}
    where the last step used the expansion $\|a-b\|^2 =\|a\|^2 -2\langle a, b \rangle + \|b\|^2$.
    
Second, by using \eqref{eq: bbn1}, we have that
\begin{align}
    -\langle R(x_{k}), x_{k+1} - x_k \rangle &= -\beta_k\langle R(x_{k}), x_0 - x_{k} \rangle + \alpha(1-\beta_k) \|R(x_k)\|^2 \notag\\
    &\quad -\alpha (1-\beta_k)\langle R(x_{k}),  ( \widetilde{J}_{\eta(F+G)} - J_{\eta(F+G)} )(x_k)\rangle.\label{eq: oog6}
\end{align}
After using \eqref{eq: oog5} and \eqref{eq: oog6} in \eqref{eq: oog4} and rearranging, we obtain
\begin{align}
    &\frac{\alpha}{2} \|R(x_{k+1})\|^2 + \frac{\beta_k}{1-\beta_k} \langle R(x_{k+1}), x_{k+1} - x_0 \rangle\notag \\
    &\leq \frac{\alpha}{2}\left( 1 - 2\beta_k\right)\|R(x_{k})\|^2 + \beta_k \langle R(x_{k}), x_{k} - x_0 \rangle\notag \\
    &\quad+\alpha \langle R(x_{k+1})-(1-\beta_k)R(x_k), ( \widetilde{J}_{\eta(F+G)} - J_{\eta(F+G)} )(x_k) \rangle - \frac{\alpha}{2} \| R(x_{k+1}) - R(x_k)\|^2.\label{eq: eer87}
\end{align}
For the third term on the right-hand side of \eqref{eq: eer87}, we apply Cauchy-Schwarz, triangle and Young's inequalities along with the definition of $\varepsilon_k$ to obtain
\begin{align}
    \alpha \langle R(x_{k+1}) - (1-\beta_k)R(x_k), (\widetilde{J}_{\eta(F+G)} - J_{\eta(F+G)})(x_k) \rangle 
    &\leq \alpha\|R(x_{k+1}) - (1-\beta_k)R(x_k)\|\varepsilon_k\notag \\
    &\leq \alpha\left(\|R(x_{k+1}) -R(x_k)\|+\beta_k\|R(x_k)\|\right)\varepsilon_k \notag\\
    &= \alpha\|R(x_{k+1}) - R(x_k)\|\varepsilon_k + \alpha\beta_k\|R(x_k)\|\varepsilon_k\notag \\
    &\leq \frac{\alpha}{2}\|R(x_{k+1}) - R(x_k)\|^2 + \frac{\alpha}{2}\varepsilon_k^2+\alpha\beta_k\|R(x_k)\|\varepsilon_k.\label{eq: byu4}
\end{align}
This is the main point of departure from the existing analysis where this inequality is bounded by $O\left( \|x_k- x^\star\|\varepsilon_k \right)$, cf. \citep[display equation after (14)]{diakonikolas2020halpern}. We instead use the last term in \eqref{eq: eer87} (which we obtained by using the firm nonexpansiveness of $(1-\alpha)\id+\alpha J_{\eta(F+G)}$) to cancel the corresponding error term in \eqref{eq: byu4}.
We use this last estimate in \eqref{eq: eer87} and get
\begin{align}
    \frac{\alpha}{2} \|R(x_{k+1})\|^2 + \frac{\beta_k}{1-\beta_k} \langle R(x_{k+1}), x_{k+1} - x_0 \rangle
    &\leq \frac{\alpha}{2}\left( 1 - 2\beta_k\right)\|R(x_{k})\|^2 + \beta_k \langle R(x_{k}), x_{k} - x_0 \rangle\notag \\
    &\quad+\frac{\alpha}{2}\varepsilon_k^2+\alpha\beta_k\|R(x_k)\|\varepsilon_k.\label{eq: oer88}
\end{align}
Noting the identities 
\[
\beta_k = \frac{1}{k+2} \implies 1-\beta_k = \frac{k+1}{k+2}, \quad 
\frac{\beta_k}{1-\beta_k} = \frac{1}{k+1}, \quad 1-2\beta_k = \frac{k}{k+2},
\]
on \eqref{eq: oer88} we obtain
\begin{align*}
    \frac{\alpha}{2}\|R(x_{k+1})\|^2 + \frac{1}{k+1} \langle R(x_{k+1}), x_{k+1} - x_0\rangle  &\leq \frac{\alpha}{2} \frac{k}{k+2} \| R(x_k)\|^2 + \frac{1}{k+2} \langle R(x_k), x_k-  x_0\rangle \\
    &\quad+ \frac{\alpha}{2}\varepsilon_k^2+\frac{\alpha}{k+2}\|R(x_k)\|\varepsilon_k,
\end{align*}
which holds for $k\geq 0$.
Multiplying both sides by $(k+1)(k+2)$ gives
\begin{align*}
    &\frac{\alpha(k+1)(k+2)}{2} \| R(x_{k+1})\|^2 + (k+2) \langle R(x_{k+1}), x_{k+1} - x_0 \rangle \\
    &\leq \frac{\alpha k (k+1)}{2} \|R(x_k)\|^2 + (k+1) \langle R(x_k), x_k - x_0 \rangle \\
    &\quad+\frac{\alpha}{2}(k+1)(k+2)\varepsilon_k^2+\alpha(k+1)\|R(x_k)\|\varepsilon_k.
\end{align*}
We sum the inequality for $k=0, 1, \dots, K-1$ to get
\begin{align}
    &\frac{\alpha K(K+1)}{2} \| R(x_{K})\|^2 + (K+1) \langle R(x_{K}), x_{K} - x_0 \rangle \notag\\
    &\leq \sum_{k=0}^{K-1} \left(\frac{\alpha}{2}(k+1)(k+2)\varepsilon_k^2+\alpha(k+1)\|R(x_k)\|\varepsilon_k\right).\label{eq: llr4}
\end{align}
By the standard estimation for the inner product on this left-hand side (using \emph{(i)} monotonicity of $R$, which is implied by $\alpha$-cocoercivity of $R$ with $\alpha > 0$; \emph{(ii)} definition of $x^\star$ as $R(x^\star) = (\id-J_{\eta(F+G)})(x^\star) =0$ which uses \Cref{fact: cohypo}(i);
\emph{(iii)} Young's inequality), we derive 
\begin{align*}
    (K+1)\langle R(x_K), x_K- x_0 \rangle &= (K+1)\langle R(x_K), x^\star- x_0 \rangle + (K+1)\langle R(x_K), x_K- x^\star \rangle \\
    &\geq (K+1)\langle R(x_K), x^\star- x_0 \rangle\\
    &\geq -\frac{\alpha K(K+1)}{4} \|R(x_K)\|^2 - \frac{K+1}{K\alpha} \|x^\star - x_0\|^2.
\end{align*}
We use this lower bound on \eqref{eq: llr4} to conclude the first assertion. The second claim  is essentially \Cref{lem: wlk3}.
\end{proof}
\subsection{Complexity of the Inner Loop}\label{subsubsec: cohypo_inner}
\begin{reptheorem}{th: fbf}{(See e.g., \citep[Theorem 3.4]{tseng2000modified})}
    Let $B$ be $\mu$-strongly monotone with $\mu > 0$ and $L_B$-Lipschitz; $A$ be maximally monotone, and $z^\star=(A+B)^{-1}(0)\neq \emptyset$. For any $\zeta>0$, running Algorithm~\ref{alg:fbf} with $\tau = \frac{1}{2L_B}$ and initial point $z_0$ for $N=\left\lceil \frac{4L_B}{\mu}\log\frac{\|z_0-z^\star\|}{\zeta} \right\rceil$ iterations give
    \begin{equation*}
        \|z_N - z^\star\| \leq\zeta,
    \end{equation*}
    where the number of calls to evaluations of $B$ and resolvents of $A$ is upper bounded by $2\left\lceil \frac{4L_B}{\mu}\log\frac{\|z_0-z^\star\|}{\zeta} \right\rceil$.
\end{reptheorem}
\begin{proof}
    We only derive the number of iterations for ease of reference which follows trivially from \citep[Theorem 3.4]{tseng2000modified}. In particular, in the notation of \citep[Theorem 3.4(c)]{tseng2000modified}, we select $\theta=\frac{1}{2}$, $\alpha = \frac{1}{2L_B}$ and assume without loss of generality that $\frac{\mu}{L_B}\leq\frac{3}{4}$ to obtain
    \begin{equation*}
        \|z_{t+1} - z^\star\|^2 \leq \left( 1-\frac{\mu}{2L_B} \right)\|z_t-z^\star\|^2,
    \end{equation*}
    which after unrolling gives that
    \begin{align*}
        \|z_N-z^\star\|^2 \leq \left(1-\frac{\mu}{2L_B}\right)^N\|z_0 - z^\star\|^2.
    \end{align*}
    Standard manipulations give that after $N=\left\lceil \frac{4L_B}{\mu}\log\frac{\|z_0-z^\star\|}{\zeta} \right\rceil$ iterations, we have  $\|z_N - z^\star \|^2\leq\zeta^2$.
\end{proof}
 
\subsection{Total complexity}\label{sec: total_compl_app}
\begin{reptheorem}{th: cohypo_det}
Let Assumptions \ref{asp:1} and \ref{asp:2} hold. Let $\eta < \frac{1}{L}$ in Algorithm \ref{alg:cohypo} and suppose $\rho < \eta$. 
    For any $k=1,\dotsc,K$, we have that $(x_k)$ from Algorithm \ref{alg:cohypo} satisfies
    \begin{equation*}
        \frac{1}{\eta^2}\|x_k - J_{\eta(F+G)}(x_k)\|^2 \leq \frac{16\|x_0-x^\star\|^2}{(\eta-\rho)^2 (k+1)^2}.
    \end{equation*}
    The number of first-order oracles used at iteration $k$ of Algorithm \ref{alg:cohypo} is upper-bounded by
    \begin{equation*}
        \left\lceil  \frac{4(1+\eta L)}{1-\eta L} \log(98\sqrt{k+2}\log(k+2))\right\rceil.
    \end{equation*}
\end{reptheorem}
\begin{proof}[Proof of Theorem \ref{th: cohypo_det}]
We recall the notations
\begin{equation*}
 \alpha = 1-\frac{\rho}{\eta} \text{~~~~and~~~} R = \id - J_{\eta(F+G)}
\end{equation*}
and start from the result of Lemma \ref{lem: cohypo_outer_mainlem} which states for $K\geq 1$ that 
\begin{align*}
        \frac{\alpha K(K+1)}{4} \| R(x_K)\|^2 \leq \frac{K+1}{K\alpha} \| x^\star - x_0\|^2 + \sum_{k=0}^{K-1} \left(\frac{\alpha}{2}(k+1)(k+2)\varepsilon_k^2+\alpha(k+1)\|R(x_k)\|\varepsilon_k\right).
    \end{align*}
Let us set
\begin{equation}\label{eq: fkl4}
    \varepsilon_k = \frac{\gamma \|R(x_k)\|}{ \sqrt{k+2}\log(k+2)}
\end{equation}
and note that we will not evaluate $\varepsilon_k$ but we will prove that for a \emph{computable} number of inner iterations $N_k$, this error criterion will be satisfied.

We substitute the definition of $\varepsilon_k$ to the previous inequality and get
\begin{align}\label{eq: ind_start1}
    \frac{\alpha K(K+1)}{4}\|R(x_K)\|^2 \leq \frac{K+1}{K\alpha} \| x_0 - x^\star\|^2 + \sum_{k=0}^{K-1} \left( \frac{\alpha \gamma^2(k+1)\|R(x_k)\|^2}{2 \log^2(k+2)} + \frac{\alpha \gamma \sqrt{k+2}\|R(x_k)\|^2}{\log (k+2)}\right).
\end{align}
We now show by induction that
\begin{equation}\label{eq: ind_result1}
    \|R(x_k)\| \leq \frac{4\|x_0 - x^\star\|}{\alpha(k+1)}~~~\forall k \geq 1.
\end{equation}
Note that $\alpha$-cocoercivity of $R$ and $R(x^\star) = 0$ gives $\|R(x_0)\|\leq \frac{1}{\alpha}\|x_0-x^\star\|$. For $k=1$, we have by $\alpha$-cocoercivity of $R$, $R(x^\star)=0$ and Lemma \ref{lem: wlk3} that
\begin{align}
    \|R(x_1)\| \leq \frac{1}{\alpha}\|x_1 - x^\star\| \leq \frac{1}{\alpha}\left( \|x_0 - x^\star\| + \frac{\gamma\|x_0 -x^\star\|}{2\sqrt{2} \log 2} \right) < \frac{2\|x_0 - x^\star\|}{\alpha},\label{eq: base_case_halpern}
\end{align}
for $\gamma =\frac{1}{98}$, which establishes the base case of induction. Now we assume \eqref{eq: ind_result1} holds for all $k \leq K-1$. Then, we use \eqref{eq: ind_start1} for $K\geq 2$ (where we also use $\frac{K+1}{K}\leq 2$):
\begin{align*}
    \frac{\alpha K(K+1)}{4}\|R(x_K)\|^2 &\leq \frac{2}{\alpha} \| x_0 - x^\star\|^2 + \sum_{k=0}^{K-1} \left( \frac{\alpha\gamma^2(k+1)\|R(x_k)\|^2}{2 \log^2(k+2)} + \frac{\alpha\gamma \sqrt{k+2}\|R(x_k)\|^2}{\log (k+2)}\right) \\
    &\leq \frac{2}{\alpha} \| x_0 - x^\star\|^2 + \sum_{k=0}^{K-1} \left(  \frac{8 \gamma^2\|x_0 - x^\star\|^2}{\alpha (k+1)\log^2(k+2)} + \frac{16\gamma\sqrt{k+2}\|x_0-x^\star\|^2}{\alpha(k+1)^2\log(k+2)} \right).
\end{align*}
Since we have that
\begin{equation*}
    \sum_{k=0}^{K-1}\frac{8}{(k+1)\log^2(k+2)} < 28 \text{~and~} \sum_{k=0}^{K-1} \frac{16\sqrt{k+2}}{(k+1)^2\log(k+2)} < 49,
\end{equation*}
the value $\gamma = \frac{1}{98}$ results in
\begin{equation*}
    \frac{\alpha K (K+1)}{4}\|R(x_K)\|^2 \leq \frac{2.6}{\alpha} \| x_0 - x^\star\|^2.
\end{equation*}
A direct implication of this inequality is that
\begin{align*}
    \|R(x_K)\|^2 &\leq \frac{10.4}{\alpha^2 K(K+1)} \| x_0 - x^\star\|^2 \\
    &\leq \frac{15.6}{\alpha^2 (K+1)^2} \| x_0 - x^\star\|^2,
\end{align*}
where we used $\frac{1}{K(K+1)} \leq \frac{1.5}{(K+1)^2}$ which holds when $K\geq 2$. This completes the induction.

We next see that with $N_k$ set as in \Cref{alg:cohypo}, we get the inexactness level specified by $\varepsilon_k$ and the oracle complexity of each iteration is as claimed in the statement.

At iteration $k$, to apply the result in Theorem \ref{th: fbf}, we identify the following settings  from \Cref{alg:cohypo} 
\begin{align*}
    &A\equiv\eta G, ~~~B(\cdot)\equiv(\id+\eta F)(\cdot)-x_k, ~~~z_0 \equiv x_k,~~~z^\star \equiv J_{\eta(F+G)}(x_k), ~~~\zeta\equiv\varepsilon_k\\
    \Longrightarrow~&z_0 - z^\star = (\id-J_{\eta(F+G)})(x_k) = R(x_k)
\end{align*}
hence $B$ is $(1+\eta L)$-Lipschitz and $(1-\eta L)$-strongly monotone due to \Cref{fact: cohypo}(iv). Existence of $z^\star$ is guaranteed by \Cref{fact: cohypo}(iii).

We now see that by the setting of
\begin{equation*}
T=\left\lceil \frac{4(1+\eta L)}{1-\eta L}\log(98\sqrt{k+2}\log(k+2))\right\rceil = \left\lceil \frac{4(1+\eta L)}{1-\eta L}\log\frac{\|R(x_k)\|}{\varepsilon_k}\right\rceil,
\end{equation*}
\Cref{th: fbf} tells us that
\begin{equation*}
    \| \widetilde{J}_{\eta(F+G)}(x_k) - J_{\eta(F+G)}(x_k)\| \leq\varepsilon_k,
\end{equation*}
for $\varepsilon_k$ given in \eqref{eq: fkl4}, as claimed.

Since each iteration of \Cref{alg:fbf} uses $2$ evaluations of $F$ and $1$ resolvent for $G$, the first-order oracle complexity is $2N_k$ and the result follows.
\end{proof}
We now continue with the proof of \Cref{cor: cohypo_det} which follows trivially from \Cref{th: cohypo_det}.
\begin{proof}[Proof of \Cref{cor: cohypo_det}]
    By \Cref{th: cohypo_det}, we have that after at most $\left\lceil \frac{4\|x_0-x^\star\|}{(\eta-\rho)\varepsilon}\right\rceil$ iterations, i.e., for a $K$ such that
    \begin{equation}\label{eq: njh5}
        K\leq \left\lceil \frac{4\|x_0-x^\star\|}{(\eta-\rho)\varepsilon}\right\rceil,
    \end{equation}
    we are guaranteed to have
\begin{equation*}
    \eta^{-1}\|(\id-J_{\eta(F+G)})(x_K) \| \leq \varepsilon.
\end{equation*}
Total number of first-oracle calls during the run of the algorithm then be calculated as
\begin{equation*}
    \sum_{k=1}^K \left\lceil \frac{4(1+\eta L)}{1-\eta L}\log(98\sqrt{k+2}\log(k+2))\right\rceil \leq K\cdot \left(\frac{4(1+\eta L)}{1-\eta L}\log(98\sqrt{K+2}\log(K+2))+1\right).
\end{equation*}
We conclude after using \eqref{eq: njh5}.
\end{proof}
\section{Proofs for Section \ref{sec: weak_mvi}}\label{sec: weak_mvi_app}
\subsection{Preliminary results}
We now derive similar properties to \Cref{fact: cohypo} but with Assumption \ref{asp:3}. These proofs are slightly more involved than \Cref{fact: cohypo} to accommodate the weaker assumption. 

We start with the definition of conic quasi-nonexpansiveness that will be used in \Cref{fact: weakmvi}. Recall that an operator $N$ is quasi-nonexpansive when $\|Nx-x^\star\|\leq\|x-x^\star\|$ where $x^\star$ is a fixed point of $N$.
\begin{definition}
$T\colon\mathbb{R}^d\to\mathbb{R}^d$ is $\alpha$-conically quasi-nonexpansive 
if there exists a quasi-nonexpansive operator $N\colon\mathbb{R}^d\to\mathbb{R}^d$
such that $T=(1-\alpha)\id + \alpha N$.
\end{definition}
This is a modification of conic nonexpansiveness in \citep[Definition 3.1]{bauschke2021generalized}. In \Cref{subsec: conic_star}, we show the conic quasi-nonexpansiveness (and related properties) of the resolvent of a star-cohypomonotone operator in view of \Cref{asp:3}, by invoking the corresponding arguments of \cite{bauschke2021generalized} restricted to a point in the domain and a fixed point of the resolvent. Then we show \emph{star}-cocoercivity of $\id-J_{\eta(F+G)}$ which facilitates the analysis of KM iteration.
\begin{fact}\label{fact: weakmvi}
    Let Assumptions \ref{asp:1} and \ref{asp:3} hold. 
    Then, we have
    \begin{enumerate}[(i)]
        \item The operator $J_{\eta(F+G)}$ is single-valued and $\dom J_{\eta(F+G)}=\mathbb{R}^d$ when $\eta < \frac{1}{L}$. \label{fact2: item1}
        \item The operator $J_{\eta(F+G)}$ is $\frac{1}{2\left(1-\frac{\rho}{\eta}\right)}$-conically quasi-nonexpansive and $\id - J_{\eta(F_G)}$ is $\left(1-\frac{\rho}{\eta}\right)$-star-cocoercive when $\rho < \eta$.
        \item For any $\bar x \in \mathbb{R}^d$, computing $J_{\eta(F+G)}(\bar x)$ is equivalent to solving the problem        \label{fact2: item3}
        \begin{equation}\label{eq: str_inc_wmi}
            \text{Find~} x\in\mathbb{R}^d \text{~such that~} 0 \in (\id+\eta(F+G))(x) - \bar x.
        \end{equation}
        The problem \eqref{eq: str_inc_wmi} has a unique solution when $\eta < \frac{1}{L}$.
        \item The operator $\id+\eta F$ is $(1+\eta L)$-Lipschitz and $(1-\eta L)$-strongly monotone when $\eta < \frac{1}{L}$.
    \end{enumerate}
\end{fact}
\begin{proof}~
\begin{enumerate}[(i)]
\item When $F$ is $L$-Lipschitz, it is maximally $L$-hypomonotone (see e.g., \citep[Lemma 2.12]{giselsson2021compositions}) and $\eta F$ is maximally $\eta L$-hypomonotone since it is $\eta L$-Lipschitz.

By \citep[Lemma 3.2(ii)]{dao2019adaptive}, we know that $\eta F+\id$ is maximally $(1-\eta L)$-(strongly) monotone. Then, using this and maximal monotonicity of $G$, we have by \citep[Corollary 25.5]{bauschke2017convex} that $\id+\eta(F+G)$ is maximally $(1-\eta L)$-(strongly) monotone. Invoking \citep[Lemma 3.2(ii)]{dao2019adaptive} again gives us that $\eta(F+G)$ is maximally $\eta L$-hypomonotone.

We can then use \citep[Lemma 2.8]{bauschke2021generalized} to obtain that $(\eta(F+G))^{-1}$ is maximally $\eta L$-cohypomonotone. This can be combined with \citep[Corollary 2.14]{bauschke2021generalized} to get the result when $\eta L < 1$.

\item Since $F+G$ has a $\rho$-weak MVI solution under Assumption \ref{asp:3}, we have that $\eta (F+G)$ has $\rho/\eta$-weak MVI solution, i.e., by simple change of variables, we have for some $\eta >0$
\begin{align*}
    &\langle \eta u, x-x^* \rangle \geq \eta \rho \| u\|^2 \text{~~where~~} u \in (F+G)(x) \\
    \iff&\langle v, x-x^* \rangle \geq \frac{\rho}{\eta} \| v\|^2 \text{~~where~~} v \in \eta(F+G)(x).
\end{align*}
 \Cref{prop: kju4} then gives us that $J_{\eta(F+G)}$ is $\frac{1}{2\left(1-\frac{\rho}{\eta}\right)}$-conically quasi-nonexpansive and then  we have that $\id - J_{\eta(F+G)}$ is $\left( 1-\frac{\rho}{\eta}\right)$ star-cocoercive by \Cref{cor: ser4}.

\item The proof is the same as \Cref{fact: cohypo}\eqref{fact1: item3} where the only difference is that now we ensure the existence of $J_{\eta(F+G)}(\bar x)$ with \eqref{fact2: item1}. Uniqueness also follows from this.

\item The proof is the same as \Cref{fact: cohypo}(iv).\qedhere
\end{enumerate}
\end{proof}
\subsubsection{Properties Related to Conic Quasi-Nonexpansiveness}\label{subsec: conic_star}
This section particularizes the notion and properties of the $\alpha$-conic nonexpansiveness
in~\cite{bauschke2021generalized} to their \emph{star} variants. The aim is to show that the properties extend to their \emph{star} or \emph{quasi}-variants when we use weak MVI condition instead of cohypomonotonicity.
This sections implicitly assumes that $J_A$ for operator $A\colon\mathbb{R}^d\rightrightarrows\mathbb{R}^d$ is well-defined, sufficient conditions for which is shown in \Cref{fact: weakmvi}.
We say that an operator $N$ is quasi-nonexpansive when $\|Nx-x^\star\|\leq\|x-x^\star\|$ where $x^\star$ is a fixed point of $N$.

\begin{lemma}{(See \citep[Lemma 3.4]{bauschke2021generalized})}\label{lem:conic-star}
Consider $T\colon\mathbb{R}^d\to\mathbb{R}^d$ and let $T=(1-\alpha)\id+\alpha N$. Then,
$N$ is quasi-nonexpansive
if and only if we have, for all $x\in\mathbb{R}^d$,
\[
2\alpha\inprod{Tx-x^\star}{(\id-T)x} \ge (1-2\alpha)\|(\id-T)x\|^2,
\]
or equivalently
\begin{align}
\left\|\left(1-\frac{1}{\alpha}\right)x + \frac{1}{\alpha}Tx - x^\star\right\| \le \|x-x^\star\|.
\label{eq:star_nonexpansive}
\end{align}
\end{lemma}
\begin{proof}
Using $\alpha^2\|a\|^2 - \|(\alpha-1)a+b\|^2 = 2\alpha\inprod{b}{a-b} - (1-2\alpha)\|a-b\|^2$ (see \citep[Lemma 3.3]{bauschke2021generalized})
with $a=x-x^\star$ and $b=Tx-x^\star$, we have
\begin{align*}
0 \le& \ 2\alpha\inprod{Tx-x^\star}{(\id-T)x} - (1-2\alpha)\|(\id-T)x\|^2 \\
=& \ \alpha^2\|x-x^\star\|^2 - \|(\alpha-1)(x-x^\star) + Tx-x_*\|^2 \\
=& \ \alpha^2\|x-x^\star\|^2 - \|(\alpha-1)(x-x^\star) + (1-\alpha)(x-x^\star) + \alpha(Nx-x^\star)\|^2 \\
=& \ \alpha^2(\|x-x^\star\|^2 - \|Nx-x^\star\|^2),
\end{align*}
which gives the assertion. Last claim follows by substituting $N=\frac{1}{\alpha} T + \left(1-\frac{1}{\alpha}\right) \id$ in the definition of quasi-nonexpansiveness for $N$.
\end{proof}

\begin{corollary}{(See \citep[Corollary 3.5(iii)]{bauschke2021generalized})}\label{cor: ser4}
$T\colon\mathbb{R}^d\to\mathbb{R}^d$ is $\alpha$-conically quasi-nonexpansive
if and only if $\id - T$ is $\frac{1}{2\alpha}$-star-cocoercive.
\end{corollary}
\begin{proof}
We use Lemma \ref{lem:conic-star}:
\begin{align*}
\inprod{Tx-x^\star}{(\id-T)x} \ge \left(\frac{1}{2\alpha}-1\right)\|(\id-T)x\|^2 
\quad \Leftrightarrow \quad \inprod{x-x^\star}{(\id-T)x} \ge \frac{1}{2\alpha}\|(\id-T)x\|^2
,\end{align*}
which is simply adding to both sides $\|(\id-T)(x)\|^2$.
\end{proof}

\begin{proposition}{(See \citep[Proposition 3.6(i)]{bauschke2021generalized})}\label{prop: kju4}
Let $A = T^{-1} - \id$ and set $N = \frac{1}{\alpha}T - \frac{1-\alpha}{\alpha}\id$,
{\it i.e.,} $T = J_A = (\id + A)^{-1} = (1-\alpha)\id + \alpha N$.
Then, $T$ is $\alpha$-conically quasi-nonexpansive if and only if 
$A$ is $\left(1-\frac{1}{2\alpha}\right)$-star-cohypomonotone, {\it i.e.,}
\[
\langle x - x^\star, u\rangle \ge -\left(1-\frac{1}{2\alpha}\right)\|u\|^2~~~\forall (x, u) \in \gra A
.\]
\end{proposition}
\begin{proof}
\begin{itemize}
\item[] We see the two directions:
\item[``$\Rightarrow$''] Let $(x,u) \in \gra A$. Then by definition of $A=T^{-1}-\id$ and manipulations, it follows that $(x,u) = (T(x+u),(\id-T)(x+u))$.
    By Lemma~\ref{lem:conic-star} invoked with $x\leftarrow x+u$, we have
    \begin{align*}
    & \quad 2\alpha\inprod{T(x+u)-x^\star}{(\id-T)(x+u)} \ge (1-2\alpha)\|(\id-T)(x+u)\|^2 \\
    \Leftrightarrow & \quad 2\alpha\inprod{x-x^\star}{u} \ge (1-2\alpha)\|u\|^2,
    \end{align*}
    where the last step substituted $(x,u) = (T(x+u),(\id-T)(x+u))$.
\item[``$\Leftarrow$'']
    Since $(Tx,(\id-T)x)\in\gra A$, we have by star-cohypomonotonicity that
    $\inprod{Tx-x^\star}{(\id-T)x} \ge \left(\frac{1}{2\alpha}-1\right)\|(\id-T)x\|^2$. In view of \Cref{lem:conic-star}, we deduce conic quasi-nonexpansiveness.
\end{itemize}
\end{proof}
\subsection{Complexity of the Outer Loop}\label{sec: weakmvi_outer}
\paragraph{Bounding the norm of iterates. } Just like \Cref{sec: halpern_app}, we start with the bound of the norms of the iterates.

\begin{lemma}\label{lem: wlk4}
    Let Assumptions \ref{asp:1} and \ref{asp:3} hold. Suppose that the iterates $(x_k)$ of \Cref{alg:weakmvi} satisfy
$\| J_{\eta(F+G)}(x_k) - \widetilde{J}_{\eta(F+G)}(x_k)\| \le \varepsilon_k$ for some $\varepsilon_k >0$ and $\rho<\eta$. Then, we have for $k\geq 0 $ that
    \begin{equation*}
        \|x_{k+1}-x^\star \| \leq \| x_k - x^\star \| + \left(1-\frac{\rho}{\eta}\right) \varepsilon_k.
    \end{equation*}
\end{lemma}
\begin{proof}
From \Cref{fact: weakmvi}(ii), we know that $J_{\eta(F+G)}$ is $\frac{1}{2\left( 1-\frac{\rho}{\eta} \right)}$-conically quasi-nonexpansive. Then, by property \eqref{eq:star_nonexpansive} derived in Lemma \ref{lem:conic-star}, since $J_{\eta(F+G)}$ is also $\frac{1}{1-\frac{\rho}{\eta}}$-conically quasi-nonexpansive due to $2\left( 1-\frac{\rho}{\eta} \right) \geq 1-\frac{\rho}{\eta}$ (see also \Cref{cor: ser4}), we have
    \begin{align}
        \left\|\frac{\rho}{\eta}x_{k} + \left(1-\frac{\rho}{\eta}\right)J_{\eta(F+G)}(x_k) - x^\star \right\| 
        \leq \|x_k-x^\star\|.\label{eq: sgr4}
    \end{align}
    By the definition of $x_{k+1}$ in Algorithm \ref{alg:weakmvi}, the definition of $\varepsilon_k$ and triangle inequality, we have for $k\geq 0$ that
    \begin{align*}
        \|x_{k+1} -x^\star \| &\leq \left\| \frac{\rho}{\eta}x_{k} + \left(1-\frac{\rho}{\eta}\right){J}_{\eta(F+G)}(x_k)  - x^\star\right\| + \left(1-\frac{\rho}{\eta}\right) \| J_{\eta(F+G)}(x_k) - \widetilde{J}_{\eta(F+G)}(x_k)\| \\
        &\leq \left\|\frac{\rho}{\eta}x_{k} + \left(1-\frac{\rho}{\eta}\right){J}_{\eta(F+G)}(x_k)  - x^\star\right\| + \left(1-\frac{\rho}{\eta}\right)\varepsilon_k.
    \end{align*}
    Combining with \eqref{eq: sgr4} gives the result.
\end{proof}
\paragraph{Iteration complexity. } Equipped with this result, we proceed to deriving the iteration complexity of the outer loop.
\begin{replemma}{lem: rate_outer}
Let Assumptions \ref{asp:1} and \ref{asp:3} hold. Suppose that the iterates $(x_k)$ of \Cref{alg:weakmvi} satisfy
$\| J_{\eta(F+G)}(x_k) - \widetilde{J}_{\eta(F+G)}(x_k)\| \le \varepsilon_k$ for some $\varepsilon_k >0$ and $\rho<\eta$.
Then, we have for $K\geq 1$ that
\begin{equation} \label{eq:hr1}
\sum_{k=0}^{K-1}\|(\id - J_{\eta(F+G)})(x_k)\|^2 \le
 \frac{2\eta^2}{(\eta-\rho)^2}\|x_0 - x^\star\|^2 + 6\sum_{k=0}^{K-1} \varepsilon_k^2+\frac{4\eta}{\eta-\rho} \sum_{k=0}^{K-1} \|x_k-x^\star\|\varepsilon_k,
\end{equation}
where 
\begin{equation*}
        \|x_{k}-x^\star \| \leq \| x_{k-1} - x^\star \| + \left(1-\frac{\rho}{\eta}\right) \varepsilon_{k-1}.
    \end{equation*}
\end{replemma}
\begin{proof}
From \Cref{fact: weakmvi}(ii), we have that $\id - J_{\eta(F+G)}$ is $\left(1-\frac{\rho}{\eta}\right)$-star cocoercive. Let us recall our running notations:
\begin{equation*}
\alpha = 1-\frac{\rho}{\eta},~~~R = \id - J_{\eta(F+G)},~~~\widetilde{R}= \id - \widetilde{J}_{\eta(F+G)}.
\end{equation*} 
As a result, we have the following equivalent representation of $x_{k+1}$ (see the definition in Algorithm \ref{alg:weakmvi}):
\begin{align}
    x_{k+1} &= x_k - \left(1-\frac{\rho}{\eta}\right)\left(\id -\widetilde{J}_{\eta(F+G)}\right)(x_k) \notag \\
    &= x_k - \alpha \widetilde{R}(x_k).\label{eq: cls4}
\end{align}
    Then, by $\alpha$-star-cocoercivity of $R$, we have
    \begin{align}\label{eq: rml4}
        \langle R(x_k), x_{k}-x^\star \rangle \geq \alpha \|R(x_k)\|^2.
    \end{align}
    A simple decomposition gives
    \begin{align}
        \langle R(x_k), x_k - x^\star \rangle &= \langle \widetilde{R}(x_k), x_k - x^\star \rangle + \langle R(x_k) - \widetilde{R}(x_k), x_k - x^\star \rangle. \label{eq: vnb4}
    \end{align}
    We estimate the first term on the right-hand side of \eqref{eq: vnb4} as
    \begin{align}
        \langle \widetilde{R}(x_k), x_k - x^\star \rangle &=
        \frac{1}{\alpha}\langle x_k -x_{k+1}, x_k - x^\star \rangle \notag \\
        &= \frac{1}{2\alpha} \left( \|x_k - x_{k+1}\|^2 + \| x_k- x^\star\|^2 - \| x_{k+1} - x^\star\|^2 \right) \notag\\
        &\leq \frac{1}{2\alpha}\left( \|x_k - x^\star\|^2 -\|x_{k+1}- x^\star\|^2 \right) + \frac{3\alpha}{4}\|R(x_k)\|^2 + \frac{3\alpha}{2} \| \widetilde{R}(x_k) - R(x_k)\|^2 \notag\\
        &\leq \frac{1}{2\alpha}\left( \|x_k - x^\star\|^2 -\|x_{k+1}- x^\star\|^2 \right) + \frac{3\alpha}{4}\|R(x_k)\|^2 + \frac{3\alpha\varepsilon_k^2}{2},\label{eq: vbn54}
    \end{align}
    where we used the definition of $x_{k+1}$ from  \eqref{eq: cls4} in the first step, standard expansion $\|a-b\|^2 = \|a\|^2-2\langle a,b\rangle + \|b\|^2$ for the second step, the definition of $x_{k+1}$ from  \eqref{eq: cls4} and Young's inequality in the third step, and the definitions of $R_k, \widetilde{R}_k, \varepsilon_k$ in the last step.
    
    For the second term on the right-hand side of \eqref{eq: vnb4}, we have by Cauchy-Schwarz inequality and the definition of $\widetilde{R}$ and $\varepsilon_k$ that
    \begin{align}
        \langle R(x_k) - \widetilde{R}(x_k), x_k - x^\star \rangle &\leq \|R(x_k) - \widetilde{R}(x_k)\| \| x_k - x^\star\| \notag \\
        &\leq \|x_k - x^\star\| \varepsilon_k.\label{eq: sdso4}
    \end{align}
    We combine \eqref{eq: vbn54} and \eqref{eq: sdso4} in \eqref{eq: vnb4}, plug in the result to \eqref{eq: rml4} and rearrange to obtain
    \begin{align*}
        \frac{\alpha}{4}\|R(x_k)\|^2 \leq \frac{1}{2\alpha}\left(\|x_k - x^\star\|^2 - \| x_{k+1} - x^\star\|^2 \right) + \frac{3\alpha\varepsilon_k^2}{2}+\|x_k-x^\star\|\varepsilon_k.
    \end{align*}
    The result follows by multiplying both sides by $4/\alpha$, summing for $k=0, 1, \dots, K-1$, and using the definition of $\alpha$. The bound on $\|x_k-x^\star\|^2$ follows by \Cref{lem: wlk4}.
\end{proof}
\subsection{Complexity of the Inner Loop}
In a modular fashion, we will use precisely the same algorithm for the inner loop, i.e., the Forward-Backward-Forward (FBF) algorithm of \cite{tseng2000modified} like the Section \ref{subsubsec: cohypo_inner}. Hence the complexity of the inner loop is the same as Theorem \ref{th: fbf}. As we see in the next section, the accuracy required by $\widetilde{J}_{\eta(F+G)}$ is slightly different leading to the number of inner loop iterations $N_k$ in \Cref{alg:weakmvi} to be slightly different than \Cref{alg:cohypo}.

\subsection{Total Complexity}\label{subsec: weakmvi_total_app}
\begin{reptheorem}{th: weakmvi_det}
Let Assumptions \ref{asp:1} and \ref{asp:3} hold. Let $\eta < \frac{1}{L}$ in Algorithm \ref{alg:weakmvi} and suppose that $\rho < \eta$. For any $K>1$, we have that
    \begin{equation*}
        \frac{1}{K}\sum_{k=0}^{K-1} \frac{1}{\eta^2}\|x_k - J_{\eta(F+G)}(x_k) \|^2 \leq \frac{11\|x_0-x^\star\|^2}{(\eta-\rho)^2 K}.
    \end{equation*}
    The number of first-order oracles used at iteration $k$ of Algorithm \ref{alg:weakmvi} is upper bounded by
    \begin{equation*}
\left\lceil \frac{4(1+\eta L)}{1-\eta L}\log(8(k+2)\log^2(k+2))\right\rceil.
\end{equation*}
\end{reptheorem}
\begin{remark}
    It is straightforward to convert this to a last-iterate result if we additionally assume cohypomonotonicity as in \cite{pethick2023stable}, but we refrain from doing so since the main point of this section is to \emph{relax} cohypomonotonicity.
\end{remark}
\begin{proof}[Proof of Theorem \ref{th: weakmvi_det}]
Recall the notations $\alpha = 1-\frac{\rho}{\eta}$ and $R= \id - J_{\eta(F+G)}$.
    Let us set 
    \begin{equation}\label{eq: weakmvi_epsk_def}
\varepsilon_k = \frac{1}{8(k+1)\log^2(k+2)}\|x_k - J_{\eta(F+G)}(x_k)\|
\end{equation}
and note, just as in the proof of Theorem \ref{th: cohypo_det}, that we will not evaluate the value of $\varepsilon_k$ but we will prove that for the number of iterations that FBF runs at each KM iteration in Algorithm \ref{alg:weakmvi}, the error criterion required by $\varepsilon_k$ is satisfied.

By using the definition of $\varepsilon_k$ in Lemma \ref{lem: wlk4} gives
    \begin{equation}\label{eq: bnt4}
        \| x_{k+1} - x^\star \| \leq \|x_k - x^\star \| + \frac{\alpha}{8(k+1)\log^2(k+2)}\|x_k - J_{\eta(F+G)}(x_k)\|.
    \end{equation}
    We note that $\alpha=1-\frac{\rho}{\eta} \leq 1$ and since $R= \id - J_{\eta(F+G)}$ is $\alpha$-star cocoercive as shown in \Cref{fact: weakmvi}(ii), we have that $\id-J_{\eta(F+G)}$ is $\alpha^{-1}$-star Lipschitz and hence by $(\id-J_{\eta(F+G)})(x^\star) = 0$ we have
    \begin{equation}\label{eq: bui5}
        \|(\id-J_{\eta(F+G)})(x_k) \| = \|(\id-J_{\eta(F+G)})(x_k) - (\id-J_{\eta(F+G)})(x^\star)\| \leq \alpha^{-1} \| x_k - x^\star\|.
    \end{equation}
    Consequently, \eqref{eq: bnt4} becomes, after summing for $k=0,1,\dots,K-1$ that
    \begin{align*}
        \| x_{K} - x^\star \| &\leq \|x_0 - x^\star \| + \sum_{i=0}^{K-1} \frac{1}{8(i+1)\log^2(i+2)}\|x_i - x^\star\|.
    \end{align*}
    With this, we can show by induction that 
    \begin{equation}\label{eq: sft4}
        \|x_{k} - x^\star\| \leq 2\|x_0- x^\star\|~~~ \forall k \geq 0,
    \end{equation} 
    because $\sum_{i=0}^\infty \frac{1}{(i+1)\log^2(i+2)}< 4$.
    
    We use  \eqref{eq: sft4} in the result of Lemma \ref{lem: rate_outer} to obtain (also noting the definitions of $\alpha$ and $R$)
    \begin{align}\label{eq: okt6}
        \sum_{k=0}^{K-1} \|R(x_k)\|^2 &\leq \frac{2}{\alpha^2}\|x_0-x^\star\|^2 + \sum_{k=0}^{K-1} 6 \varepsilon_k^2 + \frac{4}{\alpha}\sum_{k=0}^{K-1} 2\| x_0 - x^\star\|\varepsilon_k.
    \end{align}
    By using \eqref{eq: sft4} and \eqref{eq: bui5} in \eqref{eq: weakmvi_epsk_def} we also know the following upper bound on $\varepsilon_k$:
    \begin{equation*}
        \varepsilon_k \leq \frac{\|x_0-x^\star\|}{4\alpha(k+1)\log^2(k+2)}.
    \end{equation*}
    With this, \eqref{eq: okt6} becomes
    \begin{align}
        \sum_{k=0}^{K-1} \|R(x_k)\|^2 &\leq \frac{2}{\alpha^2}\|x_0-x^\star\|^2 + \sum_{k=0}^{K-1} \frac{3\|x_0-x^\star\|^2}{8\alpha^2 (k+1)^2\log^4(k+2)} + \sum_{k=0}^{K-1} \frac{2\|x_0-x^\star\|^2}{\alpha^2(k+1)\log^2(k+2)}\notag \\
        &<\frac{11}{\alpha^2} \| x_0-x^\star\|^2,\label{eq: ufg4}
    \end{align}
    since $\sum_{k=0}^{K-1}\frac{3}{8(k+1)^2\log^4(k+2)} < 2$ and $\sum_{k=0}^{K-1} \frac{2}{(k+1)\log^2(k+1)} < 7$. This establishes the first part of the assertion.

    We next see that, with $N_k$ set as in \Cref{alg:weakmvi}, we get the inexactness level specified by $\varepsilon_k$ in \eqref{eq: weakmvi_epsk_def} and we verify that the oracle complexity of each iteration is as claimed in the statement.

For the second part of the result, we proceed similar to the proof of \Cref{th: cohypo_det}. Namely, at iteration $k$, we  apply the result in Theorem \ref{th: fbf}. For this, let us identify the following from the definitions in \Cref{alg:weakmvi} 
\begin{align*}
    &A\equiv\eta G, ~~~B(\cdot)\equiv(\id+\eta F)(\cdot)-x_k, ~~~z_0 \equiv x_k,~~~z^\star \equiv J_{\eta(F+G)}(x_k), ~~~\zeta\equiv\varepsilon_k\\
    \Longrightarrow~&z_0 - z^\star = (\id-J_{\eta(F+G)})(x_k) = R(x_k).
\end{align*}
As before, we have that $B$ is $(1+\eta L)$-Lipschitz and $(1-\eta L)$-strongly monotone due to \Cref{fact: weakmvi}(iv). Existence of $z^\star$ is guaranteed by \Cref{fact: weakmvi}(iii) since $\eta < \frac{1}{L}$.

We now see that by the setting of $N_k$ from \Cref{alg:weakmvi} and definition of $\varepsilon_k$ in \eqref{eq: weakmvi_epsk_def}, we have
\begin{equation*}
N_k=\left\lceil \frac{4(1+\eta L)}{1-\eta L}\log(8(k+1)\log^2(k+2))\right\rceil = \left\lceil \frac{4(1+\eta L)}{1-\eta L}\log\frac{\|R(x_k)\|}{\varepsilon_k}\right\rceil.
\end{equation*}
With this value, \Cref{th: fbf} gives us 
\begin{equation*}
    \| \widetilde{J}_{\eta(F+G)}(x_k) - J_{\eta(F+G)}(x_k)\| \leq\varepsilon_k
\end{equation*}
as claimed.

Since each iteration of \Cref{alg:fbf} uses $2$ evaluations of $F$ and $1$ resolvent for $G$, the number of first-order oracle calls at iteration $k$ is $2N_k$ and the result follows.
\end{proof}
We continue with the proof of \Cref{cor: weakmvi_det} which follows trivially from \Cref{th: weakmvi_det}.
\begin{proof}[Proof of \Cref{cor: weakmvi_det}]
     Based on \Cref{th: weakmvi_det}, we have that after $K$ iterations where
     \begin{equation}\label{eq: wmi_k_limit}
        K \leq \left\lceil\frac{11\|x_0 - x^\star\|^2}{\eta^2\alpha^2\varepsilon^2}\right\rceil     
     \end{equation}
     we are guaranteed to obtain
    \begin{equation*}
         \min_{0\leq k \leq K-1} \eta^{-1}\|R(x_k)\|\leq\frac{1}{K}\sum_{k=0}^{K-1}\eta^{-1}\|R(x_k)\| \leq\varepsilon.
    \end{equation*}
    Total number of first-oracle calls during the run of the algorithm then be calculated as
\begin{equation*}
    \sum_{k=1}^K \left\lceil \frac{4(1+\eta L)}{1-\eta L}\log(8(k+2)\log^2(k+2))\right\rceil \leq K\cdot \left(\frac{4(1+\eta L)}{1-\eta L}\log(8(K+2)\log^2(K+2))+1\right).
\end{equation*}
We conclude after using \eqref{eq: wmi_k_limit}.
\end{proof}
\subsection{Additional Results}\label{sec: weakmvi_remark_exp}
Let us re-emphasize the strategy in the previous proof: we set a target value for $\varepsilon_k$ and then we prove that when we run the inner algorithm FBF for a certain \emph{computable} number of iterations $N_k$, the criterion enforced on $\widetilde{J}_{\eta(F+G)}$  by $\varepsilon_k$ is satisfied. However, this number of inner iterations is \emph{worst-case}. Another alternative, which could be more useful in practice is to set $\varepsilon_k$ to a computable value and monitor the progress of the inner algorithm and break when $\varepsilon_k$ is attained. One sidenote is that this is attainable in the deterministic case considered in this section, however it cannot be done in the stochastic case since the convergence guarantees are generally given in expectation.

This described strategy can  be made  rigorous with slight changes in the constants in our deterministic case. We now see this in the next proposition.

\begin{corollary}\label{cor: computable_eps}
    Let Assumptions \ref{asp:1} and \ref{asp:3} hold and let $G=\partial \iota_C$ for a convex closed set $C$. Let $\eta < \frac{1}{L}$ and $\rho < \eta$ in Algorithm \ref{alg:cohypo} with $\|\widetilde{J}_{\eta(F+G)}(x_k) - J_{\eta(F+G)}(x_k)\|\leq\frac{c}{k\log^2(k+2)}$ for any $c > 0$ and use \citep[Algorithm 4]{diakonikolas2020halpern} to obtain such $\widetilde{J}_{\eta(F+G)}(x_k)$ at iteration $k$. Then, we have that
    \begin{equation*}
        \frac{1}{K}\sum_{k=0}^{K-1} \eta^{-1}\| (\id-J_{\eta(F+G)})(x_k)\| \leq \varepsilon,
    \end{equation*}
    with the number of calls to evaluation of $F$ and resolvent of $G$ is bounded by $ \tilde O\left( \frac{(1+\eta L)((1+c)\|x_0 - x^\star\|^2+c^2)}{\varepsilon^2 (\eta-\rho)^2(1-\eta L)} \right)$.
\end{corollary}
\begin{remark}
    Note that \citep[Algorithm 4]{diakonikolas2020halpern} has a built-in stopping criterion to terminate the algorithm when the required accuracy is achieved. The value for $\varepsilon_k$ defined in this corollary is computable since it only depends on $k$ and a user-defined constant $c$. This is an alternative to \texttt{FBF} we used in the main text where we use a computable number of iterations to run the inner algorithm rather than using a stopping criterion as \citep[Algorithm 4]{diakonikolas2020halpern}. On the one hand, in practice, a stopping criterion can be more desirable since the worst-case number of iterations can be pessimistic.
    On the other hand, the strategy of using a stopping criterion is inherently more complicated in the stochastic case whereas using a worst-case computable number of inner iteration is still easily implementable. This is why we considered the latter setting throughout the paper. However, this corollary is still included for the former strategy.
\end{remark}
\begin{proof}[Proof of \Cref{cor: computable_eps}]
    We obtain the result by modifying the proof of \Cref{th: weakmvi_det}. We set
    \begin{equation*}
        \varepsilon_k = \frac{c}{(k+1)\log^2(k+2)},
    \end{equation*}
    for any $c>0$.
    
    By using this on \Cref{lem: wlk4} and summing the result for $k=0,1,\dots, K-1$ we obtain
    \begin{align*}
        \|x_k - x^\star\| &\leq \| x_0 -x^\star\|+\alpha\sum_{k=0}^{K-1}\frac{c}{(k+1)\log^2(k+2)} \\
        &\leq \| x_0 -x^\star\|+4\alpha c,
    \end{align*}
    since $\sum_{k=0}^{K-1} \frac{1}{(k+1)\log^2(k+2)}<4$.
    We use this bound on the result of \Cref{lem: rate_outer} to obtain
    \begin{align*}
        \sum_{k=0}^{K-1} \| R(x_k)\|^2 &\leq \frac{2}{\alpha^2}\|x_0 -x^\star\|^2 + \sum_{k=0}^{K-1} \frac{6c^2}{(k+1)^2\log^4(k+2)} + \frac{4}{\alpha}\sum_{k=0}^{K-1}\frac{c(\|x_0-x^\star\|+4\alpha c)}{(k+1)\log^2(k+2)}\\
        &\leq \left(\frac{2}{\alpha^2}+\frac{16c}{\alpha}\right)\|x_0 -x^\star\|^2+30c^2+64c^2,
    \end{align*}
    which gives the result after dividing by $\eta^2$ and noting that \citep[Lemma 17]{diakonikolas2020halpern} gives complexity $\widetilde{O}\left( \frac{1+\eta L}{1-\eta L} \right)$ for obtaining such a $\widetilde{J}_{\eta(F+G)}(x_k)$ with \citep[Algorithm 4]{diakonikolas2020halpern}.
\end{proof}
We continue with the result mentioned in \Cref{rem: xout}.
\begin{corollary}\label{cor: weakmvi_remark_exp}
Let Assumptions \ref{asp:1} and \ref{asp:3} hold. Let $\eta < \frac{1}{L}$ and $\rho < \eta$.
\begin{enumerate}[(i)]
    \item Let $G\equiv0$ and consider Algorithm \ref{alg:weakmvi}. Then we have that $\min_{0\leq k\leq K-1} \| F(x_k)\|\leq 2\varepsilon$ with the first-order oracle calls  bounded by 
    \begin{equation*}
        \widetilde{O}\left( \frac{(1+\eta L)\|x_0 - x^\star\|^2}{\varepsilon^2  \left(\eta-\rho\right)^2(1-\eta L)} \right).
    \end{equation*}   
    \item Let $G\equiv \partial \iota_C$ for a convex closed set $C\subseteq \mathbb{R}^d$. Given $\varepsilon > 0$, consider Algorithm \ref{alg:weakmvi} with the update $\widetilde{J}_{\eta(F+G)}(x_k)$ replaced with \citep[Algorithm 4]{diakonikolas2020halpern} with error criterion $\|\widetilde{J}_{\eta(F+G)}(x_k) - J_{\eta(F+G)}(x_k)\|\leq \frac{\eta\varepsilon^2}{(k+1)\log^2(k+3)}$. Then, for $x^{\text{out}}=\arg\min_{x\in\{x_0, \dots, x_{k-1}\}}\|x - \widetilde{J}_{\eta(F+G)}(x)\|$, we have that $\eta^{-1}\|(\id-J_{\eta(F+G)})(x^{\text{out}})\|\leq2\varepsilon+3\varepsilon^2$ with the first-order oracle calls  bounded by 
    \begin{equation}\label{eq: npo6}
        \widetilde{O}\left( \frac{(1+\eta L)\|x_0 - x^\star\|^2}{\varepsilon^2  \left(\eta-\rho\right)^2(1-\eta L)} \right).
    \end{equation}
\end{enumerate}
\end{corollary}
See also \Cref{rem: opt_meas} for details on how we can use this result to further obtain a guarantee like $\dist(0, (F+G)(x^{\text{out}})) \leq \varepsilon$.
\begin{proof}[Proof of \Cref{cor: weakmvi_remark_exp}]
    \begin{enumerate}[(i)]
        \item In this case, we start from the final steps of the proof of \Cref{th: weakmvi_det} (see \eqref{eq: ufg4}) which, after using $R=\id-J_{\eta(F+G)}$, gives us that
        \begin{equation}\label{eq: npo4}
            \frac{1}{K}
            \sum_{k=0}^{K-1} \eta^{-2}\| x_k - J_{\eta F}(x_k)\|^2 \leq \varepsilon^2,
        \end{equation}
        with the prescribed complexity bound given in \Cref{cor: weakmvi_det}. Let us define $\bar x_k = J_{\eta F}(x_k)$. 
        
        On the one hand, we use the definition of resolvent to obtain
        \begin{align*}
            \bar x_k = J_{\eta F}(x_k) \iff \bar x_k + \eta F(\bar x_k) = x_k \iff x_k - \bar x_k = \eta F(\bar x_k),
        \end{align*}
        which, in view of \eqref{eq: npo4}, means that we have
        \begin{equation}\label{eq: npo5}
            \frac{1}{K}
            \sum_{k=0}^{K-1} \| F(\bar x_k)\|^2 \leq \varepsilon^2.
        \end{equation}
        On the other hand, we know by Young's inequality and Lipschitzness of $F$ that
        \begin{align*}
            \frac{1}{K}
            \sum_{k=0}^{K-1}\|F(x_k)\|^2 &\leq \frac{1}{K}
            \sum_{k=0}^{K-1}2\|F(\bar x_k)\|^2 + \frac{1}{K}
            \sum_{k=0}^{K-1} 2\| F(x_k) - F(\bar x_k)\|^2 \\
            &\leq \frac{1}{K}
            \sum_{k=0}^{K-1}2\|F(\bar x_k)\|^2 + \frac{1}{K}
            \sum_{k=0}^{K-1} 2L^2\| x_k - \bar x_k\|^2 \\
            &\leq (2+2\eta^2 L^2) \varepsilon ^2\\
            &< 4\varepsilon^2,
        \end{align*}
        where we used \eqref{eq: npo4} and \eqref{eq: npo5}.
        \item A slight modification of the proof of \Cref{cor: computable_eps} by using $\varepsilon_k = \frac{\eta\varepsilon^2}{(k+1)\log^2(k+3)}$ gives us that
        \begin{equation}\label{eq: soi4}
            \frac{1}{K} \sum_{k=0}^{K-1} \eta^{-2}\|(\id-J_{\eta(F+G)})(x_k)\|^2 \leq \varepsilon^2
        \end{equation}
        with the complexity bound \eqref{eq: npo6}. This is because \citep[Lemma 17]{diakonikolas2020halpern} showed that \citep[Algorithm 4]{diakonikolas2020halpern} outputs a $\widetilde{J}_{\eta(F+G)}(x_k)$ satisfying the requirement set by $\varepsilon_k=\frac{\eta\varepsilon^2}{(k+1)\log^2(k+3)}$, with the same worst-case complexity as \Cref{th: fbf}. The difference is that \citep[Algorithm 4]{diakonikolas2020halpern} has a computable stopping criterion (instead of the maximum number of iterations Algorithm \ref{alg:fbf} takes) where we can check if $\varepsilon_k=\frac{\eta\varepsilon^2}{(k+1)\log^2(k+3)}$ accuracy is achieved and break the loop. 
        
        Since we have the pointwise bound $\|J_{\eta(F+G)}(x_k) -\widetilde{J}_{\eta(F+G)}(x_k)\|^2\leq\eta^2\varepsilon^4$, we derive from \eqref{eq: soi4} that
        \begin{equation*}
            \frac{1}{K} \sum_{k=0}^{K-1} \eta^{-2}\|(\id-\widetilde{J}_{\eta(F+G)})(x_k)\|^2 \leq 2(\varepsilon^2+\varepsilon^4).
        \end{equation*}
        Hence, for $x^{\text{out}}$ defined in the statement, we get
        \begin{equation}\label{eq: soi5}
            \eta^{-2}\|(\id-\widetilde{J}_{\eta(F+G)})(x^{\text{out}})\|^2 \leq 2(\varepsilon^2+\varepsilon^4).
        \end{equation}
        Then, by using the pointwise bound $\|J_{\eta(F+G)}(x_k) -\widetilde{J}_{\eta(F+G)}(x_k)\|^2\leq\eta^2\varepsilon^4$ for all $k$, we know that
        \begin{align*}
            \eta^{-1}\|(\id-J_{\eta(F+G)})(x^{\text{out}}) \| &\leq \eta^{-1}\|(\id-\widetilde{J}_{\eta(F+G)})(x^{\text{out}}) \| + \eta^{-1}\|(J_{\eta(F+G)}-\widetilde{J}_{\eta(F+G)})(x^{\text{out}}) \| \\
            &\leq \varepsilon^2+\sqrt{2(\varepsilon^4+\varepsilon^2)} < 2\varepsilon+3\varepsilon^2,
        \end{align*}
        which uses \eqref{eq: soi5} and the implication of the error criterion $\|J_{\eta(F+G)}(x_k) - \widetilde{J}_{\eta(F+G)}(x_k)\|\leq\eta\varepsilon^2$, completing the proof.\qedhere
    \end{enumerate}
\end{proof}

\begin{table*}[t]
\centering\small
\begin{tabular}{ l|l|c c c c} 
 Assumption & Reference  & Limit of $\rho$ &Constraints& Oracle$^\dagger$  &Complexity \\
\hline
weak MVI 
&\cite{diakonikolas2021efficient} & $\frac{1}{4\sqrt{2}L}$ & $\times$ & Single & $O(\varepsilon^{-4})$ \\ [2mm]
&\cite{choudhury2023singlecall} & $\frac{1}{2L}$ & $\times$ & Single & $O(\varepsilon^{-4})$ \\ [2mm]
&\cite{bohm2022solving} & $\frac{1}{2L}$ & $\times$ &  Single & $O(\varepsilon^{-6})$ \\ [2mm]
&\cite{pethick2023solving} & $\frac{1}{2L}$ & $\checkmark$ & Multiple  & $\widetilde{O}(\varepsilon^{-4})$ \\ [2mm]
 & \Cref{th: weakmvi_stoc} & $\frac{1}{L}$ &$\checkmark$ &  Single  &$\widetilde{O}(\varepsilon^{-4})$ \\ [2mm]
\hline
cohypomonotone &
\cite{pethick2023stable} & $\frac{1}{2L}$ & $\checkmark$ &  Single & $\widetilde{O}(\varepsilon^{-6})$ (best)$^\ddagger$ \\ [2mm]
 &
\cite{pethick2023stable} & $\frac{1}{2L}$ & $\checkmark$ &  Single & $\widetilde{O}(\varepsilon^{-16})$ (last) \\ [2mm]
& \cite{chen2022near}$^*$ & $\frac{1}{2L}$ & $\times$ & Single & $\widetilde{O}(\varepsilon^{-2})$ \\ [2mm]
 & \Cref{cor: cohypo_stoc}$^*$ & $\frac{1}{L}$ &$\checkmark$ &  Single & $\widetilde{O}(\varepsilon^{-4})$ \\[2mm]
\hline
\end{tabular}
\caption{\small Comparison of first order algorithms for stochastic problems.
Complexity refers to the number of oracle calls to get the fixed point residual $\mathbb{E}\|(\id-J_{\eta(F+G)})(x^\out)\|\leq\varepsilon$. See also Remark \ref{rem: opt_meas}. $^\dagger$Oracle access refers to the number of operator evaluations algorithm makes with one random seed given $F(x)=\mathbb{E}_{\xi\sim\Xi}[F_\xi(x)]$. For example, ``Single" refers to algorithms that only access one sample per seed, i.e., only $F_{\xi_t}(x_t)$, ``Multiple" is for algorithms that access multiple samples per seed, i.e., $F_{\xi_t}(x_t)$ and $F_{\xi_t}(x_{t-1})$. Algorithms with ``Multiple" access also make the additional assumption that $\mathbb{E}_{\xi\sim \Xi}\|F_\xi(x) - F_\xi(y)\|^2 \leq L^2 \| x-y\|^2$ which is stronger than mere Lipschitzness of $F$. $^\ddagger$(best) refers to \emph{best iterate} in view of \Cref{eq: weakmvi_rem1}; (last) refers to a last iterate convergence rate. $^*$These works have complexity as \emph{expected} number of oracle calls due to the use of MLMC estimator. See also \Cref{sec: table_clarif} for derivations of the complexities when they are not written explicitly in the existing works.
}
\label{table:2}
\end{table*}

\section{Proofs for Section \ref{sec: extensions}}\label{sec: extensions_app}
\paragraph{Notation. } We use the following definitions for conditional expectations: For expectation conditioned on the filtration generated by the randomness of $x_k, \dots, x_1$, we use $\mathbb{E}_k[\cdot]$ while analyzing \Cref{alg:cohypo_stoc_app} and \Cref{alg:weakmvi_stoc}. In the notation of \Cref{alg:fbf_stoc}, we similarly use $\mathbb{E}_{t+1/2}[\cdot]$ for the expectation conditioned on the filtration generated by the randomness of $z_{t+1/2}, z_t, \dots, z_1, z_{1/2}$.
$\mathrm{Unif}$ denotes the uniform distribution and $\mathrm{Geom}$ denotes the geometric distribution.

\Cref{table:2} summarizes the existing works for stochastic min-max problems satisfying cohypomonotonicity or weak MVI conditions.

\subsection{Analysis of the inner loop for stochastic problems}
The main change for  algorithms in the stochastic case is computing the resolvent approximation $\widetilde{J}_{\eta(F+G)}(x_k)$. We now need to invoke $\texttt{FBF}$ with unbiased oracles for $F$, see for example \eqref{eq: stoc_cohypo_step}. For ease of reference, we specify the algorithm below. Note that \Cref{alg:cohypo_stoc_app} is precisely \Cref{alg:cohypo} when \eqref{eq: stoc_cohypo_step} is used for estimating the resolvent and \Cref{alg:fbf_stoc} is precisely \Cref{alg:fbf} when unbiased oracle $\widetilde{B}$ is inputted rather than full operator $B$. \Cref{alg:fbf_stoc} is a stochastic version of FBF, which is analyzed in the monotone case by \cite{bohm2022two}.

\begin{algorithm}[h]
\caption{Stochastic Inexact Halpern iteration for problems with cohypomonotonicity}
\begin{algorithmic}
    \STATE {\bfseries Input:} Parameters $\beta_k=\frac{1}{k+2}, \eta, L, \rho$, $\alpha=1-\frac{\rho}{\eta}$, $K>0$, initial iterate $x_0\in\mathbb{R}^d$,
    subroutine $\texttt{FBF}$ given in Algorithm \ref{alg:fbf_stoc} \\
    \vspace{.2cm}
    \FOR{$k = 0, 1, 2,\ldots, K-1$}
        \STATE $\widetilde{J}_{\eta(F+G)}(x_k) = \texttt{FBF}\left(x_k, N_k, \eta G, \id + \eta \widetilde{F}, 1+\eta L\right)$, where $N_k=\left\lceil \frac{1734(k+2)^3\log^2(k+2)}{(1-\eta L)^2}\right\rceil$
        \STATE $x_{k+1} = \beta_k x_0 + (1-\beta_k)((1-\alpha)x_k + \alpha\widetilde{J}_{\eta(F+G)}(x_k))$
        \ENDFOR
      \end{algorithmic}
\label{alg:cohypo_stoc_app}
\end{algorithm}
\begin{algorithm}[h]
\caption{$\texttt{FBF}(z_0, T, A, \widetilde{B}_{\mathrm{in}}, L_B)$ from \citep{tseng2000modified} -- Stochastic}
\begin{algorithmic}
    \STATE {\bfseries Input:} Parameter $\tau_t=\frac{2}{(t+1)\mu+6L_B}$, initial iterate $z_0\in\mathbb{R}^d$, $\widetilde{B}(\cdot) = \widetilde{B}_{\mathrm{in}}(\cdot) - z_0$ \\
    \vspace{.2cm}
    \FOR{$t = 0, 1, 2,\ldots, N-1 $}
        \STATE $z_{t+1/2} = J_{\tau_t A}(z_t - \tau_t \widetilde{B}(z_t)) $ 
        \STATE $z_{t+1} = z_{t+1/2} +\tau_t \widetilde{B}(z_t) - \tau_t \widetilde{B}(z_{t+1/2})$
        \ENDFOR
      \end{algorithmic}
\label{alg:fbf_stoc}
\end{algorithm}

More particularly, we solve the following stochastic strongly monotone inclusion problem:
\begin{equation*}
    \text{Find~} x^\star \in \mathbb{R}^d \text{~such that~} 0\in (A+B)(x^\star), \text{~where~} B=\mathbb{E}_{\xi\sim \Xi}[B_{\xi}].
\end{equation*}

Similar results to next theorem appeared in \cite{hsieh2019convergence,kotsalis2022simple}. We provide a proof for being complete and precise since we could not find a particular reference for stochastic FBF with strong monotonicity and explicit constants. It is also worth noting that we do not focus on optimizing the non-dominant terms. A tight bound for all the terms can be found in \cite{kotsalis2022simple} who analyzed a different algorithm.
\begin{theorem}\label{th: stoc_fbf}
    Let $z^\star=(A+B)^{-1}(0)\neq \emptyset$, the operator $B$ be $L_B$-Lipschitz and $\mu$-strongly monotone with $\mu>0$, $A$ be maximally monotone. Let $\widetilde{B}\colon\mathbb{R}^d\to\mathbb{R}^d$ satisfy $\mathbb{E}[\widetilde{B}(x)] = B(x)$ and $\mathbb{E}\|\widetilde{B}(x) - B(x)\|^2 \leq \sigma^2$. Then, we have that the last iterate of \Cref{alg:fbf_stoc} after running for $T$ iterations, when initialized with $z_0$, and step size $\tau_t = \frac{2}{(t+1)\mu+6L_B}$ satisfies the bound
    \begin{equation*}
        \mathbb{E}\|z_T - z^\star\|^2 \leq \frac{6L_B/\mu\|z_0-z^\star\|^2 + 48\sigma^2/\mu^2}{T+6L_B/\mu}.
    \end{equation*}
    Each iteration of the algorithm uses two evaluations of $\widetilde{B}$ and one resolvent of $A$
\end{theorem}
\begin{proof}
    Note that  the definition of $z_{t+1/2}$ implies $\tau_t A(z_{t+1/2}) \ni z_t - z_{t+1/2} - \tau_t B_{\xi_t}(z_t)$. The definition of $z^\star$ implies $\tau_t A(z^\star) \ni -\tau_t B(z^\star)$ By using this with monotonicity of $A$, we get
    \begin{align*}
        \langle z_{t+1/2} - z_t + \tau_t \widetilde{B}(z_t) - \tau_t B(z^\star), z^\star-z_{t+1/2} \rangle \geq 0.
    \end{align*}
    By the definition of $z_{t+1}$, we then have
    \begin{align}\label{eq: ghr4}
        \langle z_{t+1} - z_t + \tau_t \widetilde{B}(z_{t+1/2})- \tau_t B(z^\star), z^\star-z_{t+1/2} \rangle \geq 0.
    \end{align}
    By taking expectation conditioned on $z_{t+1/2}$ and also using strong monotonicity of $B$, we also have 
    \begin{align}
        \mathbb{E}_{t+1/2}\langle \tau_t \widetilde{B}(z_{t+1/2})- \tau_t B(z^\star),  z_{t+1/2} - z^\star \rangle &= \langle \tau_t B(z_{t+1/2})- \tau_t B(z^\star),  z_{t+1/2} - z^\star \rangle \notag\\
        &\geq \mu\tau_t \| z^\star - z_{t+1/2}\|^2 \notag\\
        &\geq \frac{\mu \tau_t}{2} \| z^\star - z_{t+1}\|^2 - \mu \tau_t \| z_{t+1} - z_{t+1/2}\|^2\notag \\
        &\geq \frac{\mu \tau_t}{2} \| z^\star - z_{t+1}\|^2 - \frac{1}{3} \| z_{t+1} - z_{t+1/2}\|^2,\label{eq: eqw1}
    \end{align}
    where the third step is by Young's inequality and last step is by the definition of $\tau_t$, i.e., $\tau_t \mu = \frac{2\mu}{(t+1)\mu+6L_B} \leq \frac{2\mu}{6L_B}\leq \frac{1}{3}$ since $\mu \leq L_B$.
    
    We have, by the elementary identities $\langle a,b \rangle = \frac{1}{2}\left( \| a\|^2 + \|b\|^2 - \| a-b\|^2\right) =\frac{1}{2}\left( -\| a\|^2 - \|b\|^2 + \| a+b\|^2\right)$, that
    \begin{align}
        \langle z_{t+1} - z_t, z^\star - z_{t+1/2} \rangle &= \langle z_{t+1} - z_t, z^\star - z_{t+1} \rangle + \langle z_{t+1} - z_t,  z_{t+1}-z_{t+1/2}  \rangle \notag \\
        &= \frac{1}{2} \left( \| z_t - z^\star \|^2 - \| z_{t+1} -z^\star \|^2 - \| z_t - z_{t+1/2}\|^2 + \| z_{t+1} - z_{t+1/2}\|^2 \right).\label{eq: eqw2}
    \end{align}
    Using \eqref{eq: eqw1} and \eqref{eq: eqw2} on \eqref{eq: ghr4} after taking total expectation, using tower rule and dividing both sides by $\tau_t$ gives
    \begin{align}
        \left( \frac{1}{2\tau_t} + \frac{\mu}{2} \right) \mathbb{E} \| z^\star - z_{t+1} \|^2 \leq \frac{1}{2\tau_t}\mathbb{E} \| z^\star - z_{t} \|^2 + \frac{5}{6\tau_t } \mathbb{E}\|z_{t+1} - z_{t+1/2}\|^2 - \frac{1}{2\tau_t} \mathbb{E}\|z_{t} - z_{t+1/2}\|^2.\label{eq: sor5}
    \end{align}    
    Definition of $z_{t+1}$ in \Cref{alg:fbf_stoc} gives
    \begin{align*}
        \frac{5}{6}\|z_{t+1} - z_{t+1/2} \| &= \frac{5\tau_t^2}{6} \| \widetilde{B}(z_t) - \widetilde{B}(z_{t+1/2}) \|^2 \\
        &\leq \frac{5\tau_t^2}{2} \left( \|  \widetilde{B}(z_t) - B(z_t)\|^2 + \| B(z_{t}) - B(z_{t+1/2})\|^2+ \|  B(z_{t+1/2}) - \widetilde{B}(z_{t+1/2})\|^2  \right) \\
        &\leq 5\tau_t^2 \sigma^2 + \frac{5\tau_t^2 L_B^2}{2} \| z_{t} - z_{t+1/2}\|^2,
    \end{align*}    
    where the last line is by the variance bound assumed on $\widetilde{B}$ and Lipschitzness of $B$.
    
    With this, we get in place of \eqref{eq: sor5} that
    \begin{align}
        \left( \frac{1}{2\tau_t} + \frac{\mu}{2} \right) \mathbb{E} \| z^\star - z_{t+1} \|^2 &\leq \frac{1}{2\tau_t}\mathbb{E} \| z^\star - z_{t} \|^2 + \frac{1}{\tau_t} \left(\frac{5\tau_t^2L_B^2}{2} - \frac{1}{2} \right)\mathbb{E}\|z_{t} - z_{t+1/2}\|^2+5\tau_t \sigma^2.\label{eq: lab4}
    \end{align}
    The definition of $\tau_t = \frac{2}{(t+1)\mu+6L_B}$ has two consequences:
    \begin{equation*}
        \frac{1}{2\tau_t} + \frac{\mu}{2} = \frac{6L_B+(t+3)\mu}{4}
    \end{equation*}
    and
    \begin{equation*}
        \tau_t = \frac{2}{(t+1)\mu+6L_B} \leq \frac{1}{3L_B} \Longrightarrow \tau_t^2 \leq \frac{1}{5L_B^2} \iff 5\tau_t^2 L_B^2 \leq 1.
    \end{equation*}
    This last estimate shows that the second term on the right-hand side of \eqref{eq: lab4} is nonpositive.
    
    Then, we obtain, after multiplying both sides of \eqref{eq: lab4} by $\left(\frac{1}{2\tau_t} + \frac{\mu}{2} \right)^{-1}=\frac{4}{6L_B+(t+3)\mu}$ that
    \begin{equation}\label{eq: kig5}
        \mathbb{E}\| z^\star - z_{t+1} \|^2 \leq \left(\frac{(t+1)\mu+6L_B}{(t+3)\mu+6L_B}\right) \| z^\star - z_{t} \|^2 + \frac{40\sigma^2}{(6L_B+(t+1)\mu)(6L_B+(t+3)\mu)}.
    \end{equation}
    We next show by induction that
    \begin{equation*}
        \mathbb{E}\|z^\star - z_t \|^2 \leq \frac{6L_B/\mu\|z_0-z^\star\|^2 + 48\sigma^2/\mu^2}{t+6L_B/\mu} ~~~ \forall t \geq 0.
    \end{equation*}
    For brevity, let us denote $\kappa = 6L_B/\mu$.

    The base case $t=0$ holds by inspection. Next we assume the assertion holds for $t = T$ and consider \eqref{eq: kig5} to deduce
    \begin{align*}
        &\mathbb{E}\| z^\star - z_{T+1} \|^2\\ &\leq \frac{T+1+\kappa}{T+3+\kappa} \frac{\kappa\|z_0-z^\star\|^2 + 48\sigma^2/\mu^2}{T + \kappa} + \frac{40\sigma^2/\mu^2}{(T+1+\kappa)(T+3+\kappa)} \\
        &= \left(\frac{(T+1+\kappa)}{(T+3+\kappa)(T+\kappa)} + \frac{1}{1.2(T+1+\kappa)(T+3+\kappa)}\right) \left( \kappa\|z_0-z^\star\|^2 + 48\sigma^2/\mu^2 \right).
    \end{align*}
    As a result, the inductive step will be implied by
    \begin{align*}
        \left(\frac{(T+1+\kappa)^2}{(T+1+\kappa)(T+3+\kappa)(T+\kappa)} + \frac{(T+\kappa)}{1.2(T+1+\kappa)(T+3+\kappa)(T+\kappa)}\right) \leq \frac{1}{T+1+\kappa},
    \end{align*}
    which, after letting $\nu = T+\kappa$, is equivalent to
    \begin{align*}
        \left(\frac{1.2(\nu+1)^2}{(\nu+3)\nu} + \frac{\nu}{(\nu+3)\nu}\right) \leq 1.2 \iff 1.2(\nu+1)^2 \leq 1.2\nu^2 + 2.6\nu \iff 1.2 \leq 0.2\nu \iff 6\leq \nu.
    \end{align*}
    This holds because $\nu = T+\kappa = T+6L_B/\mu \geq 6$ since $L_B/\mu > 1$.
    This completes the induction.
\end{proof}

\subsection{Stochastic Problem with Cohypomonotonicity}\label{sec: app_stoc_cohypo}
We have a stochastic version of \Cref{lem: wlk3} proof of which is almost equivalent.
\begin{lemma}\label{lem: stoc_bddness}
    Let Assumptions \ref{asp:1} and \ref{asp:2} hold. For the sequence $(x_k)$ generated by Algorithm \ref{alg:cohypo_stoc_app} with $\mathbb{E}_k\| J_{\eta(F+G)}(x_k) - \widetilde{J}_{\eta(F+G)}(x_k)\|^2 \le \varepsilon_k^2$, we have for $k \geq 0$ that
    \begin{equation*}
        \mathbb{E}\|x_{k+1}-x^\star \| \leq \|x_0 - x^\star \| + \frac{\alpha}{k+2}\sum_{i=0}^k (i+1) \mathbb{E}[\varepsilon_i].
    \end{equation*}
    \end{lemma}
\begin{proof}
    The proof follows the same steps as \Cref{lem: wlk3} after taking expectation on \eqref{eq: hoi4} and using Jensen's inequality since
    \begin{align*}
    \mathbb{E}_k\left[\|\widetilde{J}_{\eta(F+G)}(x_k) -J_{\eta(F+G)}(x_k)\|\right]\leq \sqrt{\mathbb{E}_k\left[\|\widetilde{J}_{\eta(F+G)}(x_k) -J_{\eta(F+G)}(x_k)\|^2\right]}\leq\varepsilon_k.
    \end{align*}
    Hence the result follows by tower rule and the same induction as the proof of \Cref{lem: wlk3}.
\end{proof}
\begin{lemma}\label{eq: stoc_sfr4}
    Let Assumptions \ref{asp:1} and \ref{asp:2} hold. Consider Algorithm \ref{alg:cohypo_stoc_app} with $\mathbb{E}_k\|\widetilde{J}_{\eta(F+G)}(x_k) - J_{\eta(F+G)}(x_k) \|^2 \leq\varepsilon_k^2$. Then, we have for any $\gamma > 0$ and $K\geq 1$ that
    \begin{align*}
        \frac{\alpha K(K+1)}{4} \mathbb{E}\| R(x_K)\|^2 \leq \frac{K+1}{K\alpha} \| x^\star - x_0\|^2 + \sum_{k=0}^{K-1} \left(\frac{\alpha(\gamma+1)}{2\gamma}(k+1)(k+2)\mathbb{E}[\varepsilon_k^2]+\frac{\gamma\alpha\mathbb{E}\|R(x_k)\|^2}{2}\right).
    \end{align*}
\end{lemma}
\begin{proof}
    We follow the proof of \Cref{lem: cohypo_outer_mainlem} until \eqref{eq: eer87} and then we take expectation to obtain
    \begin{align}
    &\frac{\alpha}{2} \mathbb{E}\|R(x_{k+1})\|^2 + \frac{\beta_k}{1-\beta_k} \mathbb{E}\langle R(x_{k+1}), x_{k+1} - x_0 \rangle\notag \\
    &\leq \frac{\alpha}{2}\left( 1 - 2\beta_k\right)\mathbb{E}\|R(x_{k})\|^2 + \beta_k \mathbb{E}\langle R(x_{k}), x_{k} - x_0 \rangle\notag \\
    &\quad+\alpha \mathbb{E}\langle R(x_{k+1})-(1-\beta_k)R(x_k), ( \widetilde{J}_{\eta(F+G)} - J_{\eta(F+G)} )(x_k) \rangle - \frac{\alpha}{2} \mathbb{E}\| R(x_{k+1}) - R(x_k)\|^2.\label{eq: per87}
\end{align}
We then consider \eqref{eq: byu4} after taking expectation and using Cauchy-Schwarz, triangle and Young's inequalities to obtain
    \begin{align}
    &\alpha \mathbb{E}\langle R(x_{k+1}) - (1-\beta_k)R(x_k), (\widetilde{J}_{\eta(F+G)} - J_{\eta(F+G)})(x_k) \rangle \notag \\
    &\leq \alpha \mathbb{E}\left[\left( \|R(x_{k+1}) - R(x_k)\|+\beta_k\|R(x_k)\| \right)\|\widetilde{J}_{\eta(F+G)}(x_k) + J_{\eta(F+G)}(x_k)\|\right]\notag \\
    &\leq \frac{\alpha}{2}\mathbb{E}\|R(x_{k+1}) - R(x_k)\|^2 +\frac{\gamma\alpha\beta_k^2}{2}\mathbb{E}\|R(x_k)\|^2  +  \frac{\alpha}{2}\left(1+\frac{1}{\gamma}\right) \mathbb{E} \|\widetilde{J}_{\eta(F+G)}(x_k) - J_{\eta(F+G)}(x_k)\|^2 \notag \\
    &\leq \frac{\alpha}{2}\mathbb{E}\|R(x_{k+1}) - R(x_k)\|^2 +\frac{\gamma\alpha\beta_k^2}{2}\mathbb{E}\|R(x_k)\|^2  +  \frac{\alpha}{2}\left(1+\frac{1}{\gamma}\right) \mathbb{E}[\varepsilon_k^2],\label{eq: pjg4}
\end{align}
where the last step also used the tower rule along with the definition of $\varepsilon_k$.

We then use the same arguments as those after \eqref{eq: oer88} to get the result.
\end{proof}
\subsubsection{Proof for \Cref{cor: cohypo_stoc_main}}\label{subsubsec: app_halpern_stoc}
\Cref{cor: cohypo_stoc_main} is essentially the summary of the results proven below. 
\begin{theorem}\label{th: cohypo_sto}
    Let Assumptions \ref{asp:1}, \ref{asp:2}, and \ref{asp:4} hold. Let $\eta < \frac{1}{L}$ in Algorithm \ref{alg:cohypo_stoc_app} and $\rho < \eta$. 
    For any $k \geq 1$, we have that
    \begin{equation*}
        \frac{1}{\eta^2}\mathbb{E}\|x_k - J_{\eta(F+G)}(x_k)\|^2 \leq \frac{36(\|x_0-x^\star\|^2+\sigma^2)}{(\eta-\rho)^2k^2}.
    \end{equation*}
    The number of first-order oracles used at iteration $k$ of Algorithm \ref{alg:cohypo_stoc_app} is upper bounded by
    \begin{equation}\label{eq: stoc_num_fbf}
        2\left\lceil \frac{1734(k+2)^3\log^2(k+2)}{(1-\eta L)^2}\right\rceil.
    \end{equation}
\end{theorem}
\begin{corollary}\label{cor: cohypo_stoc}
    Let Assumptions \ref{asp:1} and \ref{asp:2} hold. Let $\eta < \frac{1}{L}$ in Algorithm \ref{alg:cohypo_stoc_app} and $\rho < \eta$. For any $\varepsilon >0$, we have that $\mathbb{E}\left[\eta^{-1}\|x_k - J_{\eta(F+G)}(x_k)\|\right] \leq \varepsilon$ with stochastic first-order oracle complexity
    \begin{equation*}
        \widetilde{O}\left( \frac{\|x_0-x^\star\|^4+\sigma^4}{(\eta-\rho)^4(1-\eta L)^2\varepsilon^4} \right)
    \end{equation*}
\end{corollary}
\begin{proof}
    This corollary immediately follows from \Cref{th: cohypo_sto} by combining the number of outer iterations and the number of stochastic first-order oracle calls for each outer iteration.
\end{proof}
\begin{remark}
    The complexity in the previous corollary has the same dependence on $\|x_0-x^\star\|, \sigma$ as \cite{pethick2023solving,bravo2024stochastic}. As we see in the next remark, the dependence on $(\eta-\rho)$ can be improved by using the knowledge of the target accuracy $\varepsilon$ and the variance upper bound $\sigma^2$ as done in \cite{diakonikolas2021efficient,lee2021fast,chen2022near}.
\end{remark}
\begin{remark}
By using parameters depending on target accuracy $\varepsilon$ and noise variance $\sigma^2$, we can improve the complexity to
\begin{equation*}
        \widetilde{O}\left( \frac{\|x_0-x^\star\|^2\sigma^2}{(\eta-\rho)^2(1-\eta L)^2\varepsilon^4} \right)
    \end{equation*}
\end{remark}
\begin{proof}[Proof of \Cref{th: cohypo_sto}]
    Let us set 
\begin{equation}\label{eq: smk4}
    \varepsilon_k^2 = \frac{\gamma^2 (\alpha^2\|R(x_k)\|^2+8\sigma^2)}{\alpha^2(k+2)^{3}\log^2(k+2)}
\end{equation}
and plug this in to the result of Lemma \ref{eq: stoc_sfr4} to obtain
\begin{align}
    &\frac{\alpha K(K+1)}{4} \mathbb{E} \| R(x_K)\|^2 \notag\\
    &\leq \frac{K+1}{K\alpha} \| x_0 - x^\star \|^2 + \sum_{k=0}^{K-1} \mathbb{E}\left( \frac{(\gamma^2 +\gamma)(\alpha^2\|R(x_k)\|^2+8\sigma^2)}{2\alpha (k+2)\log^2(k+2)} +  \frac{\gamma\alpha\|R(x_k)\|^2}{2}\right)\notag \\
    &= \frac{K+1}{K\alpha} \| x_0 - x^\star \|^2 + \sum_{k=0}^{K-1} \left( \frac{4(\gamma^2+\gamma) \sigma^2}{\alpha (k+2)\log^2(k+2)} + \frac{\alpha(\gamma^2+\gamma)\mathbb{E}\|R(x_k)\|^2}{2(k+2)\log^2(k+2)} + \frac{\gamma\alpha\mathbb{E}\|R(x_k)\|^2}{2}\right)\notag\\
    &< \frac{K+1}{K\alpha} \| x_0 - x^\star \|^2 +\frac{12(\gamma^2+\gamma)\sigma^2}{\alpha}+ \sum_{k=0}^{K-1} \left( \frac{\alpha(\gamma^2+\gamma)\mathbb{E}\|R(x_k)\|^2}{2(k+2)\log^2(k+2)} + \frac{\gamma\alpha\mathbb{E}\|R(x_k)\|^2}{2}\right),\label{eq: hnb5}
\end{align}
since $\sum_{k=0}^{K-1} \frac{1}{(k+2)\log^2(k+2)}<3$.

We now show by induction that 
\begin{equation*}
    \mathbb{E}\|R(x_k)\|^2 \leq \frac{36(\|x_0-x^\star\|^2+\sigma^2)}{\alpha^2(k+1)^2}.
\end{equation*}
The base case for the induction with $K=0, 1$ hold the same way as the proof of \Cref{th: cohypo_det} where the only change is we use \Cref{lem: stoc_bddness} and the definition of $\varepsilon_k$ in \eqref{eq: smk4}, see also \eqref{eq: base_case_halpern}.

Let us consider \eqref{eq: hnb5} for $K\geq 2$ and assume that the assertion holds for $k \leq K-1$. We then have that
\begin{align*}
    &\frac{\alpha K(K+1)}{4} \mathbb{E}\|R(x_K)\|^2 \\
    &\leq \frac{2}{\alpha} \| x_0 - x^\star \|^2 +\frac{12(\gamma^2+\gamma)\sigma^2}{\alpha}+ \sum_{k=0}^{K-1} \left( \frac{18(\gamma^2+\gamma)(\|x_0-x^\star\|^2+\sigma^2)}{\alpha(k+2)(k+1)^2\log^2(k+2)} + \frac{18\gamma(\|x_0-x^\star\|^2+\sigma^2)}{\alpha(k+1)^2}\right),
\end{align*}
where we also used $\frac{K+1}{K}\leq 2$.

By using $\sum_{k=0}^\infty \frac{18}{(k+1)^2} < 30$ and $\sum_{k=0}^\infty \frac{18}{(k+2)(k+1)^2\log^2(k+2)} < 21$ and $\gamma = \frac{1}{17}$, we have that
\begin{align*}
    \frac{\alpha K(K+1)}{4}\mathbb{E}\|R(x_K)\|^2 \leq \frac{6(\|x_0-x^\star\|^2 + \sigma^2)}{\alpha}.
\end{align*}
We use $\frac{1}{K(K+1)}\leq \frac{1.5}{(K+1)^2}$ which holds for $K\geq 2$ to complete the induction.

To see the number of first-order oracles, we use the result for stochastic FBF in \Cref{th: stoc_fbf}. For our subproblem at iteration $k$, this result implies
\begin{align*}
    \mathbb{E}_k\left[ \| \widetilde{J}_{\eta(F+G)}(x_k) - J_{\eta(F+G)}(x_k)\|^2\right] &\leq \frac{6\left( \frac{1+\eta L}{1-\eta L}\|x_k - J_{\eta(F+G)}(x_k)\|^2 + \frac{8\sigma^2}{(1-\eta L)^2} \right)}{N_k} \\
    &\leq \frac{\frac{6}{(1-\eta L)^2}\left( \|x_k - J_{\eta(F+G)}(x_k)\|^2 + 8\sigma^2 \right)}{N_k}.
\end{align*}
Recall that \eqref{eq: smk4}, with $\gamma = \frac{1}{17}$ and $R=\id-J_{\eta(F+G)}$, requires
\begin{equation*}
    \mathbb{E}_k\left[ \| \widetilde{J}_{\eta(F+G)}(x_k) - J_{\eta(F+G)}(x_k)\|^2\right] \leq \frac{ (\alpha^2\|(\id-J_{\eta(F+G)})(x_k)\|^2+8\sigma^2)}{289 \alpha^2(k+2)^{3}\log^2(k+2)}
\end{equation*}
Noting that $\frac{1}{\alpha^2} > 1$, a sufficient condition to attain this requirement is
\begin{equation*}
\frac{\frac{6}{(1-\eta L)^2}\left( \|x_k - J_{\eta(F+G)}(x_k)\|^2 + 8\sigma^2 \right)}{N_k} \leq \frac{ \|x_k - J_{\eta(F+G)}(x_k)\|^2+8\sigma^2}{289 (k+2)^{3}\log^2(k+2)},
\end{equation*}
verifying the required number of iterations $N_k$ as given in \Cref{alg:cohypo_stoc_app} to be sufficient for the inexactness criterion. Since each iteration of FBF takes 2 stochastic operator evaluations $\widetilde{F}$ and one resolvent of $G$, we have the result.
\end{proof}

\subsection{Stochastic Problem with weak MVI condition}\label{sec: app_stoc_wmi}
As motivated in \Cref{sec: main_stoc_wmi}, we use the multilevel Monte Carlo (MLMC) estimator \citep{giles2008multilevel,blanchet2015unbiased,asi2021stochastic,hu2021bias}. In \Cref{sec: main_stoc_wmi}, we only sketched the main changes in \Cref{alg:weakmvi} because of space limitations. We start by explicitly writing down the algorithm with MLMC estimator.

\begin{algorithm*}[h]
\caption{Inexact KM iteration for problems with weak MVI}
\begin{algorithmic}
    \STATE {\bfseries Input:} Parameters $\eta, L, \rho$, $\alpha=1-\frac{\rho}{\eta}$, $\alpha_k = \frac{\alpha}{\sqrt{k+2}\log(k+3)}$ $K>0$, initial iterate $x_0\in\mathbb{R}^d$, subroutine $\texttt{MLMC{-}FBF}$ given in Algorithm \ref{alg:fbf_mlmc} \\
    \vspace{.2cm}
    \FOR{$k = 0, 1, 2,\ldots, K-1 $}
        \STATE $N_k = \lceil \frac{96(1-\eta L)^{-2}}{\min\{\frac{\alpha_k}{120\alpha(k+1)}, \frac{1}{120}\}} \rceil$ and $M_k =\lceil \frac{672\times 120(\log_2 N_k)}{(1-\eta L)^2} \rceil$
        \STATE $\widetilde{J}^{(m)}_{\eta(F+G)}(x_k) = \texttt{MLMC{-}FBF}\left(x_k, N_k, \eta G, \id + \eta \widetilde{F},1+\eta L\right)$ independently for each $m=1,\dots, M_k$
        \STATE $\widetilde{J}_{\eta(F+G)}(x_k) = \frac{1}{M_k}\sum_{i=1}^{M_k} \widetilde{J}_{\eta(F+G)}^{(i)}(x_k)$
        \STATE $x_{k+1} = (1-\alpha_k)x_k + \alpha_k \widetilde{J}_{\eta(F+G)}(x_k)$
        \ENDFOR
      \end{algorithmic}
\label{alg:weakmvi_stoc}
\end{algorithm*}

\begin{algorithm}[h]
\caption{$\texttt{MLMC{-}FBF}(z_0, N, A, B, L_B)$}
\begin{algorithmic}
    \STATE {\bfseries Input:} Initial iterate $z_0\in\mathbb{R}^d$, subsolver $\texttt{FBF}$ from \Cref{alg:fbf_stoc} \\
    \vspace{.2cm}
    \STATE Define $y^i = \texttt{FBF}(z_0, 2^i, \widetilde{B}, A, L_B)$ for any $i\geq 0$. Draw $I\sim \mathrm{Geom}(1/2)$\\
    \vspace{.2cm}
    \STATE \textbf{Output:} $y^{\out} = y^0 + 2^I(y^I - y^{I-1})$ if $2^I \leq N$, otherwise $y^{\text{out}} = y^0$.
      \end{algorithmic}
\label{alg:fbf_mlmc}
\end{algorithm}

We start by modifying the proof of \Cref{lem: rate_outer} for the stochastic problem, which is the most important for getting the final complexity.

\begin{lemma}\label{lem: stoc_mvi_iter_compl}
Let Assumptions \ref{asp:1} and \ref{asp:3} hold. Suppose that the iterates generated by \Cref{alg:weakmvi_stoc} satisfy $\mathbb{E}_k\|\widetilde{J}_{\eta(F+G)}(x_k) - J_{\eta(F+G)}(x_k) \|^2 \leq\varepsilon_{k,v}^2$ and $\|\mathbb{E}_k[\widetilde{J}_{\eta(F+G)}(x_k)] - J_{\eta(F+G)}(x_k)\| \leq \varepsilon_{k, b}$.
Then, we have that
\begin{align*}
\frac{\alpha}{4}\sum_{k=0}^{K-1}\alpha_k \mathbb{E}\|(\id - J_{\eta(F+G)})(x_k)\|^2  \leq \frac{1}{2}\|x_0 - x^\star\|^2+ \frac{3}{2}\sum_{k=0}^{K-1} \alpha_k^2 \mathbb{E}[\varepsilon_{k, v}^2] + \sum_{k=0}^{K-1} \alpha_k \mathbb{E}[\|x_k-x^\star\|\varepsilon_{k, b}].
\end{align*}
\end{lemma}
\begin{proof}
We proceed mostly as the proof of \Cref{lem: rate_outer} apart from minor changes due to the stochastic setting such as iteration-dependent step sizes. 

From \Cref{fact: weakmvi}(ii), we have that $\id - J_{\eta(F+G)}$ is $\left(1-\frac{\rho}{\eta}\right)$-star cocoercive. Recall our running notations:
\begin{equation*}
\alpha=1-\frac{\rho}{\eta}, ~~~R = \id - J_{\eta(F+G)},~~~\widetilde{R}= \id - \widetilde{J}_{\eta(F+G)}.
\end{equation*} 
As a result, we have the following equivalent representation of $x_{k+1}$ (see the definition in Algorithm \ref{alg:weakmvi_stoc}):
\begin{align}
    x_{k+1}&= x_k - \alpha_k \widetilde{R}(x_k).\label{eq: cls4_stoc}
\end{align}
    By $\alpha$-star-cocoercivity of $R = \id - J_{\eta(F+G)}$, we have
    \begin{align}\label{eq: rml4_stoc}
        \langle R(x_k), x_{k}-x^\star \rangle \geq \alpha \|R(x_k)\|^2.
    \end{align}
    By a simple decomposition, we write
    \begin{align}
        \langle R(x_k), x_k - x^\star \rangle &= \langle \widetilde{R}(x_k), x_k - x^\star \rangle + \langle R(x_k) - \widetilde{R}(x_k), x_k - x^\star \rangle. \label{eq: vnb4_stoc}
    \end{align}
    For the expectation of the first term on the right-hand side of \eqref{eq: vnb4_stoc}, we derive
    that (\emph{cf.} \eqref{eq: vbn54})
    \begin{align}
        \mathbb{E}\langle \widetilde{R}(x_k), x_k - x^\star \rangle &=
        \frac{1}{\alpha_k}\mathbb{E}\langle x_k -x_{k+1}, x_k - x^\star \rangle \notag \\
        &= \frac{1}{2\alpha_k} \mathbb{E}\left( \|x_k - x_{k+1}\|^2 + \| x_k- x^\star\|^2 - \| x_{k+1} - x^\star\|^2 \right) \notag\\
        &\leq \frac{1}{2\alpha_k}\mathbb{E}\left( \|x_k - x^\star\|^2 -\|x_{k+1}- x^\star\|^2 \right) + \frac{3\alpha_k}{4}\mathbb{E}\|R(x_k)\|^2 + \frac{3\alpha_k}{2} \mathbb{E}\| \widetilde{R}(x_k) - R(x_k)\|^2 \notag\\
        &\leq \frac{1}{2\alpha_k}\mathbb{E}\left( \|x_k - x^\star\|^2 -\|x_{k+1}- x^\star\|^2 \right) + \frac{3\alpha}{4}\mathbb{E}\|R(x_k)\|^2 + \frac{3\alpha_k\mathbb{E}[\varepsilon_{k, v}^2]}{2},\label{eq: vbn54_stoc}
    \end{align}
    where we used the definition of $x_{k+1}$ from  \eqref{eq: cls4_stoc} in the first step, standard expansion $\|a-b\|^2 = \|a\|^2-2\langle a,b\rangle + \|b\|^2$ for the second step, the definition of $x_{k+1}$ from  \eqref{eq: cls4_stoc} and Young's inequality in the third step, the definition of $\varepsilon_{k, v}$ with tower rule and $\alpha_k \leq \alpha$ in the last step.
    
    For the second term on the right-hand side of \eqref{eq: vnb4_stoc}, we have, by Cauchy-Schwarz inequality and the definition of $\widetilde{R}$ and $\varepsilon_{k, b}$, that
    \begin{align}
        \mathbb{E}\langle R(x_k) - \widetilde{R}(x_k), x_k - x^\star \rangle &= \mathbb{E}[\mathbb{E}_k\langle R(x_k) - \widetilde{R}(x_k), x_k - x^\star \rangle] \notag \\
        &= \mathbb{E}\langle \mathbb{E}_k[R(x_k) - \widetilde{R}(x_k)], x_k - x^\star \rangle \notag \\
        &\leq \mathbb{E}\left[\|R(x_k) - \mathbb{E}_k[\widetilde{R}(x_k)]\| \| x_k - x^\star\|\right] \notag \\
        &\leq \mathbb{E}\left[\|x_k - x^\star\| \varepsilon_{k, b}\right],\label{eq: sdso4_stoc}
    \end{align}
    where the first step is by tower rule and the second step is by $x_k - x^\star$ being measurable under the conditioning of $\mathbb{E}_k$.
    
    We combine \eqref{eq: vbn54_stoc} and \eqref{eq: sdso4_stoc} in \eqref{eq: vnb4_stoc}, plug in the result to \eqref{eq: rml4_stoc} and rearrange to obtain
    \begin{align*}
        \frac{\alpha}{4}\mathbb{E}\|R(x_k)\|^2 \leq \frac{1}{2\alpha_k}\mathbb{E}\left(\|x_k - x^\star\|^2 - \| x_{k+1} - x^\star\|^2 \right) + \frac{3\alpha_k\mathbb{E}[\varepsilon_{k, v}^2]}{2}+\mathbb{E}[\|x_k-x^\star\|\varepsilon_{k, b}].
    \end{align*}
    We conclude after multiplying both sides by $\alpha_k$ and summing for $k=0, 1, \dots, K-1$.
\end{proof}
The next lemma considers the bias and variance of the MLMC estimator and follows the same arguments as \citep[Proposition 1]{asi2021stochastic}. The only change is that we use the algorithm FBF (see \Cref{alg:fbf_stoc} for a stochastic version) as the subsolver and we consider a strongly monotone inclusion problem rather than minimization. These do not alter the estimations significantly as can be seen in the proof.
\begin{lemma}\label{lem: mlmc_fbf}
    Under the same setting as \Cref{th: stoc_fbf} and $N\geq 2$, for the output of \Cref{alg:fbf_mlmc}, we have that
    \begin{align*}
        \| \mathbb{E}[y^{\out}] - z^\star \|^2 &\leq \frac{12L_B/\mu\|z_0 - z^\star\|^2 + 96\sigma^2/\mu^2}{N} \\
        \mathbb{E}\| y^{\out} - z^\star \|^2 &\leq 14(6L_B/\mu\|z_0-z^\star\|^2 + 48\sigma^2/\mu^2) \log_2 N
    \end{align*}
    where the expected number of calls to $\widetilde{F}$ is $O(\log_2 N)$.
\end{lemma}
\begin{proof}
We argue as \citep[Property 1]{asi2021stochastic}. The only difference is that we call \Cref{th: stoc_fbf} which is our main solver for the strongly monotone problem.

    Let us denote $i_N = \max\{ i \geq 0\colon 2^i \leq N \}$. For a given event $E$, consider also the following notation for the characteristic function: $\mathbf{1}_{E} = 1$ if $E$ is true and $\mathbf{1}_E = 0$ if $E$ is false. 

    Then, we have by the definition of $y^{\out}$ in \Cref{alg:fbf_mlmc} that
    \begin{align}
        \mathbb{E}[y^{\out}] &= \mathbb{E}[y^0] + \mathbb{E}[\mathbf{1}_{\{2^I \leq N\}}\cdot 2^I(y^{I} - y^{I-1})] \notag \\
        &= \mathbb{E}[y^0] + \sum_{i=1}^{i_N} \Pr(I=i) 2^i \mathbb{E}[y^{i} - y^{i-1}] \notag\\
        &= \mathbb{E}[y^0] + \mathbb{E}[y^{i_N} - y^{0}]\notag \\
        &= \mathbb{E}[y^{i_N}].\label{eq: hpy5}
    \end{align}
    By the definition of $i_N$, we have that $2^{i_N} \geq \frac{N}{2}$ and hence, by Jensen's inequality and \Cref{th: stoc_fbf}, we have
    \begin{align*}
        \|\mathbb{E}[y^{i_N}] - z^\star \|^2 &\leq \mathbb{E}\|y^{i_N} - z^\star \|^2 \\
        &\leq \frac{12L_B\|z_0- z^\star\|^2 + 96\sigma^2/\mu}{N\mu},
    \end{align*}
    which is the claimed bound on the bias due to \eqref{eq: hpy5}. 
    
    We continue with estimating the variance of $y^{\out}$. First, Young's inequality gives that
    \begin{align}\label{eq: mhj4}
        \mathbb{E}\|y^{\out} - z^\star \|^2 \leq 2\mathbb{E} \| y^{\out} - y^0 \|^2 + 2\mathbb{E}\| y^0 - z^\star \|^2.
    \end{align}
    We estimate the first term on the right-hand side: 
    \begin{align}
        \mathbb{E} \| y^\out - y^0\|^2 &= \sum_{i=1}^{i_N} \Pr(I=i) \mathbb{E}\| 2^i(y^i - y^{i-1})\|^2 \notag \\
        &= \sum_{i=1}^{i_N} 2^i \mathbb{E}\| y^i - y^{i-1}\|^2 \notag \\
        &\leq \sum_{i=1}^{i_N} 2^{i+1}\left( \mathbb{E}\| y^i - z^\star\|^2 + \mathbb{E} \| y^{i-1} - z^\star\|^2 \right),\label{eq: mhj5}
    \end{align}
    where the last step is by Young's inequality.
    
    By the definitions of $y^i, y^{i-1}$ and \Cref{th: stoc_fbf}, we have that
    \begin{align*}
        \mathbb{E} \| y^i - z^\star \|^2 \leq \frac{6L_B\|z_0 - z^\star\|^2+48\sigma^2/\mu}{2^i \mu}, \\
        \mathbb{E} \| y^{i-1} - z^\star \|^2 \leq \frac{6L_B\|z_0 - z^\star\|^2+48\sigma^2/\mu}{2^{i-1} \mu}.
    \end{align*}
    This gives, in view of \eqref{eq: mhj5}, that
    \begin{equation*}
        \mathbb{E} \| y^\out - y^0 \|^2 \leq \frac{6(6L_B\|z_0-z^\star\|^2 + 48\sigma^2/\mu)}{\mu} i_N.
    \end{equation*}
    The second term on the right-hand side of \eqref{eq: mhj4} is estimated the same way by using \Cref{th: stoc_fbf}:
    \begin{align*}
        \mathbb{E}\|y^0-z^\star\|^2\leq \frac{6L_B\|z_0 - z^\star\|^2+48\sigma^2/\mu}{\mu}.
    \end{align*}
    Combining the last two estimates in \eqref{eq: mhj4} gives the claimed bound on the variance after using $i_N \leq \log_2 N$.
    
    The expected number of calls to $\widetilde{B}$ is calculated as
    \begin{equation*}
        2+2\sum_{i=1}^{i_N} P(I=i)(2^i + 2^{i-1}) = O(1+i_N) = O(1+\log_2 N),
    \end{equation*}
    since each iteration of stochastic FBF uses $2$ unbiased samples of $F$. This completes the proof.
\end{proof}
In fact, \Cref{alg:weakmvi_stoc} computes independent draws of $\texttt{MLMC{-}FBF}$ and averages them to get a better control on the variance as \citep[Theorem 1]{asi2021stochastic}. 
\begin{corollary}\label{cor: mlmc_bounds_avg}
    Let $\widetilde{J}_{\eta(F+G)}(x_k)$ be defined as in \Cref{alg:weakmvi_stoc} and consider the setting of \Cref{th: stoc_fbf}. Then, for any $b_k, v$, we have the bias and variance bounds given as
    \begin{align*}
        \| \mathbb{E}_k[\widetilde{J}_{\eta(F+G)}(x_k)]- J_{\eta(F+G)}(x_k) \|^2 &\leq b_k^2(\|(\id+J_{\eta(F+G)})(x_k)\|^2+\sigma^2), \\
        \mathbb{E}_k\| \widetilde{J}_{\eta(F+G)}(x_k)- J_{\eta(F+G)}(x_k) \|^2 &\leq v^2(\|(\id+J_{\eta(F+G)})(x_k)\|^2+\sigma^2),
    \end{align*}
    where
    \begin{equation*}
        N_k = \left\lceil \frac{\max\{12L_B/\mu, 96/\mu^2\}}{\min\{b_k^2, \frac{v^2}{2}\}} \right\rceil, \text{~~~and~~~} M_k =\left\lceil \frac{2\log_2 N_k \max\{84L_B/\mu, 672/\mu^2\}}{v^2} \right\rceil.
    \end{equation*}
    Each iteration makes in expectation $O(\log N_k\cdot M_k)$ calls to stochastic first-order oracle.
\end{corollary}
\begin{proof}
This proof follows the arguments in \citep[Theorem 1]{asi2021stochastic}. The difference is that we set the values of $N_k, M_k$ independent of $\|R(x_k)\|^2$ and $\sigma^2$, to make $N_k, M_k$ computable, which results in these terms appearing in the bias and variance upper bounds.

    We first note that $\mathbb{E}_k[\widetilde J_{\eta(F+G)}(x_k)] = \mathbb{E}_k[\widetilde J_{\eta(F+G)}^{(1)}(x_k)]$. We next have by direct expansion that
    \begin{align*}
        \mathbb{E}_k\| \widetilde{J}_{\eta(F+G)}(x_k) - J_{\eta(F+G)}(x_k) \|^2 &= \frac{1}{M_k} \mathbb{E}_k\|\widetilde{J}_{\eta(F+G)}^{(1)}(x_k) - J_{\eta(F+G)}(x_k)\|^2 \\
        &\quad+ \left( 1- \frac{1}{M_k}\right) \| \mathbb{E}_k[\widetilde{J}_{\eta(F+G)}^{(1)}(x_k)] - J_{\eta(F+G)}(x_k)\|^2,
    \end{align*}
    since $\widetilde{J}_{\eta(F+G)}^{(i)}$ are independent draws of the same estimator.
    
    By applying the identity $\mathbb{E}\|X\|^2=\mathbb{E}\|X-\mathbb{E}X\|^2 + \|\mathbb{E}X\|^2$ with $X = \widetilde{J}_{\eta(F+G)}^{(1)}(x_k) - J_{\eta(F+G)}(x_k)$, we obtain
    \begin{align}
        \mathbb{E}_k\| \widetilde{J}_{\eta(F+G)}(x_k) - J_{\eta(F+G)}(x_k) \|^2 &= \frac{1}{M_k} \mathbb{E}_k\|\widetilde{J}_{\eta(F+G)}^{(1)}(x_k) - \mathbb{E}_k[\widetilde{J}_{\eta(F+G)}^{(1)}(x_k)]\|^2\notag \\
        &\quad + \| \mathbb{E}_k[\widetilde{J}_{\eta(F+G)}^{(1)}(x_k)] - J_{\eta(F+G)}(x_k)\|^2.\label{eq: vft3}
    \end{align}
    On the one hand, the fact $\mathbb{E}\|X-\mathbb{E}X\|^2\leq \mathbb{E}\|X\|^2$ gives
    \begin{equation}\label{eq: vft4}
        \mathbb{E}_k\|\widetilde{J}_{\eta(F+G)}^{(1)}(x_k) - \mathbb{E}_k[\widetilde{J}_{\eta(F+G)}^{(1)}(x_k)] \|^2 \leq \mathbb{E}_k\|\widetilde{J}_{\eta(F+G)}^{(1)}(x_k) - J_{\eta(F+G)}(x_k)\|^2.
    \end{equation}
    On the other hand, the bounds in \Cref{lem: mlmc_fbf} gives, after substituting $z_0 = x_k$ and $z^\star = J_{\eta(F+G)}(x_k)$ that
    \begin{subequations}
    \begin{align}
        \|\mathbb{E}_k[\widetilde{J}_{\eta(F+G)}^{(1)}(x_k)] - J_{\eta(F+G)}(x_k)\|^2&\leq \frac{12L_B/\mu\|(\id+J_{\eta(F+G)})(x_k)\|^2 + 96\sigma^2/\mu^2}{N_k}, \label{eq: llo1}\\
        \mathbb{E}_k\|\widetilde{J}_{\eta(F+G)}^{(1)}(x_k) - J_{\eta(F+G)}(x_k)\|^2 &\leq \left(84L_B/\mu\|(\id+J_{\eta(F+G)})(x_k)\|^2 + 672\sigma^2/\mu^2\right)\log_2 N_k.\label{eq: llo2}
    \end{align}
    \end{subequations}
    Using $\mathbb{E}_k[\widetilde J_{\eta(F+G)}(x_k)] = \mathbb{E}_k[\widetilde J_{\eta(F+G)}^{(1)}(x_k)]$ gives the bias bound after using the definition of $N_k$ and \eqref{eq: llo1}
    
    Plugging in \eqref{eq: llo2} and \eqref{eq: vft4} in \eqref{eq: vft3} gives the variance bound after substituting the values of $N_k$ and $M_k$.
\end{proof}

\subsubsection{Proof for \Cref{cor: weakmvi_stoc_main}}\label{subsubsec: mlmc_main_app}
\Cref{cor: weakmvi_stoc_main} is essentially the summary of the results proven below. 

Let us remark the recent work \citep[Corollary 5.4]{bravo2024stochastic} that studied stochastic KM iteration for nonexpansive operators on normed spaces. This work assumes access to an unbiased oracle of the nonexpansive operator at hand and get the complexity $\widetilde{O}(\varepsilon^{-4})$. As mentioned in \Cref{sec: main_stoc_wmi}, this corresponds to requiring unbiased samples of $J_{\eta(F+G)}$ in our setting, which is difficult due to the definition of the resolvent. We get the same complexity up to logarithmic factors without access to unbiased samples of $J_{\eta(F+G)}$, which we go around by using the MLMC technique. We also do not require nonexpansiveness from $J_{\eta(F+G)}$ and work with conic quasi-nonexpansiveness. 
\begin{theorem}\label{th: weakmvi_stoc}
    Let Assumptions \ref{asp:1}, \ref{asp:3}, and \ref{asp:4} hold. Consider \Cref{alg:weakmvi_stoc} with $\eta < \frac{1}{L}$ and $\rho < \eta$. Then, we have for $K\geq 1$ that 
    \begin{equation*}
        \mathbb{E}_{x^{\out}\sim\mathrm{Unif}\{x_0, \dots, x_{K-1}\}}[\mathbb{E} \|(\id-J_{\eta(F+G)})(x^\out) \|^2] \leq \frac{64(\|x_0-x^\star\|^2+ \alpha^2\sigma^2)\log(K+3)}{\alpha^2 \sqrt{K}},
    \end{equation*}
    where $\alpha= 1-\frac{\rho}{\eta}$. Each iteration makes, in expectation, $O(\log^2(k+2))$ calls to stochastic oracle $\widetilde{B}$ and resolvent of $A$. Hence to obtain $\mathbb{E}\|(\id-J_{\eta(F+G)})(x^\out)\|\leq\varepsilon$, we have the expected stochastic first-order complexity $\widetilde{O}(\varepsilon^{-4})$.
\end{theorem}
The main reason for the length of the following proof is the lack of boundedness of $(x_k)$. In particular, proving this theorem is rather straightforward when we assume a bounded domain. We have to handle the complications without this assumption.
There are also additional difficulties that arise because we are making sure that the inputs to $\texttt{MLMC{-}FBF}$ will not involve unknown quantities such as $\|x_0 -x^\star\|$ or $\sigma$ to run the algorithm. These are, for example, used in \cite{chen2022near} for setting the parameters. Because of this reason, the bounds for $\varepsilon_{k, v}$ and $\varepsilon_{k, b}$ involve $\|(\id+J_{\eta(F+G)})(x_k)\|$ and $\sigma^2$. 

The main reason for the difficulty here is $\|(\id+J_{\eta(F+G)})(x_k)\|^2$, since we do not have a uniform bound on this quantity, unlike $\sigma^2$ and this term appears in many summands. We will carry these terms coming from the MLMC bounds to get a recursion involving the sum of $\|(\id+J_{\eta(F+G)})(x_k)\|^2$ for different ranges on both sides. We then go around the issue of lacking of a bound on $(x_k)$ by using an inductive argument on $\sum_{k=0}^K \|(\id+J_{\eta(F+G)})(x_k)\|^2$.
\begin{proof}[Proof of \Cref{th: weakmvi_stoc}]
    Recall our running notations:
\begin{equation*}
\alpha=1-\frac{\rho}{\eta}, ~~~R = \id - J_{\eta(F+G)},~~~\widetilde{R}= \id - \widetilde{J}_{\eta(F+G)}.
\end{equation*} 
    We start by following the proof of \Cref{lem: wlk4}. By $\alpha$-star-cocoercivity of $\id-J_{\eta(F+G)}$ and $\alpha \geq \alpha_k$ (which gives that $J_{\eta(F+G)}$ is $\frac{1}{\alpha_k}$-star-conic nonexpansive), we can use property \eqref{eq:star_nonexpansive} derived in Lemma \ref{lem:conic-star} to obtain
    \begin{equation*}
        \left\| (1-\alpha_k) x_k + \alpha_k J_{\eta(F+G)}(x_k) - x^\star \right\|\leq \| x_k - x^\star\|
    \end{equation*}
    and by the definition of $x_{k+1}$, we get
    \begin{align}
        \|x_{k+1} - x^\star\| &\leq  \|(1-\alpha_k) x_k + \alpha_k J_{\eta(F+G)}(x_k) - x^\star \| + \alpha_k \|\widetilde{J}_{\eta(F+G)}(x_k) - J_{\eta(F+G)}(x_k) \| \notag \\
        &\leq \|x_k - x^\star \| + \alpha_k \|\widetilde{J}_{\eta(F+G)}(x_k) - J_{\eta(F+G)}(x_k) \|.\label{eq: pkl4}
    \end{align}
    Summing the inequality for $0, \dots, k-1$ gives
    \begin{align}
        &\|x_k - x^\star\| \leq \| x_0 - x^\star \| + \sum_{i=0}^{k-1} \alpha_i \|\widetilde{J}_{\eta(F+G)}(x_i) - J_{\eta(F+G)}(x_i)\| \notag \\
        \Longrightarrow~&\mathbb{E}\|x_k - x^\star\|^2 \leq 2\mathbb{E}\| x_0 - x^\star \|^2 + 2k\sum_{i=0}^{k-1} \alpha_i^2 \mathbb{E}\|\widetilde{J}_{\eta(F+G)}(x_i) - J_{\eta(F+G)}(x_i)\|^2,\label{eq: lab1}
    \end{align}
    where we first squared both sides, used Young's inequality and then took expectation.
    
    We continue by restating the result of \Cref{lem: stoc_mvi_iter_compl} after applying Young's inequality on the last term to obtain
    \begin{align}
\frac{\alpha}{4}\sum_{k=0}^{K-1}\alpha_k \mathbb{E}\|(\id - J_{\eta(F+G)})(x_k)\|^2  &\leq \frac{1}{2}\|x_0 - x^\star\|^2+ \frac{3}{2}\sum_{k=0}^{K-1} \alpha_k^2 \mathbb{E}[\varepsilon_{k, v}^2] \notag \\
&\quad+ \sum_{k=0}^{K-1} \left(\frac{\alpha_k^2}{2\alpha^2(k+1)} \mathbb{E}\|x_k-x^\star\|^2 + \frac{(k+1)\alpha^2}{2}\mathbb{E}[\varepsilon_{k, b}^2]\right).\label{eq: nhy2}
\end{align}
We now estimate the second and third terms on the right-hand side.
By using \Cref{cor: mlmc_bounds_avg} and the definition of $R(x_k)$, $\alpha_k = \frac{\alpha}{\sqrt{k+2}\log(k+3)}\leq \frac{\alpha}{\sqrt{2}\log 3}$ and using $v^2 =\frac{1}{60}\leq \frac{\sqrt{2}\log 3}{24}$, we obtain
\begin{align}
    \frac{3}{2}\sum_{k=0}^{K-1} \alpha_k^2 \mathbb{E}[\varepsilon_{k, v}^2] &\leq \frac{3}{2}\sum_{k=0}^{K-1} \alpha_k^2 v^2\left( \mathbb{E}\|R(x_k)\|^2+ \sigma^2\right) \notag \\
    &\leq \frac{\alpha \alpha_{K-1}}{16} ( \mathbb{E}\|R(x_{K-1})\|^2+\sigma^2) + \frac{3}{2}\sum_{k=0}^{K-2} \alpha_k^2 v^2 \left( \mathbb{E} \|R(x_k)\|^2+ \sigma^2\right).\label{eq: nhy2.5}
\end{align}
We continue with the first part of the third term on the right-hand side of \eqref{eq: nhy2} and bound it using \eqref{eq: lab1}:
\begin{align}
    \sum_{k=0}^{K-1} \frac{\alpha_k^2}{2\alpha^2(k+1)}\mathbb{E}\|x_k-x^\star\|^2 &\leq \frac{1}{2}\|x_0-x^\star\|^2 + \sum_{k=1}^{K-1} \frac{\alpha_k^2}{2\alpha^2(k+1)}\mathbb{E}\|x_k-x^\star\|^2 \notag \\
    &\leq \frac{1}{2}\|x_0-x^\star\|^2 + \sum_{k=1}^{K-1} \frac{\alpha_k^2}{2\alpha^2(k+1)}\left( 2\|x_0-x^\star\|^2 + 2k\sum_{i=0}^{k-1}\alpha_i^2\mathbb{E}[\varepsilon_{i, v}^2] \right),\label{eq: nhy3}
\end{align}
where the last line identified $\varepsilon_{i, v}^2$ in view of \Cref{lem: stoc_mvi_iter_compl}.

We focus on the last term here to get
\begin{align*}
    \sum_{k=1}^{K-1} \frac{\alpha_k^2}{2\alpha^2(k+1)}\cdot 2k\sum_{i=0}^{k-1}\alpha_i^2 \mathbb{E}[\varepsilon_{i, v}^2] &= \frac{1}{\alpha^2}\sum_{i=0}^{K-2} \sum_{k=i+1}^{K-1} \frac{k}{k+1}\alpha_k^2 \alpha_i^2   \mathbb{E}[\varepsilon_{i, v}^2] \\
    &\leq \frac{1}{\alpha^2} \left( \sum_{k=0}^{K-1}\alpha_k^2 \right) \sum_{i=0}^{K-2}\alpha_i^2 \mathbb{E}[\varepsilon_{i, v}^2]\\
    &\leq  \frac{1}{\alpha^2} \left( \sum_{k=0}^{K-1}\alpha_k^2 \right) \sum_{i=0}^{K-2}\alpha_i^2v^2\left( \mathbb{E}\|R(x_i)\|^2 + \sigma^2 \right),
\end{align*}
where the last step used \Cref{cor: mlmc_bounds_avg}.

Plugging in back to \eqref{eq: nhy3} gives
\begin{align}
    \sum_{k=0}^{K-1} \frac{\alpha_k^2}{2\alpha^2(k+1)}\mathbb{E}\|x_k-x^\star\|^2 &\leq \left(\frac{1}{2} + \sum_{k=1}^{K-1} \frac{\alpha_k^2}{\alpha^2(k+1)}  \right) \|x_0-x^\star\|^2 \notag \\
    &\quad+ \left( \sum_{k=0}^{K-1}\frac{\alpha_k^2}{\alpha^2} \right) \sum_{i=0}^{K-2}\alpha_i^2v^2\left(  \mathbb{E}\|R(x_i)\|^2 + \sigma^2 \right).\label{eq: nhy4}
\end{align}
By using $\alpha_k = \frac{\alpha}{\sqrt{k+2}\log(k+3)}$,
we have
\begin{equation*}
    \sum_{k=0}^{K-1} \frac{\alpha_k^2}{\alpha^2} < 3, ~~~ \sum_{k=0}^{K-1} \frac{\alpha_k^2}{\alpha^2(k+1)} < 0.25, ~~~\alpha_k \leq \alpha ~\forall k \geq 0,
\end{equation*}
which helps estimate \eqref{eq: nhy4} as
\begin{align}
    \sum_{k=0}^{K-1} \frac{\alpha_k^2}{2\alpha^2(k+1)}\mathbb{E}\|x_k-x^\star\|^2 &\leq \frac{3}{4} \|x_0-x^\star\|^2 +3 \sum_{i=0}^{K-2}\alpha_i^2v^2\left(  \mathbb{E}\|R(x_i)\|^2 + \sigma^2 \right).\label{eq: jke4}
\end{align}
We finally estimate the second part of the third term on the right-hand side of \eqref{eq: nhy2} by using \Cref{cor: mlmc_bounds_avg}:
\begin{align*}
    \alpha^2\sum_{k=0}^{K-1} \frac{k+1}{2} \mathbb{E}[\varepsilon_{k, b}^2] &\leq \alpha^2\sum_{k=0}^{K-1} (k+1)b_k^2\mathbb{E}[\|R(x_k)\|^2 + \sigma^2].
\end{align*}
We use the setting $b_k^2=\frac{\alpha_k}{120\alpha(k+1)}$ and $b_{K-1}^2 < \frac{\alpha_{K-1}}{16\alpha K}$ to obtain
\begin{align}
    \alpha^2\sum_{k=0}^{K-1} \frac{k+1}{2} \mathbb{E}[\varepsilon_{k, b}^2] &\leq \frac{\alpha \alpha_{K-1}}{16} ( \mathbb{E} \|R(x_{K-1})\|^2 + \sigma^2)+ \alpha^2\sum_{k=0}^{K-2} (k+1)b_k^2\mathbb{E}[\|R(x_k)\|^2 + \sigma^2].\label{eq: nhy5}
\end{align}
We collect \eqref{eq: nhy2.5}, \eqref{eq: jke4}, and \eqref{eq: nhy5} in \eqref{eq: nhy2} to get
\begin{align}
    \frac{\alpha}{8} \sum_{k=0}^{K-1} \alpha_k \mathbb{E}\|R(x_k)\|^2 &\leq \frac 54 \|x_0-x^\star\|^2 + \frac{\alpha\alpha_{K-1}}{8} \sigma^2  +\frac 92\sum_{k=0}^{K-2} \alpha_k^2 v^2 \left( \mathbb{E}\|R(x_k)\|^2+ \sigma^2\right) \notag\\
    &\quad+\alpha^2\sum_{k=0}^{K-2} (k+1)b_k^2\mathbb{E}[\|R(x_k)\|^2 + \sigma^2].\label{eq: nhy6}
\end{align}
We now show by induction that
\begin{equation}
    \alpha\sum_{k=0}^{K-1} \alpha_k\mathbb{E}\|R(x_k)\|^2 \leq C\left( \|x_0-x^\star\|^2+\alpha^2\sigma^2 \right) ~~~\forall K \geq 1,\label{eq: ind_result_mlmc}
\end{equation}
for some $C$ to be determined.
With $\alpha < 1$, \eqref{eq: nhy6} becomes
\begin{align}
    \frac{\alpha}{8} \sum_{k=0}^{K-1} \alpha_k \mathbb{E}\|R(x_k)\|^2 &\leq 1.25\|x_0-x^\star\|^2 + \alpha^2 \sigma^2+4.5\sum_{k=0}^{K-2} v^2 (\alpha\alpha_k \mathbb{E} \|R(x_k)\|^2 + \alpha^2\sigma^2) \notag\\
    &\quad+\alpha\sum_{k=0}^{K-2} \frac{(k+1)b_k^2}{\alpha_k}\mathbb{E}[\alpha\alpha_k\|R(x_k)\|^2 + \alpha^2 \sigma^2].\label{eq: nhy7}
\end{align}
Let us set
\begin{equation*}
    C=32,~~~ b_k^2=\frac{\alpha_k}{120\alpha(k+1)},~~~v^2 = \frac{1}{60}
\end{equation*} 
and use the inductive assumption $\alpha\sum_{k=0}^{K-2}\alpha_k \mathbb{E}\|R(x_k)\|^2 \leq 32(\|x_0-x^\star\|^2 + \alpha^2\sigma^2)$ in \eqref{eq: nhy7} to obtain
\begin{align*}
    \frac{\alpha}{8} \sum_{k=0}^{K-1} \alpha_k \mathbb{E}\|R(x_k)\|^2 &\leq 4(\|x_0-x^\star\|^2 + \alpha^2 \sigma^2),
\end{align*}
which verifies $\alpha\sum_{k=0}^{K-1}\alpha_k \mathbb{E}\|R(x_k)\|^2 \leq 32(\|x_0-x^\star\|^2+\alpha^2\sigma^2)$.

For the base case,  we use $\alpha_0 = \frac{\alpha}{\sqrt{2}\log 3} < 1$ and $\alpha^{-1}$-star-Lipschitzness of $R=\id-J_{\eta(F+G)}$ to get $\alpha\alpha_0\|R(x_0)\|^2 \leq \|x_0-x^\star\|^2$.
This establishes the base case and completes the induction.

By using $\alpha_k\geq \alpha_K = \frac{\alpha}{\sqrt{K+2}\log(K+3)}$ in \eqref{eq: ind_result_mlmc} with $C=32$ and multiplying both sides by $\frac{1}{K\alpha_K}$ and using $\frac{\sqrt{K+2}}{K}\leq\frac{2}{\sqrt{K}}$ which is true for $K\geq 1$, we get the claimed rate result. Finally, in view of \Cref{cor: mlmc_bounds_avg}, and definitions of $b_k, v$, each iteration makes expected number of calls $O(\log^{2}(k+1))$. By using this expected cost of each iteration, we also get the final expected stochastic first-order complexity result.
\end{proof}

\section{Additional Remarks on Related Work} \label{sec: app_relwork}
There exist a line of works that attempted to construct local estimation of Lipschitz constants to offer an improved range for $\rho$ depending on the curvature \citep{pethick2022escaping,alacaoglu2023beyond}. However, these results cannot bring global improvements in the worst-case range of $\rho$ where the limit for $\rho$ is still $\frac{1}{2L}$. This is because it is easy to construct examples where the local Lipschitz constants are the same as the global Lipschitz constant.

The work \cite{hajizadeh2023linear} gets linear rate of convergence for interaction dominant problems which is shown to be closely related to cohypomonotonity, see \Cref{ex: interaction}. One important difference is that cohypomonotonicity is equivalent to $\alpha$ interaction dominance with $\alpha \geq 0$ whereas \cite{hajizadeh2023linear} requires $\alpha > 0$ for linear convergence. This is an important difference because \emph{(i)} we know that cohypomonotonicity relaxes monotonicity and \emph{(ii)} we know that even monotonicity is not sufficient for linear convergence. For monotone problems $O(\varepsilon^{-1})$ is the optimal first-order oracle complexity (see, e.g., \cite{yoon2021accelerated}) and hence it is also optimal with cohypomonotonicity.

In the literature for fixed point iterations, several works considered inexact Halpern or KM iterations without characterizing explicit first-order complexity results, see for example \cite{leucstean2021quantitative,bartz2022conical,kohlenbach2022proximal,combettes2004proximal}. In particular, \citet{bartz2022conical} used conic nonexpansiveness to analyze KM iteration. The dependence of the range of $\rho$ on $L$ arises when we start characterizing the first-order complexity. This is the reason these works have not been included in comparisons in Table \ref{table:1}.

For the stochastic cohypomonotone problems, the best complexity result to our knowledge is due to \cite{chen2022near}. This paper can obtain the optimal complexity $\widetilde{O}(\varepsilon^{-2})$ with cohypomonotone stochastic problems with a 6-loop algorithm using many carefully designed regularization techniques, extending the work of \citet{allen2018make} that focused on minimization. Some disadvantages of this approach compared to ours: \emph{(i)} the bound for cohypomonotonicity is $\rho \leq\frac{1}{2L}$; \emph{(ii)} the algorithm needs estimates of variance upper bound $\sigma^2$ and, more importantly, $\| x^0 - x^\star\|^2$; \emph{(iii)} the result is only given for unconstrained problems, which also makes it difficult to assume a bounded domain since there is no guarantee a priori for the iterates to stay bounded for an unconstrained problem. Given that the 6-loop algorithm and analysis of \cite{chen2022near} is rather complicated, it is not clear to us if their arguments generalize to constraints or if the other drawbacks can be alleviated.

The work \cite{tran2023randomized} focused on problems with $\rho$-weakly MVI solutions for $\rho < \frac{1}{8L}$ and derived $O(\varepsilon^{-2})$ for a randomized coordinate algorithm. Due to randomization, the complexity result in this work holds for the expectation of the optimality measure. Because of the coordinatewise updates, the problem focused in this work is deterministic, similar to the  setup in \Cref{sec: weak_mvi}. \citet{bravo2024stochastic_halpern} studied stochastic inexact Halpern iteration in normed spaces and obtained complexity $O(\varepsilon^{-5})$ for finding fixed-points, by using an oracle providing unbiased samples of a nonexpansive map.

\subsection{Clarifications about \Cref{table:2}}\label{sec: table_clarif}
Since the complexity results have not been written explicitly in some of the references, we provide details on how we computed the complexities that we report for the existing works. 

\cite{choudhury2023singlecall}: We use Theorem 4.5 in this corresponding paper to see that squared operator norm is upper bounded by $O(K^{-1})$. To make the operator norm smaller than $\varepsilon$, the order of $K$ is $\varepsilon^{-2}$. The batch-size has order $K$ and hence the total number of oracle calls is $O(K^2) = O(\varepsilon^{-4})$.

\cite{bohm2022two}: We use Theorem 3.3 in this corresponding paper. The paper stated that to make the squared operator norm smaller than $\varepsilon$, number of iterations is $O(\varepsilon^{-2})$ and the batch size is $O(\varepsilon^{-3})$. This gives complexity $O(\varepsilon^{-3})$ for making the \emph{squared} operator norm smaller than $\varepsilon$. Hence, to make the operator norm smaller than $\varepsilon$, the complexity is $O(\varepsilon^{-6})$.

\citep{pethick2023stable}: \emph{(i)} For ``best rate'' result, we use Corollary E.3(i) in this corresponding paper. The dominant term in the bound for the squared residual is $O(K^{-1})$. Hence to make the  norm of the residual smaller than $\varepsilon$ (equivalently, the squared norm smaller than $\varepsilon^{-2}$), one needs $K$ to be of the order $\varepsilon^{-2}$. Then, the squared variance is assumed to decrease at the order of $k^2$ which requires the batch size at iteration $k$ to be $k^2$. Then the complexity is upper bounded by $\sum_{k=1}^K \tau k^2 = \widetilde{O}(K^3) =\widetilde{O}(\varepsilon^{-6})$, \emph{(ii)} for the ``last iterate'', we use the Corollary E.3(ii) given in the paper to see that the dominant term in the bound of the squared residual is $O\left(\frac{1}{\sqrt{K}}\right)$. To make the squared residual smaller than $\varepsilon^2$, this means $K$ is of the order $\varepsilon^{-4}$. The squared variance is assumed to decrease at the rate $k^3$ which requires a batch size of $k^3$ at iteration $k$. Then, with the same calculation as before, the complexity of stochastic first-order oracles to make the residual less than $\varepsilon$ is $\sum_{k=1}^K \tau k^3 = \widetilde{O}(K^4) = \widetilde{O}(\varepsilon^{-16})$.

 \end{document}